\RequirePackage{fix-cm}
\documentclass[smallextended]{svjour3}       
\smartqed  
\usepackage{amsmath,multirow}
\usepackage{amssymb}
\usepackage{amscd}
\usepackage{amsfonts}
\usepackage{enumerate}
\usepackage{mathrsfs}
\usepackage{xypic}
\usepackage{stmaryrd}
\usepackage{graphicx}
\usepackage{fancybox}
\usepackage{color}
\usepackage{amsmath}
\usepackage{exscale}
\usepackage{enumitem,array}%
\usepackage{amssymb}
\usepackage{amsmath}
\usepackage{amssymb}
\usepackage{amscd}
\usepackage{amsfonts}
\usepackage{enumerate}
\usepackage{mathrsfs}
\usepackage{xy}
\usepackage{stmaryrd}
\usepackage{graphicx}
\usepackage{fancybox}
\usepackage{color}
\usepackage{multirow}
\usepackage{exscale}
\usepackage{enumitem,array}%
\usepackage{subfigure}
\usepackage{extarrows}
\usepackage{amsfonts}
\usepackage{xypic}
\usepackage{tikz}
\usepackage{algorithm}
\usepackage{algorithmicx}
\usepackage{algpseudocode}
\usepackage{manyfoot}
\usepackage{booktabs}
\usepackage{bm}
\usepackage{amssymb}

\usepackage{verbatim}
\usepackage{graphicx}
\usepackage{amsopn}
\usepackage{hyperref}
\hypersetup{
    colorlinks=true,
    linkcolor=blue,       
    citecolor=blue,       
    urlcolor=blue,       
}

\usepackage{epstopdf}
\usepackage{epsfig,multirow}
\usepackage{mathtools}
\usepackage{tikz}

\DeclarePairedDelimiter{\norm}{\lVert}{\rVert}

\newcommand{\fe}{\mathrm{e}}
\newcommand{\eps}{\varepsilon}
\newcommand{\bR}{\bm{R}}

\newcommand{\abs}[1]{\left\vert#1\right\vert}
\newcommand{\bx}{\bm{x}}
\newcommand{\bv}{\bm{v}}
\newcommand{\bu}{\bm{u}}
\newcommand{\bw}{\bm{w}}
\newcommand{\bz}{\bm{z}}
\newcommand{\hbz}{\hat{\bm{z}}}
\newcommand{\mf}{\bm{f}}

\newcommand{\br}{\bm{r}}

\newcommand{\mI}{\mathcal{I}}

\newcommand{\bZ}{\bm{Z}}

\newcommand{\tit}{\tilde{t}}

\newcommand{\ii}{\textmd{i}}
\newcommand{\dd}{\textmd{d}}

\begin{document}

\title{A novel class of high-order uniformly accurate exponential integrators with local linear extension for the charged-particle dynamics under strong magnetic field \thanks{This work was supported by the National Key Research and Development Program of China (Grant No. 2021YFA1003002), National Natural Science Foundation of China (Grant No. 12371403), Shaanxi Fundamental Science Research Project for Mathematics and Physics (Grant No. 25JSY046) and Fundamental Research Funds for the Central Universities (xzy022026045).}
}

\titlerunning{High-order UA exponential integrators with LLE for CPD}        

\author{Lina Wang        \and
        Bin Wang \and Beibei Zhu
}

\authorrunning{L. Wang, B. Wang, B. Zhu} 

\institute{L. Wang \at
              School of Mathematics and Statistics, Xi'an Jiaotong University, 710049 Xi'an, China \\
              \email{wanglina@stu.xjtu.edu.cn}           
           \and
           B. Wang \at
          School of Mathematics and Statistics, Xi'an Jiaotong University, 710049 Xi'an, China \\
           \email{wangbinmaths@xjtu.edu.cn}
             \and
           B. Zhu \at
           Corresponding author. School of Mathematics and Physics, University of Science and Technology Beijing, 100083 Beijing, China \\
           \email{zhubeibei@ustb.edu.cn}
}

\date{Received: date / Accepted: date}

\maketitle

\begin{abstract}
In this paper, we develop a novel class of high-order uniformly accurate exponential integrators for  charged-particle dynamics under strong magnetic field. The small parameter $0<\eps\ll 1$ induces rapid temporal oscillations, rendering traditional numerical methods prohibitively expensive due to severe step-size restrictions. To address this issue, a linearization technology that introduces auxiliary polynomial variables is employed to recast the original charged-particle dynamics as a higher-dimensional system. Classical exponential integrators are subsequently applied to this augmented formulation, which inherently carries richer structural information, thereby yielding a family of uniformly accurate exponential integrators that can reach arbitrarily high order without requiring any order conditions. For strong magnetic fields under the maximal-ordering scaling, we rigorously establish through algebraic techniques that the schemes with auxiliary polynomial variables of degree $k\geq2$ achieve an improved position error bound of $\mathcal{O}(\min\{\eps^{k+1}h,\eps h^{k+1}\})$ and a uniform velocity error bound of $\mathcal{O}(\min\{\eps^{k}h,h^{k+1}\})$, where $h$ denotes the time-step size. Under the additional condition that the magnetic field gradient vanishes at the origin, these bounds sharpen to $\mathcal{O}(\min\{\eps^{k+2},\eps h^{k+1}\})$ and $\mathcal{O}(\min\{\eps^{k+1},h^{k+1}\})$ for the position and velocity, respectively, for $k\geq1$. The same sharper bounds hold for constant strong magnetic fields. Numerical experiments confirm the advantages of the proposed methods.

\end{abstract}

\section{Introduction}

Charged particles play a crucial role in a wide range of fields, including plasma physics, astrophysics, and magnetic confinement fusion devices \cite{AKN,BS,L}. In plasma physics modeling, single charged particles are treated as part of the kinetic description of plasmas using the Particle-in-Cell (PIC) method. In magnetically confined fusion devices such as tokamak, strong external magnetic fields must be applied to confine the particles to the desired trajectories, which typically introduces additional time scales into the simulation and led to the complexity of the calculations. Consequently, the numerical computation of charged-particle dynamics (CPD) in the regime of strong magnetic fields has attracted considerable research attention.

The Vlasov equation with external and explicitly given electric-field function $E:(t,\bx)\in\mathbb{R}^{+}\times\mathbb{R}^{2}\rightarrow E(t,\bx)\in\mathbb{R}^2$ is given by (\cite{CCLMZ1,CCLMZ2})
\begin{equation}\label{equ-18-1}
\begin{aligned}
&\partial_{t}f(t,\bx,\bv)+\bv\cdot \nabla_{\bx}f(t,\bx,\bv)+\left(E(t,\bx)+\frac{b(\bx)}{\eps}J\bv\right)\cdot\nabla_{\bv}f(t,\bx,\bv)=0, \\
&f(0,\bx,\bv)=f_{0}(\bx,\bv), \quad (\bx,\bv)\in\mathbb{R}^2\times\mathbb{R}^2, 
\end{aligned}
\end{equation}
where $f:(t,\bx,\bv)\in[0,T]\times\mathbb{R}^2 \times\mathbb{R}^2\rightarrow f(t,\bx,\bv)\in\mathbb{R}$ denotes the unknown distribution function and $ f_{0}(\bx,\bv)\in\mathbb{R}$  the prescribed initial distribution. The unidirectional magnetic field $B: \bx\in\mathbb{R}^{2}\mapsto B(\bx)=(0,0,b(\bx))\in\mathbb{R}^{3}$ induces the Lorentz force $\bv\times B(\bx)$, which admits, in the two-dimensional setting, the reduced form $b(\bx)J\bv$ with $J=\begin{pmatrix} 0 & 1 \\ -1 & 0 \end{pmatrix}$. Here, $0<\eps\leq 1$ is a dimensionless parameter inversely proportional to the magnetic-field strength. Using the PIC discretization, 
\begin{equation*}
f(t,\bx,\bv)\approx\sum\limits_{k=1}^{N_{p}}\omega_{k}\delta(\bx-\bx_{k}(t))\delta(\bv-\bv_{k}(t)), \ t\geq 0, \ \bx,\bv\in\mathbb{R}^2,
\end{equation*}
we get the characteristic equation for $1\leq k\leq N_{p}$,
\begin{equation}\label{equ-18-2}
\begin{aligned}
&\dot{\bx}_{k}(t)=\bv_{k}(t), \ \dot{\bv}_{k}(t)=\frac{b(\bx_{k}(t))}{\eps}J\bv_{k}(t)+E(t,\bx_{k}(t)), \\
&\bx_{k}(0)=\bx_{k,0}, \ \bv_{k}(0)=\bv_{k,0}.
\end{aligned}
\end{equation}

Without loss of generality, we restrict our attention in this paper to the CPD of form \eqref{equ-18-2},
\begin{equation}\label{2d}
\begin{aligned}
&\dot{\bx}(t)=\bv(t),  \quad \bx(0)=\bx_{0}, \\
&\dot{\bv}(t)=\frac{b(\bx)}{\eps}J\bv(t)+E(\bx(t)), \quad \bv(0)=\bv_{0},
\end{aligned}
\end{equation}
which encompasses both the constant strong magnetic field case, for which $b(\bx)=C$ with $C$ constant, and the maximal ordering scaling case, for which $b(\bx)=b(\eps\bx)$.
The maximal ordering scaling strong magnetic field $\widetilde{B}(\bx)=\frac{B(\eps\bx)}{\eps}$ is motivated by the physical premise that the magnetic field is substantially stronger than the electric field, namely, $\frac{\norm{E(\bx)}}{\norm{\widetilde{B}(\bx)}}=\mathcal{O}(\eps)\ll1$. In addition, the magnetic field varies slowly on the same scale of order $\eps$ and satisfies $\rho\frac{\norm{\nabla_{\bx}\widetilde{B}(\bx)}}{\norm{\widetilde{B}(\bx)}}=\mathcal{O}(\eps)\ll 1$, where $\rho$ is the characteristic length of the dynamics, e.g., the gyro-radius. Compared with the standard scaling $\widetilde{B}(\bx)=\frac{B(\bx)}{\eps}$ commonly used in guiding-center theory, the maximal ordering scaling is more restrictive and thus renders the numerical treatment more challenging. A detailed discussion and rigorous analysis of this scaling are given in \cite{Pos}. 
Throughout this work, we assume that the magnetic field is uniformly about zero: there exists $c_{0}>0$ such that $b(\bx)\geq c_{0}>0$ for all $\bx\in\mathbb{R}^{2}$. Here $\bx(t),\bv(t)\in\mathbb{R}^2$ are the unknown, and $E(\bx)=(E_{1}(\bx),E_{2}(\bx))^{\intercal}=-\nabla U(\bx)$ is generated by a given scalar potential $U(\bx)$. As is well known, the energy $$H(\bx,\bv)=\frac{1}{2}\abs{\bv}^2+U(\bx)$$ of \eqref{2d} is conserved along the solution. 

Numerous numerical algorithms have been developed for the CPD. In the general magnetic field with $\eps=1$, the Boris algorithm \cite{Boris}, first introduced in 1970, remains in widespread use due to its excellent numerical performance. Subsequent investigations have increasingly emphasized structure-preserving methods, including time-symmetric algorithm \cite{HL}, symplectic or K-symplectic algorithm \cite{HZSLQ,Tao,Webb,ZQTLHX}, and energy-preserving algorithm \cite{BIZ,BMR}. For the CPD in the strong magnetic field regime, the parameter $\eps\ll 1$ induces highly oscillatory behavior in time. Conventional numerical methods require a step size smaller than the particle gyration period to accurately capture the particle trajectory, resulting in a severe computational burden. This challenge has motivated recent research to focus on CPD in the strong magnetic field and to develop new methods, including the filtered Boris algorithm \cite{HLW}, the large-stepsize integrator \cite{HLS}, some splitting methods \cite{WZ,ZLWLT}, and splitting methods \cite{LW} that incorporate relaxation techniques. 
 
For highly oscillatory stiff systems, exponential integrators have emerged over recent decades (\cite{CM,HOS,Tok,MW}) as a powerful means of overcoming step-size restrictions. By employing the variation-of-constants formula, exponential integrators handle the linear part exactly. Notable developments include exponential time differencing \cite{CM,KT} and exponential Runge-Kutta \cite{HO1}. For second-order problems, specialized schemes   have been introduced \cite{F,WWX}. Subsequently, exponential integrators have incorporated local linearization \cite{Tok,CBCOJ,KO} to improve accuracy through the Taylor expansion of nonlinear terms, leading to methods such as local linearization Runge–Kutta \cite{CBJC} and Rosenbrock-type exponential integrators \cite{HOS}. However, because the solutions of highly oscillatory systems are typically non-smooth in time, the errors of conventional exponential integrators remain dependent on $\eps$.
 
Various multiscale methods have been proposed to achieve uniform accuracy for highly oscillatory equations. The multi-revolution composition method \cite{PMMV,PCZ} and the stroboscopic averaging method \cite{CLMV,CCMM} rely on period-averaging strategies. The nested Picard iterative method guarantees uniform accuracy by repeatedly applying the constant variation formula \cite{CG,CW}, whereas the multi-scale exponential wave integrator is derived from the asymptotic expansion of the exact solution \cite{BC,FY}. The two-scale method \cite{CCLM} decouples the fast and slow scales, treating the fast scale as an independent periodic variable. Recently, for the CPD with strong magnetic fields, a two-scale method with improved accuracy exceeding the uniform accuracy is proposed in \cite{WJ} for two-dimensional CPD with a non-uniform strong magnetic fields. Utilizing PIC discretization, some two-scale methods with uniform accuracy have been proposed to solve the Vlasov equation under strong magnetic fields \cite{CCLMZ1,CCLMZ2,CLMZ1,CLMZ2}. This approach generally assumes that the eigenvalues of linearization matrix are integer multiples of a common frequency, making it inapplicable to multi-frequency problems where the frequencies lack a common period.

Very recently, a locally linear extension exponential integrator was proposed in \cite{QDF} for highly oscillatory non-autonomous first-order differential equations. At each step, a polynomial auxiliary variable is used to construct a high-dimensional extension system in which the $k$-th order Taylor expansion of the nonlinear term becomes linear. An exponential integrator with the variation-of-constants formula then solves the linear part exactly, and the result is projected back to the original space.  This dimensional lifting transforms higher-order Taylor components into linear terms that can be handled exactly, thus achieving uniform accuracy. 

Motivated by the local linear extension approach of \cite{QDF}, we extend this framework to general first-order highly oscillatory systems whose linearized matrices are diagonalizable with eigenvalues having zero real parts, and apply it to CPD under strong magnetic fields. The main contribution of this work is the CPD-specific adaptation and convergence analysis under uniformly bounded initial data. In contrast, the formulation in \cite{QDF} assumes small initial data of the form $\bu(0)=\eps^{\nu}\bu_{\mathrm{in}}$ with $\nu\geq 1$. A rigorous theoretical analysis establishes error estimates for the CPD system under constant and maximal-ordering scaling strong magnetic fields. The main contributions of this work are summarized as follows:
\begin{itemize}
 \item We extend local linear extension exponential integrators to CPD system with a zero eigenvalue in their linear dynamics, thereby broadening the spectral setting of \cite{QDF}.
 \item We establish convergence for the highly oscillatory CPD under uniformly bounded initial data, thus removing the small-initial-data constraint imposed in \cite{QDF}. The CPD-specific solution estimates provide the uniform bounds essential to the convergence analysis.

 \item Introducing $k\geq 2$ auxiliary variables yields fully explicit and computationally efficient schemes of order $k+1$ in $h$ without traditional order conditions, with arbitrarily high order attainable as $k$ increases.

 \item Under the maximal ordering scaling strong magnetic field, the schemes with $k\geq 2$ yield an $\mathcal{O}(\min\{\eps^{k+1}h,\eps h^{k+1}\})$ error in position and an $\mathcal{O}(\min\{\eps^{k}h,h^{k+1}\})$ error in velocity. Furthermore, if the magnetic-field gradient vanishes at the origin, the schemes with $k\geq 1$ yield the sharper bounds $\mathcal{O}(\min\{\eps^{k+2},\eps h^{k+1}\})$ for the position error and $\mathcal{O}(\min\{\eps^{k+1},h^{k+1}\})$ for the velocity error. These sharper bounds also hold for a constant strong magnetic field.
\end{itemize}

The remainder of this paper is organized as follows. Section \ref{sec-2} constructs a new class of high-order local linear extension exponential integrators via a reformulation of the CPD system. In Section \ref{sec-3}, convergence results for the position and velocity are rigorously established for constant and maximal-ordering scaling strong magnetic fields. Section \ref{sec-4} presents numerical experiments that confirm the effectiveness of the proposed method.

\section{The local linear extension exponential integrators}\label{sec-2}
In this section, we present the construction method for the local linear expansion system of CPD \eqref{2d} and the corresponding exponential integrators. 

As a preliminary step, we introduce the multi-index notation and the associated index sets used throughout the analysis. For each positive integer $j$, define
\begin{align*}
\mI_{4}^{[[j]]}\coloneqq \left\{\xi=(\xi_{1},\xi_{2},\xi_{3},\xi_{4})^{\intercal}:\sum\limits_{\ell=1}^{4}\xi_{\ell}=j \ \mbox{with} \ \xi_{1},\cdots,\xi_{4}\in\{0,\cdots,j\}   \right\}.
\end{align*}
Each element $\xi\in\mI_{4}^{[[j]]}$  is called a multi-index, and its length is denoted by $\abs{\xi}=j$.
Let $\mI_{4}^{[k]}:=\bigcup_{j=0}^{k}\mI_{4}^{[[j]]}$ be the set of all multi-indices of length at most $k$.
For $\bz\in\mathbb{R}^{4}$ and a given point $\hat{\bz}\in\mathbb{R}^{4}$, we associate to each multi-index $\xi \in \mI_{4}^{[k]}$ the polynomial
\begin{equation}\label{equ-20-2}
(\bz-\hat{\bz})^{\xi}\coloneqq  \prod_{\ell=1}^{4}(z_{\ell}-\hat{z}_{\ell})^{\xi_{\ell}},
\end{equation}
where $z_{\ell}$ and $\hat{z}_{\ell}$ denote the $\ell$-th components of $\bz$ and $\hat{\bz}$, respectively. By convention, $(\bz-\hat{\bz})^{\bm{0}}=1$.

Building on the foregoing preliminaries, we introduce the local linear extension variables as follows.

\begin{definition}(\cite{QDF})\label{def-2} 
$(i)$ We first define the polynomial sets of exact degree $j$ and cumulative degree up to $k$ by
\begin{equation}\label{equ-20-3}
P_{\bz}^{[[j,\hat{\bz}]]}=\left\{(\bz-\hat{\bz})^{\xi}:\xi \in\mI_{4}^{[[j]]}\right\}, \quad
P_{\bz}^{[k,\hat{\bz}]}=\bigcup_{j=0}^{k}P_{\bz}^{[[j,\hat{\bz}]]}.
\end{equation}
Their cardinalities are denoted by $D^{[[j]]}:=\abs{P_{\bz}^{[[j,\hat{\bz}]]}}$, and $D^{[k]}:=\abs{P_{\bz}^{[k,\hat{\bz}]}}$.
For each $0\leq j\leq k$, let
\[
\tau^{[[j]]}:\mI_{4}^{[[j]]}\longrightarrow \{1,\ldots,D^{[[j]]}\}
\]
be a bijection that specifies an ordering of the multi-indices of total degree $j$. We then define the $D^{[[j]]}$-dimensional vector $\bz^{[[j,\hbz]]}$ whose $i$-th component is 
$(\bz-\hat{\bz})^{(\tau^{[[j]]})^{-1}(i)}$. The specific choices of $\tau^{[[j]]}$ are as follows. For $j=0$, we set $\tau^{[[0]]}(\bm{0})=1$. For $j=1$, we set
$\tau^{[[1]]}(e_{\ell})=\ell$, where $e_{\ell}$ denotes the $\ell$-th canonical basis vector in $\mathbb{R}^{4}$.
For $j\geq2$, the bijection $\tau^{[[j]]}$ is defined by sorting the multi-indices in $\mI_{4}^{[[j]]}$ in standard lexicographic order. The $k$-th order local linear extension variable is then arranged in increasing order of length as
\begin{align}\label{deg-order}
\bz^{[k,\hat{\bz}]}=\left(\bz^{[[0,\hat{\bz}]]},(\bz^{[[1,\hat{\bz}]]})^{\intercal},\cdots,(\bz^{[[k,\hat{\bz}]]})^{\intercal}\right)^{\intercal}.
\end{align} 

$(ii)$ To extract the components $m_{1}$ through $m_{2}$ from a $D^{[k]}$-dimensional vector, we introduce the projection matrix
\begin{equation*}
\Pi_{m_1}^{m_2}=(\mathbf{0}_{m\times (m_{1}-1)}, \bm{I}_{m}, \mathbf{0}_{m\times (D^{[k]}-m_{2})}) \ \mbox{with} \ m=m_{2}-m_{1}+1.
\end{equation*}
Here, $\bm{0}_{m\times n}$ denotes the $m\times n$ zero matrix and $\bm{I}_{m}$ the identity matrix of order $m$. In particular, we have $\bz^{[[0,\hat{\bz}]]}=\Pi_{1}^{1} \bz^{[k,\hat{\bz}]}=1$ and $\bz^{[[1,\hat{\bz}]]} =\Pi_{2}^{5}\bz^{[k,\hat{\bz}]}=\bz-\hat{\bz}$.

$(iii)$ The partial differential operator induced by the multi-index $\xi\in\mI_{4}^{[k]}$ is defined as
\begin{equation*}
\frac{\partial^{\xi}}{\partial\bz^{\xi}}\coloneqq \begin{cases} \bm{I}_{D^{[k]}}, & \abs{\xi}=0, \\
\displaystyle \frac{\partial^{\abs{\xi}}} {\partial z_{1}^{\xi_{1}}\partial z_{2}^{\xi_{2}}\partial z_{3}^{\xi_{3}}\partial z_{4}^{\xi_{4}}}, & \abs{\xi}\geq 1.
\end{cases}
\end{equation*}
\end{definition}

Given the mappings $\tau^{[[j]]}$ for $0\leq j\leq k$, a global bijection $\tau^{[k]}:\mI_{4}^{[k]}\rightarrow\{1,\ldots,D^{[k]}\}$ is induced. More precisely, for any multi-index $\xi\in\mI^{[[j]]}$ with $j\geq 1$, its global index is defined by $\tau^{[k]}(\xi)=D^{[j-1]}+\tau^{[[j]]}(\xi)$.
Based on the preceding preparations, we present the  construction process of the local linear extended exponential integrator, which   proceeds in three steps: \\
\textbf{Step 1.} The CPD system is first reformulated into the compact form \eqref{LLES} through an equivalent transformation. \\
\textbf{Step 2.} The nonlinear part of \eqref{LLES} is subsequently expanded in a Taylor series with respect to all variables, yielding new variables defined as higher-order polynomials of the unknown components, for which the governing equations are derived. \\
\textbf{Step 3.} At each time step, the augmented high-dimensional system is integrated using a classical exponential integrator. The resulting solution is then projected back onto the original CPD system via a projection operator.

\subsection{Reformulation of the CPD system}
We begin by providing a reformulation of the CPD system. We mainly consider the case of maximal ordering scaling $b(\bx)=b(\eps\bx)$, and the constant strong magnetic field is merely a simplified form of the maximal order. 
Setting $b_{0}=b(\eps\bx_{0})$, system \eqref{2d} is equivalently rewritten as
\begin{equation}\label{equ-19-1}
\begin{aligned}
&\dot{\bx}(t)=\bv(t),  \quad \bx(0)=\bx_{0}, \\
&\dot{\bv}(t)=\frac{b_{0}}{\eps}J\bv(t)+\frac{b(\eps\bx(t))-b_{0}}{\eps}J\bv(t)+E(\bx(t)), \quad \bv(0)=\bv_{0}.
\end{aligned}
\end{equation}
Introducing the variable transformation $\bw=\eps\bv$, we further reformulate the system as
\begin{equation}\label{2d-theo}
\begin{aligned}
&\dot{\bx}(t)=\frac{\bw(t)}{\eps}, \quad \bx(0)=\bx_{0}, \\
&\dot{\bw}(t)=\frac{b_{0}}{\eps}J\bw(t)+\frac{b(\eps\bx(t))-b_{0}}{\eps}J\bw(t)+\eps E(\bx(t)), \quad \bw(0)=\eps\bv_{0}.
\end{aligned}
\end{equation}
With $\bx=(x_{1},x_{2})^{\intercal}$, $\bw=(w_{1},w_{2})^{\intercal}$, and $\bz=(\bx^{\intercal},\bw^{\intercal})^{\intercal}\in\mathbb{R}^{4}$, system \eqref{2d-theo} can be expressed in the compact form
\begin{equation}\label{LLES}
\dot{\bz}=\frac{1}{\eps}A_{1}\bz+\bm{f}(\bz), \ \bz(0)=\bz_{0},
\end{equation}
where 
\begin{equation}\label{equ-10-3-1}
A_{1}=\begin{pmatrix} \bm{0}_{2\times 2} & \bm{I}_{2} \\ \bm{0}_{2\times 2} & b_{0}J \end{pmatrix},
\quad
\bm{f}(\bz)=\begin{pmatrix} \bm{0}_{2\times 1} \\ F(\bz) \end{pmatrix},
\quad
\bz_{0}=\begin{pmatrix} \bx_{0} \\ \eps\bv_{0} \end{pmatrix},
\end{equation}
and
\begin{equation}\label{equ-10-7-1}
F(\bz)\coloneqq (F_{1}(\bz),F_{2}(\bz))^{\intercal}=\frac{b(\eps\bx)-b_{0}}{\eps}J\bw+\eps E(\bx).
\end{equation} 
The matrix $A_{1}$ is diagonalizable, with a zero eigenvalue of multiplicity two and two purely imaginary eigenvalues $\pm b_{0}\ii$.

The zero eigenvalue of $A_{1}$ does not generate oscillations. Indeed, for $\lambda=0$ the corresponding propagator reduces to $e^{\lambda t/\eps}=1$, so that the associated component is non-oscillatory. Thus, admitting a zero eigenvalue as a relaxation of the purely imaginary spectrum assumption in \cite{QDF} is natural and harmless, and does not
affect the uniform error estimates, as will be shown in Section~\ref{sec-3}.

\subsection{Governing equations}
In this section, we introduce auxiliary variables corresponding to the Taylor polynomial of the nonlinear term in \eqref{LLES} and derive their governing equations.

Let $\hat{\bz}$ be a prescribed reference point. The $k$-th order Taylor expansion of $\mf$ about $\hat{\bz}$ takes the form
\begin{equation}\label{equ-21-1}
\mf(\bz)=\sum\limits_{j=0}^{k}\sum\limits_{\abs{\xi}=j}\frac{1}{\xi_{1}!\cdots \xi_{4}!}\frac{\partial^{\xi}\mf(\hat{\bz})}{\partial\bz^{\xi}}
(\bz-\hbz)^{\xi}+\bar{\br}^{k+1}(\bz;\hbz).
\end{equation}
The associated $k+1$-th order Taylor remainder of $\bm{f}$ with respect to $\hbz$ is given by
\begin{equation}\label{equ-10-21-1}
\bar{\br}^{k+1}(\bz;\hat{\bz})=\begin{pmatrix}
                                 \bm{0}_{2\times 1} \\
                                 \br^{k+1}(\bz;\hat{\bz})
                                 \end{pmatrix}
\end{equation}
with
\begin{equation}\label{equ-21-3}
\begin{aligned}
\br^{k+1}(\bz;\hat{\bz})=\sum\limits_{\abs{\xi}=k+1}\frac{1}{\xi_{1}!\cdots \xi_{4}!}\frac{\partial^{\xi}F(\hat{\bz}+\theta (\bz-\hat{\bz}))}{\partial\bz^{\xi}}(\bz-\hat{\bz})^{\xi}, \ 0<\theta<1.
\end{aligned}
\end{equation}

As the variables in $P_{\bz}^{[k,\hbz]}$ are still undetermined, we proceed by derive governing equations of $\bz^{[k,\hat{\bz}]}$, thereby closing the system. Given a multi-index $\xi\in\mI_{4}^{[[j]]}$ with $\abs{\xi}=j$, we apply \eqref{LLES} and \eqref{equ-20-2} to obtain
\begingroup
\allowdisplaybreaks[4]
\begin{align}
\frac{d}{dt}(\bz-\hbz)^{\xi}=&\sum\limits_{j=1}^{4}\xi_{j}(z_{1}-\hat{z}_{1})^{\xi_{1}}\cdots (z_{j}-\hat{z}_{j})^{\xi_{j}-1}\cdots (z_{4}-\hat{z}_{4})^{\xi_{4}}\frac{d(z_{j}-\hat{z}_{j})}{dt} \notag\\
=&\sum\limits_{j=1}^{4}\xi_{j}(\bz-\hat{\bz})^{\xi-e_{j}}\left(\frac{1}{\eps}\sum\limits_{\ell=1}^{4}(A_{1})_{j\ell}z_{\ell}+f_{j}(\bz)\right) \notag\\
=&\frac{1}{\eps}\sum\limits_{j=1}^{4}\sum\limits_{\ell=1}^{4}\xi_{j}(A_{1})_{j\ell}(\bz-\hat{\bz})^{\xi-e_{j}+e_{\ell}}+\frac{1}{\eps}\sum\limits_{j=1}^{4}\sum\limits_{\ell=1}^{4}\xi_{j} (A_{1})_{j\ell}\hat{z}_{\ell} \notag\\
&(\bz-\hat{\bz})^{\xi-e_{j}}+\sum\limits_{j=1}^{4}\sum_{\ell=0}^{k-\abs{\xi}+1}\sum\limits_{\abs{\eta}=\ell}\frac{\xi_{j}}{\eta_{1}!\cdots\eta_{4}!}\frac{\partial^{\eta}f_{j}(\hbz)}{\partial\bz^{\eta}}
(\bz-\hat{\bz})^{\xi+\eta-e_{j}} \notag\\
&+\sum\limits_{j=1}^{4}\xi_{j}(\bz-\hat{\bz})^{\xi-e_{j}}\bar{r}_{j}^{k-\abs{\xi}+2}(\bz;\hat{\bz}).
\label{equ-21-4}
\end{align}
\endgroup
The last equality results from a Taylor expansion of $f_{j}$ up to order $k-\abs{\xi}+1$, in which $f_{j}$ denotes the $j$-th component of $\bm{f}$ and $\bar{r}_{j}^{k-\abs{\xi}+2}$ its Taylor remainder. 
Inserting the matrix and vector structures from \eqref{equ-10-3-1} and \eqref{equ-10-21-1} into \eqref{equ-21-4}, we derive
\begin{equation}\label{equ-10-6-2}
\begin{aligned}
&\frac{d}{dt}(\bz-\hat{\bz})^{\xi}\\
=&\frac{1}{\eps}\bigg(\xi_{1}(\bz-\hbz)^{\xi-e_{1}+e_{3}}+\xi_{2}(\bz-\hbz)^{\xi-e_{2}+e_{4}}+b_{0}\xi_{3}(\bz-\hbz)^{\xi-e_{3}+e_{4}} \\
&-b_{0}\xi_{4}(\bz-\hbz)^{\xi-e_{4}+e_{3}}\bigg)+\frac{1}{\eps}\bigg(\xi_{1}\hat{z}_{3}(\bz-\hbz)^{\xi-e_{1}}+\xi_{2}\hat{z}_{4}(\bz-\hbz)^{\xi-e_{2}}  \\
&+b_{0}\xi_{3}\hat{z}_{4}(\bz-\hbz)^{\xi-e_{3}}-b_{0}\xi_{4}\hat{z}_{3}(\bz-\hbz)^{\xi-e_{4}}\bigg)+\sum\limits_{\ell=0}^{k-\abs{\xi}+1}\sum\limits_{\abs{\eta}=\ell}\frac{1}{\eta_{1}!\cdots\eta_{4}!} \\
&\bigg(\xi_{3}\frac{\partial^{\eta}F_{1}(\hbz)}{\partial\bz^{\eta}}(\bz-\hbz)^{\xi+\eta-e_{3}}+\xi_{4}\frac{\partial^{\eta}F_{2}(\hbz)}{\partial\bz^{\eta}}(\bz-\hbz)^{\xi+\eta-e_{4}}   \bigg) \\
&+\xi_{3}(\bz-\hbz)^{\xi-e_{3}}r_{1}^{k-\abs{\xi}+2}(\bz;\hbz)+\xi_{4}(\bz-\hbz)^{\xi-e_{4}}r_{2}^{k-\abs{\xi}+2}(\bz;\hbz).
\end{aligned}
\end{equation}

Except for the final two nonlinear remainder terms, every term in \eqref{equ-10-6-2} is a linear combination of polynomials in $P_{\bz}^{[k,\hbz]}$. With the ordering prescribed by \eqref{deg-order}, we collect the coefficients associated with like polynomials. After discarding the remainder terms, \eqref{equ-10-6-2} can thus be written as a product of the corresponding coefficient vectors and $\bz^{[k,\hbz]}$. Differentiating each component of $\bz^{[k,\hbz]}$ according to \eqref{equ-10-6-2}, collecting the resulting coefficient vectors, and separating the $\mathcal{O}(1/\eps)$ contribution from the $\mathcal{O}(1)$ terms yields the coefficient matrices $A_{1}^{[k]}(\hat{\bz})$ and $A_{0}^{[k]}(\hat{\bz})$. Consequently, the governing equation for $\bz^{[k,\hbz]}$ takes the from:
\begin{equation}\label{local-exten}
\frac{d\bz^{[k,\hbz]}}{dt}=\frac{1}{\eps}A_{1}^{[k]}(\hbz)\bz^{[k,\hbz]}+A_{0}^{[k]}(\hbz)\bz^{[k,\hbz]}+\bR^{[k]}(\bz^{[k,\hbz]};\hbz),
\end{equation}
which is referred as local linear extension system of \eqref{LLES}. Let $i=\tau^{[k]}(\xi)$ denote the row of \eqref{local-exten} associated with the equation generated by multi-index $\xi$ of length $j$. It follows from \eqref{equ-10-6-2} that the corresponding component of $\bR^{[k]}(\bz^{[k,\hbz]};\hbz)$ is given by
\begin{equation}\label{equ-10-6-3}
\bR_{i}^{[k]}(\bz^{[k,\hbz]};\hbz)=\xi_{3}(\bz-\hbz)^{\xi-e_{3}}r_{1}^{k-\abs{\xi}+2}(\bz;\hbz)+\xi_{4}(\bz-\hbz)^{\xi-e_{4}}r_{2}^{k-\abs{\xi}+2}(\bz;\hbz).
\end{equation}
Consequently, 
\begin{equation}\label{equ-10-8-1}
\bR^{[k]}(\bz^{[k,\hbz]};\hbz)=\begin{pmatrix}
          0 \\ 
          \begin{pmatrix} 0 \\ 0 \\ r_{1}^{k+1}(\bz;\hbz) \\ r_{2}^{k+1}(\bz;\hbz) \end{pmatrix} \\
          \biggl(\xi_{3}(\bz-\hbz)^{\xi-e_{3}}r_{1}^{k}(\bz;\hbz)+\xi_{4}(\bz-\hbz)^{\xi-e_{4}}r_{2}^{k}(\bz;\hbz)\biggr)_{\xi\in\mI_{4}^{[[2]]}} \\
          \vdots \\
          \biggl(\xi_{3}(\bz-\hbz)^{\xi-e_{3}}r_{1}^{2}(\bz;\hbz)+\xi_{4}(\bz-\hbz)^{\xi-e_{4}}r_{2}^{2}(\bz;\hbz)\biggr)_{\xi\in\mI_{4}^{[[k]]}}
        \end{pmatrix}.
\end{equation}

Since the compact formulation does not provide explicit closed-form expressions for $A_{1}^{[k]}$ and $A_{0}^{[k]}$, we construct these matrices elementwise through the algorithmic procedure presented in Algorithm 1 of \cite{QDF}. The procedure enumerates all multi-indices in $P_{\bz}^{[k,\hbz]}$ as row indices. For each equation associated with a prescribed multi-index, the column indices are identified from the positions of the corresponding polynomials, thereby determining the local linear extension system.

\subsection{Exponential integrator}
To conclude this section, we present the construction of the exponential integrator. We introduce a uniform discretization of the time interval $[0,T]$ with step size $h=T/N$, and set $t_n = nh$ for $n=0,\dots,N$. At each exact solution point $\bz(t_{n})$, the associated local linear extension system on $[t_{n},t_{n+1}]$ is given by:
\begin{equation}\label{equ-5-15-1}
\begin{aligned}
&\frac{d\bz^{[k,\bz(t_{n})]}}{dt}=\frac{1}{\eps}A_{1}^{[k]}(\bz(t_{n}))\bz^{[k,\bz(t_{n})]}+A_{0}^{[k]}(\bz(t_{n}))\bz^{[k,\bz(t_{n})]}+\bR^{[k]}(\bz^{[k,\bz(t_{n})]}), \\
&\bz^{[k,\bz(t_{n})]}(t_{n})=e_{1}^{D^{[k]}}, \ t_{n}\leq t\leq t_{n+1}, 
\end{aligned}
\end{equation}
where $e_{1}^{D^{[k]}}$ denotes the first canonical basis vector of the $D^{[k]}$-dimensional space, and for brevity we write $\bR^{[k]}(\bz^{[k,\bz(t_{n})]})$ in place of $\bR^{[k]}(\bz^{[k,\bz(t_{n})]};\bz(t_{n}))$. Taking $\hbz=\bz(t_{n})$ in Definition~\ref{def-2} (ii) gives $$\Pi_{2}^{5}\bz^{[k,\bz(t_{n})]}(t)=\bz(t)-\bz(t_{n}).$$ Consequently, we obtain
\begin{equation*}
\bz(t_{n+1})=\Pi_{2}^{5}\bz^{[k,\bz(t_{n})]}(t_{n+1})+\bz(t_{n}), \ n=0, \cdots, N-1,
\end{equation*}
from which the exact solution of \eqref{LLES} is recovered. Omitting $\bR^{[k]}$ produces the linear truncated system
\begin{equation}\label{trunc}
\begin{aligned}
&\frac{d\tilde{\bz}^{[k,\bz(t_{n})]}}{dt}=\frac{1}{\eps}A_{1}^{[k]}(\bz(t_{n}))\tilde{\bz}^{[k,\bz(t_{n})]}+A_{0}^{[k]}(\bz(t_{n}))\tilde{\bz}^{[k,\bz(t_{n})]}, \\
&\tilde{\bz}^{[k,\bz(t_{n})]}(t_{n})=e_{1}^{D^{[k]}}, \ t_{n}\leq t \leq t_{n+1},
\end{aligned}
\end{equation}
with projection $\tilde{\bz}(t_{n+1})=\Pi_{2}^{5}\tilde{\bz}^{[k,\bz(t_{n})]}(t_{n+1})+\bz(t_{n})$.
Let $\bZ_{n}$ denote the numerical solution at $t_{n}$. Substituting $\bZ_{n}$ into the system \eqref{trunc} yields
\begin{equation}\label{LLES-1}
\begin{aligned}
&\frac{d\bZ^{[k,\bZ_{n}]}}{dt}=\frac{1}{\eps}A_{1}^{[k]}(\bZ_{n})\bZ^{[k,\bZ_{n}]}+A_{0}^{[k]}(\bZ_{n})\bZ^{[k,\bZ_{n}]}, \quad t_{n}\leq t\leq t_{n+1}, \\
&\bZ^{[k,\bZ_{n}]}(t_{n})=e_{1}^{D^{[k]}},
\end{aligned}
\end{equation}
and whose exact solution provides $\bZ^{[k,\bZ_{n}]}(t)$ as the next step
\begin{equation}\label{scheme-1}
\bZ^{[k,\bZ_{n}]}(t_{n+1})=\fe^{(\frac{1}{\eps}A_{1}^{[k]}(\bZ_{n})+A_{0}^{[k]}(\bZ_{n}))h}e_{1}^{D^{[k]}}.
\end{equation}
The projection 
\begin{equation}\label{scheme-2}
\bZ_{n+1}=\Pi_{2}^{5}\bZ^{[k,\bZ_{n}]}(t_{n+1})+\bZ_{n}
\end{equation}
recovers the numerical solution of system \eqref{LLES} at $t_{n+1}$. Finally, the numerical solution of the CPD system \eqref{2d-theo} is given by
\begin{equation}\label{scheme-3}
\bx_{n+1}=(\bm{I}_{2},\bm{0}_{2\times 2})\bZ_{n+1}, \ \bv_{n+1}=\frac{1}{\eps}(\bm{0}_{2\times 2}, \bm{I}_{2})\bZ_{n+1}.
\end{equation}
Accordingly, \eqref{scheme-1}--\eqref{scheme-3} constitute the numerical scheme for solving the CPD system.

In summary, Figure \ref{flow} illustrates the detailed flowchart of the local linear extension exponential integrator applied to the CPD system \eqref{2d}.

\usetikzlibrary{calc,positioning,shapes.geometric,arrows.meta}
\begin{figure}[t!]
\centering
\resizebox{\textwidth}{!}{%
\begin{tikzpicture}[every node/.style={font=\large}, node distance=1.8cm,
    >=Stealth,
    startstop/.style={rectangle, rounded corners=6pt, minimum width=2cm, minimum height=1cm, text centered, draw=black, fill=orange!20},
    process/.style={rectangle, rounded corners=6pt, minimum width=3cm, minimum height=1.1cm, text centered, draw=black, fill=blue!10},
    arrow/.style={->},
    dashedarrow/.style={->, dashed}]

\node (start) [startstop] {Start};

\node (in1) [process, below=of start, node distance=22mm]
  {$\dot{\bx}=\bv, \ \dot{\bv}=\dfrac{b(\eps\bx)}{\eps}J\bv+E(\bx)$};

\node (pro1) [process, below=of in1]
  {$\dot{\bz}=\dfrac{1}{\eps}A_{1}\bz+\bm{f}(\bz)$};

\node (pro2) [process, below=of pro1]
  {$\dfrac{d\bz^{[k,\hbz]}}{dt}=\dfrac{1}{\eps}A_{1}^{[k]}(\hbz)\bz^{[k,\hbz]}+A_{0}^{[k]}(\hbz)\bz^{[k,\hbz]}+\bR^{[k]}(\bz^{[k,\hbz]};\hbz)$};

\coordinate (R) at ($(pro2.east)+(6.6cm,0)$);

\node (pro3) [process] at (R |- pro2)
  {$\bZ^{[k,\bZ_{n}]}(t_{n+1}), \ n=0,1,\ldots,N-1$};

\node (pro4) [process] at (R |- pro1)
  {$\bZ_{n+1}, \ n=0,1,\ldots,N-1$};

\node (pro5) [process] at (R |- in1)
  {$\bx_{n+1},\ \bv_{n+1}, \ n=0,1,\ldots,N-1$};

\node (end) [startstop] at (R |- start) {End};

\draw[arrow] (start) -- (in1);
\draw[arrow] (in1)-- node[left, align=left] {$\bm{w}=\eps\bv$}
     node[right] {$\bz=(\bx^{\intercal},\bw^{\intercal})^{\intercal}$}(pro1);

\draw[arrow] (pro1) -- node[left, align=left] {Taylor expansion\\$\bm{f}(\bz)$ about $\hbz$}
     node[right, align=left] {Differentiate each\\component of $\bz^{[k,\hbz]}$}(pro2);
\draw[arrow] (pro2)-- node[above] {solved by \eqref{scheme-1}} (pro3);
\draw[arrow] (pro3)-- node[left, align=left] {project} (pro4);
\draw[arrow] (pro4)-- node[left, align=left] {$\bv_{n+1}=\dfrac{1}{\eps}\bw_{n+1}$} (pro5);
\draw[arrow] (pro5) -- (end);

\draw[dashedarrow] (pro5)-- node[above] {numerical solution}(in1);
\draw[dashedarrow] (pro4)-- node[above] {numerical solution}(pro1);

\draw[dashed]($(in1.west)+(-2.2,1)$)-| ($(pro2.east)+(0.1,-0.7)$)-| ($(in1.west)+(-2.2,1)$);
\node[above] at ($(in1.west)+(-1.1,1)$) {formulations};
\draw[dashed]($(pro3.west)+(-0.1,-0.7)$)-| ($(pro5.east)+(0.26,1)$)-| ($(pro3.west)+(-0.1,-0.7)$);
\node[above right] at ($(pro5.east)+(-1.4,1)$) {solutions};
\end{tikzpicture}
}
\caption{Flowchart of the local linear exponential integration process.}
\label{flow}
\end{figure}

\section{Convergence results}\label{sec-3}
In this section, we establish the convergence of the proposed exponential integrators. Since the linear component of \eqref{LLES-1}, governed by $A_{1}^{[k]}(\hbz)$, plays a central role in determining the numerical behavior of these schemes, we first investigate the algebraic structure and spectral properties of $A_{1}^{[k]}$ as derived from \eqref{scheme-1}. The technical proofs are omitted for brevity, and a complete algebraic treatment can be found in \cite{QDF}.

\begin{lemma}[\cite{QDF}]\label{lem-1}
The matrix $A_1^{[k]}$ has the following structure and properties:
\begin{itemize}
\item[(i)] There exists an invertible lower-triangular matrix $S^{[k]}(\hbz)$ such that $S^{[k]}(\hbz)$ and its inverse admit the block decompositions
\begin{equation}\label{equ-10-10-2}
\begin{aligned}
& S^{[k]}(\hbz)=\begin{pmatrix}
S_{0,0}(\hbz) &        &        &        \\
S_{1,0}(\hbz) & S_{1,1}(\hbz) &        &        \\
\vdots        & \vdots        & \ddots &        \\
S_{k,0}(\hbz) & S_{k,1}(\hbz) & \cdots & S_{k,k}(\hbz)
\end{pmatrix},
\\
&\bigl(S^{[k]}(\hbz)\bigr)^{-1}=\begin{pmatrix}
\bar{S}_{0,0}(\hbz) &              &              &              \\
\bar{S}_{1,0}(\hbz) & \bar{S}_{1,1}(\hbz) &              &              \\
\vdots             & \vdots             & \ddots       &              \\
\bar{S}_{k,0}(\hbz) & \bar{S}_{k,1}(\hbz) & \cdots       & \bar{S}_{k,k}(\hbz)
\end{pmatrix}.
\end{aligned}
\end{equation}
For $0\leq i\leq j\leq k$, the subblocks $S_{j,i}(\hbz)$ and $\bar{S}_{j,i}(\hbz)$ are $D^{[[j]]}\times D^{[[i]]}$ matrices whose entries are
homogeneous polynomials of degree $j-i$ in $\hbz$. In particular,
\begin{equation}\label{equ-10-6-3}
S_{j,j}(\hbz)=\bar{S}_{j,j}(\hbz)=I_{D^{[[j]]}}, \quad S_{1,0}(\hbz)=-\hbz, \quad \bar{S}_{1,0}(\hbz)=\hbz.
\end{equation}
\item[(ii)] The matrix $A_{1}^{[k]}(\hbz)$ has the block bidiagonal structure 
\begin{equation}\label{equ-10-10-1}
A_{1}^{[k]}(\hbz)=\begin{pmatrix}
                   0 & & &   \\
                   B_{1}^{[[1]]}(\hbz) & A_{1} &  & \\
                      & \ddots & \ddots & \\
                    & & B_{1}^{[[k]]}(\hbz) & A_{1}^{[[k]]} 
                  \end{pmatrix},
\end{equation}
where $\{A_{1}^{[[j]]}\}_{j=2}^{k}$ and $\{B_{1}^{[[j]]}(\hbz)\}_{j=1}^{k}$ are matrices of size $D^{[[j]]}\times D^{[[j]]}$ and $D^{[[j]]}\times D^{[[j-1]]}$, respectively.
Furthermore, the similarity transformation induced by $S^{[k]}(\hbz)$ carries $A_{1}^{[k]}(\hbz)$ into $A_{1}^{[k]}(\bm{0})$, which exhibits a block-diagonal structure of the form
\begin{equation}\label{equ-22-2}
A_{1}^{[k]}(\mathbf{0})=\bigl(S^{[k]}(\hbz)\bigr)^{-1}A_{1}^{[k]}(\hbz)S^{[k]}(\hbz)=\operatorname{diag}\{0,A_{1},\cdots,A_{1}^{[[k]]}\}.
\end{equation}
\end{itemize}
\end{lemma}

Building on Lemma \ref{lem-1}, we establish the following result, from which the boundedness of the matrix exponential follows.

\begin{lemma}[\cite{QDF}]\label{lem-2}
Let $A_{1}$ denote the matrix defined in \eqref{LLES}, and let $A_{1}^{[k]}(\hbz)$ be generated from the local linear extension system \eqref{local-exten}. Then $A_{1}^{[k]}(\hbz)$ is diagonalizable, and all of its eigenvalues have zero real parts.
\end{lemma}

\begin{lemma}[\cite{QDF}]\label{lem-5}
Let $\hbz$ denote a reference point in a bounded domain independent of $\eps$, and let $A_{1}^{[k]}(\hbz)$ be the matrix generated by the local linear extension system \eqref{local-exten}. Then
\begin{equation*}
\left\Vert\fe^{\frac{1}{\eps}A_{1}^{[k]}(\hbz)t}\right\Vert\leq C
\end{equation*}
holds uniformly with respect to $\eps$, $t$, and $\hbz$.
\end{lemma}

Before proceeding to the convergence analysis, we first characterize the solution properties of the CPD system \eqref{2d-theo} under maximal ordering scaling. 
To this end, the following assumption on the electric field $E(\bx)$ and magnetic field $b(\bx)$ is introduced:  
\begin{flalign*}
&\textbf{Assumption (A)}\ 
E(\bx)\text{ and }b(\bx)\text{ have Lipschitz continuous derivatives up} \\
&\text{ to order }k, \text{with Lipschitz constants independent of } \eps. &&
\end{flalign*} 
Specifically, for any $\bx_{1}, \bx_{2}\in\mathbb{R}^2$,
\begin{align}\label{equ-5-7-1}
\norm{E(\bx_{1})-E(\bx_{2})}\leq C_{E}\norm{\bx_{1}-\bx_{2}}, \quad \norm{b(\bx_{1})-b(\bx_{2})}\leq C_{b}\norm{\bx_{1}-\bx_{2}}, 
\end{align}
with constants $C_{E}, C_{b}>0$ independent of $\eps$. Throughout, $\norm{\cdot}$ denotes the Euclidean norm for vectors and the Frobenius norm for matrices. The symbol $C$ denotes a generic positive constant independent of $\eps$, whose value may vary from one line to line.

\begin{lemma}\label{lem-3}
Suppose that electric field $E(\bx)$ and magnetic field $b(\bx)$ satisfy Assumption (A) and that the initial data of the CPD system \eqref{2d} are uniformly bounded. Then the following assertions hold.
\begin{itemize}
\item[(i)] The solution of system \eqref{2d-theo} and its first time derivatives satisfy
\begin{align}
&\norm{\bx(t)}\leq C, \quad \norm{\bw(t)}\leq C\eps, \quad \norm{\bx(t)-\bx_{0}}\leq C\eps, \label{solu-bound-1}\\
&\norm{\dot{\bx}(t)}\leq C, \quad \norm{\dot{\bw}(t)}\leq C, \quad 0\leq t\leq T, \label{bound-2}
\end{align}
for some constants $C>0$ independent of $\eps$. Moreover, rewriting system \eqref{2d-theo} in the form \eqref{LLES} gives 
\[
\norm{F(\bz)}\leq C\eps, \quad \norm{\dot{\bz}}\leq C.
\]
\item[(ii)] For the $k$-th order derivatives of $F(\hbz)$ in \eqref{equ-21-3}, we further establish the bounds
\begin{equation}\label{equ-5-19-1}
\sum\limits_{\xi\in\mI_{4}^{[[k]]}}\left\Vert\frac{1}{\xi_{1}!\cdots\xi_{4}!}\frac{\partial^{k}F(\hbz)}{\partial z_{1}^{\xi_{1}}\cdots \partial z_{4}^{\xi_{4}}}\right\Vert\leq 
\begin{cases}
C, & k=2,\\
C\eps, & k\neq 2.
\end{cases}
\end{equation}
for $k\in \mathbb{N}^{+}$.
\item[(iii)] If the magnetic field $b(\eps\bx)$ satisfies $\nabla_{\alpha} b(\alpha)|_{\alpha=\bm{0}}=\bm{0}$, then the $k$-th order derivatives of $F(\hbz)$ in \eqref{equ-21-3} admit the improved estimate
    \begin{equation}\label{equ-10-18-1}
    \sum\limits_{\xi\in\mI_{4}^{[[k]]}}\left\Vert\frac{1}{\xi_{1}!\cdots\xi_{4}!}\frac{\partial^{k}F(\hbz)}{\partial z_{1}^{\xi_{1}}\cdots \partial z_{4}^{\xi_{4}}}\right\Vert\leq C\eps, \ k\geq 1.
    \end{equation} 
\end{itemize}
\end{lemma}

\begin{proof}
\textbf{(i)} Introducing the change of variables $\tilde{\bw}(t)=\fe^{-b_{0}Jt/\eps}\bw(t)$, system \eqref{2d-theo} can be recast as
\begin{subequations}\label{equ-22-8}
\begin{align}
&\dot{\bx}(t)=\frac{1}{\eps}\fe^{b_{0}Jt/\eps}\tilde{\bw}(t), \quad \bx(0)=\bx_0, \label{equ-22-8a} \\
&\dot{\tilde{\bw}}(t)=\frac{b(\eps\bx)-b_{0}}{\eps}J\tilde{\bw}(t)+\fe^{-b_{0}Jt/\eps}\eps E(\bx(t)), \quad \tilde{\bw}(0)=\eps \bv_{0}. \label{equ-22-8b}
\end{align}
\end{subequations}
Taking the inner product on both sides of \eqref{equ-22-8a} and \eqref{equ-22-8b} with $\bx(t)$ and $\tilde{\bw}(t)$, respectively, and using $J^{\intercal}=-J$ together with the Cauchy–Schwarz inequality yields
\begin{equation*}
\frac{d}{dt}\norm{\bx(t)}^2\leq \frac{2}{\eps}\norm{\bx(t)}\norm{\tilde{\bw}(t)}, \quad \frac{d}{dt}\norm{\tilde{\bw}(t)}^{2}\leq 2\eps\norm{E(\bx(t))}\norm{\tilde{\bw}(t)},
\end{equation*}
we thus conclude that
\begin{equation*}
\begin{aligned}
\frac{d}{dt}\norm{\bx(t)}&\leq \frac{1}{\eps}\norm{\tilde{\bw}(t)}, \\
\frac{1}{\eps}\frac{d}{dt}\norm{\tilde{\bw}(t)}&\leq \norm{E(\bx(t))}\leq \norm{E(\bx_{0})}+C_{E}\norm{\bx(t)-\bx_{0}} \\
&\leq \norm{E(\bx_{0})}+C_{E}\norm{\bx(t)}+C_{E}\norm{\bx_{0}},
\end{aligned}
\end{equation*}
where the estimate \eqref{equ-5-7-1} has been used.
Combining the two differential inequalities above and applying Gronwall's inequality, we obtain
\begin{align*}
\norm{\bx(t)}+\frac{1}{\eps}\norm{\tilde{\bw}(t)}\leq C_{1}, \quad 0\leq t\leq T,
\end{align*}
where $C_{1}=(\norm{\bx_{0}}+\norm{\bv_{0}}+(\norm{E(\bx_0)}+C_{E}\norm{\bx_{0}})T)\fe^{T\text{max}\{1,C_{E}\}}$. Consequently,
\[
\norm{\tilde{\bw}(t)}=\norm{\bw(t)}\leq C_{1}\eps.
\]
Next, integrating \eqref{equ-22-8a} and performing integration by parts yields
\begin{equation*}
\begin{aligned}
\bx(t)=&\bx_{0}+\int_{0}^{t}\frac{\fe^{b_{0}J\theta /\eps}}{\eps}\tilde{\bw}(\theta)d\theta \\
=&\bx_{0}-\frac{\eps J}{b_{0}}\left( \frac{\tilde{\bw}(t)}{\eps}\fe^{b_{0}Jt/\eps}-\frac{\tilde{\bw}(0)}{\eps}\right)+\frac{\eps J}{b_{0}}\int_{0}^{t}\frac{\fe^{b_{0}J\theta/\eps}}{\eps}\dot{\tilde{\bw}}(\theta)d\theta.
\end{aligned}
\end{equation*}
Substituting \eqref{equ-22-8b} into the last integral then gives
\begin{equation*}
\begin{aligned}
\bx(t)=&\bx_{0}-\frac{\eps J}{b_{0}}\left( \frac{\tilde{\bw}(t)}{\eps}\fe^{b_{0}Jt/\eps}-\frac{\tilde{\bw}(0)}{\eps}\right) \\
&+\frac{\eps J}{b_{0}}\int_{0}^{t}\left(\fe^{b_{0}J\theta/\eps}\frac{b(\eps\bx(\theta))-b_{0}}{\eps^{2}}J\tilde{\bw}(\theta)+E(\bx(\theta))\right)d\theta.
\end{aligned}
\end{equation*}
From this, we deduce that for any $0\leq t\leq T$,
\begin{equation*}
\frac{1}{\eps}\norm{\bx(t)-\bx_{0}}\leq \frac{1}{b_{0}}(C_{1}+\norm{\bv_{0}})+\frac{T\norm{E}_{\infty}}{b_{0}}+\frac{\eps C_{1}C_{b}}{b_{0}}\int_{0}^{t}\frac{1}{\eps}\norm{\bx(\theta)-\bx_{0}}d\theta,
\end{equation*}
where $\norm{E}_{\infty}=\sup\{E(\bx):\norm{\bx}\leq C_{1}\}$. An application of Gronwall's inequality then gives
\begin{equation*}
\frac{1}{\eps}\norm{\bx(t)-\bx_{0}}\leq C, \quad \forall \ 0\leq t\leq T.
\end{equation*}
Furthermore, expressing \eqref{2d-theo} in the form \eqref{LLES}, we obtain
\begin{equation*}
\norm{F(\bz)}=\left\lVert\frac{b(\eps\bx)-b_{0}}{\eps}J\bw+\eps E(\bx)\right\rVert\leq C_{b}\norm{\bx-\bx_{0}} \norm{\bw}+C \eps\leq C\eps^{2}+C\eps\leq C\eps,
\end{equation*}
and 
\begin{equation*}
\begin{aligned}
&\norm{\dot{\bx}(t)}=\frac{\norm{\bw(t)}}{\eps}\leq C, \\
&\norm{\dot{\bw}(t)}\leq \frac{b_{0}}{\eps}\norm{\bw(t)}+\left\lVert\frac{b(\eps\bx)-b_{0}}{\eps}J \bw+\eps E(\bx)\right\rVert\leq C(1+\eps)\leq C.
\end{aligned}
\end{equation*}
It follows that $\norm{\dot{\bz}}=\sqrt{\norm{\dot{\bx}}^{2}+\norm{\dot{\bw}}^{2}}\leq C$.

\textbf{(ii)} We next derive the estimates for derivatives of arbitrary order of $F(\bz)$. In view of
\eqref{equ-10-7-1}, we write
\begin{equation*}
F(\bz)\coloneqq \begin{pmatrix} F_{1}(\bz)\\ F_{2}(\bz) \end{pmatrix}=
\begin{pmatrix}
\dfrac{b(\eps\bx)-b_{0}}{\eps}w_{2}+\eps E_{1}(\bx) \\
-\dfrac{b(\eps\bx)-b_{0}}{\eps}w_{1}+\eps E_{2}(\bx)
\end{pmatrix}.
\end{equation*}
For $\xi\in\mI_{4}^{[[k]]}$, set $p=(\xi_{1},\xi_{2})$ and $q=(\xi_{3},\xi_{4})$, so that
\[
\abs{\xi}=\abs{p}+\abs{q}, \qquad
\partial_{\bz}^{\xi}=\partial_{\bx}^{p}\partial_{\bw}^{q}.
\]
Let $\alpha_i=\eps x_i$ for $i=1,2$. For $k\geq 1$, we distinguish the following three cases.

\textit{Case 1.} Suppose that $\abs{q}=0$. Then $\abs{p}=k$, and the chain rule yields
\begin{equation*}
\begin{aligned}
\partial_{\bx}^{p}F_{1}&=\eps^{k-1}\partial_{\alpha}^{p}b(\eps\bx)w_{2}+\eps\partial_{\bx}^{p}E_{1}(\bx),\\
\partial_{\bx}^{p}F_{2}&=-\eps^{k-1}\partial_{\alpha}^{p}b(\eps\bx)w_{1}+\eps\partial_{\bx}^{p}E_{2}(\bx).
\end{aligned}
\end{equation*}
Consequently,
\begin{equation}\label{deri-1}
\norm{\partial_{\bx}^{p}F_{l}}\leq C\eps, \qquad k\geq 1, \quad l=1,2.
\end{equation}

\textit{Case 2.} Suppose that $\abs{q}=1$. Then $\abs{p}=k-1$, and
\begin{equation*}
\begin{aligned}
\partial_{\bx}^{p}\partial_{w_{2}}F_{1}&=\begin{cases}
\dfrac{b(\eps\bx)-b_{0}}{\eps}, & \mbox{if } k=1, \\ 
\eps^{k-2}\partial_{\alpha}^{p}b(\eps\bx), & \mbox{if } k\geq 2,
\end{cases} \\
\partial_{\bx}^{p}\partial_{w_{1}}F_{2}&=\begin{cases}
-\dfrac{b(\eps\bx)-b_{0}}{\eps}, & \mbox{if } k=1, \\
-\eps^{k-2}\partial_{\alpha}^{p}b(\eps\bx), & \mbox{if } k\geq 2,
\end{cases}\\
\partial_{\bx}^{p}\partial_{w_{1}}F_{1}&=\partial_{\bx}^{p}\partial_{w_{2}}F_{2}=0.
\end{aligned}
\end{equation*}
It follows that
\begin{equation}\label{deri-2}
\begin{aligned}
\norm{\partial_{\bx}^{p}\partial_{w_{2}}F_{1}}&=\norm{\partial_{\bx}^{p}\partial_{w_{1}}F_{2}}\leq
\begin{cases}
C\eps, & \mbox{if } k=1,\\
C\eps^{k-2}, & \mbox{if } k\geq 2,
\end{cases}\\
\norm{\partial_{\bx}^{p}\partial_{w_{1}}F_{1}}&=\norm{\partial_{\bx}^{p}\partial_{w_{2}}F_{2}}=0.
\end{aligned}
\end{equation}

\textit{Case 3.} For $\abs{q}\geq 2$, each component $F_{l} \ (l=1,2)$ is at most linear in $\bw$, which consequently implies that
\begin{equation}\label{equ-5-19-2}
\partial_{\bz}^{\xi}F_{1}=\partial_{\bz}^{\xi}F_{2}=0, \qquad k\geq 2.
\end{equation}
Combining \eqref{deri-1}, \eqref{deri-2}, and \eqref{equ-5-19-2} then establishes \eqref{equ-5-19-1}. 

\textbf{(iii)} For the bound in (ii) with $k=2$, the estimate is partially limited by \textit{Case 2}, which leads to the absence of an $\eps$. We therefore derive a refined estimate for \textit{Case 2}. When $k=2$, and $\abs{q}=1$, $\abs{p}=1$, we have
\[
\partial_{\bx}\partial_{w_{2}}F_{1}=-\partial_{\bx}\partial_{w_{1}}F_{2}=\nabla_{\alpha} b(\alpha)|_{\alpha=\eps\bx}.
\]
Under the assumption $\nabla_{\alpha}b(\alpha)|_{\alpha=\bm{0}}=\bm{0}$, we obtain
\[
\nabla_{\alpha}b(\alpha)|_{\alpha=\eps\bx}=\nabla_{\alpha}b(\bm{0})+\int_{0}^{1}\partial^{2}_{\alpha} b(\theta\eps\bx)\eps\bx \mathrm{d}\theta
=\eps\int_{0}^{1}\partial^{2}_{\alpha} b(\theta\eps\bx)\bx \mathrm{d}\theta.
\]
The Lipschitz continuity of $\partial^{2}_{\alpha}b$ and the boundedness of $\bx(t)$ then imply that
\begin{equation*}
\norm{\partial_{\bx}\partial_{w_{2}}F_{1}}=\norm{\partial_{\bx}\partial_{w_{1}}F_{2}}=\norm{\nabla_{\alpha}b(\alpha)|_{\alpha=\eps\bx}}\leq C\eps.
\end{equation*}
This sharper bound yields estimate \eqref{equ-10-18-1}. 
\hfill$\square$
\end{proof}

Lemma \ref{lem-3} provides the uniform bounds essential for the local linear extension analysis with uniformly bounded initial data, contrasting with the small-initial-data scaling in \cite{QDF}. Leveraging the solution properties established in Lemma \ref{lem-3}, we derive the following refined result.

\begin{lemma}\label{lem-4}
Assume that the electric field $E(\bx)$ and the magnetic field $b(\bx)$ have Lipschitz continuous first-order derivatives, and that the initial data of the CPD system \eqref{2d} are uniformly bounded. Then, for any $0\leq \tilde{t}<t\leq T$, the solution of the CPD system \eqref{2d-theo} satisfies
\begin{equation*}
\norm{\bx(t)-\bx(\tilde{t})}+\norm{\bw(t)-\bw(\tilde{t})}\leq C\min\{\eps,t-\tilde{t}\}.
\end{equation*}
Consequently, when the system is written in the form \eqref{LLES}, we have
\begin{equation}\label{equ-10-7-2}
\norm{\bz(t)-\bz(\tilde{t})}\leq C\min\{\eps,t-\tilde{t}\}.
\end{equation}
\end{lemma}

\begin{proof}
Under the assumptions of this lemma, the estimates in part~\textnormal{(i)} of Lemma~\ref{lem-3} are available. Using \eqref{bound-2}, we obtain
\begin{equation}\label{equ-10-7-3}
\begin{aligned}
&\norm{\bx(t)-\bx(\tit)}\leq \int_{\tit}^{t}\norm{\dot{\bx}(s)}ds\leq C(t-\tit), \\
&\norm{\bw(t)-\bw(\tit)}\leq \int_{\tit}^{t}\norm{\dot{\bw}(s)}ds\leq C(t-\tit).
\end{aligned}
\end{equation}
On the other hand,  \eqref{solu-bound-1} gives $\norm{\bx(t)-\bx_0}\leq C\eps$ and $\norm{\bw(t)}\leq C\eps$, which imply
\begin{equation}\label{equ-10-7-4}
\begin{aligned}
&\norm{\bx(t)-\bx(\tit)}\leq \norm{\bx(t)-\bx_{0}}+\norm{\bx(\tit)-\bx_{0}}\leq C\eps, \\
&\norm{\bw(t)-\bw(\tit)}\leq \norm{\bw(t)}+\norm{\bw(\tit)}\leq C\eps.
\end{aligned}
\end{equation}
Combining \eqref{equ-10-7-3} and \eqref{equ-10-7-4} yields
\[
\norm{\bx(t)-\bx(\tit)}\leq C\min\{\eps,t-\tit\}, \quad \norm{\bw(t)-\bw(\tit)}\leq C\min\{\eps,t-\tit\},
\]
and hence
\[
\norm{\bx(t)-\bx(\tit)}+\norm{\bw(t)-\bw(\tit)}\leq C\min\{\eps,t-\tit\}.
\]
Since $\bz=(\bx^{\intercal},\bw^{\intercal})^{\intercal}$, it follows that
\[
\norm{\bz(t)-\bz(\tit)}=\sqrt{\norm{\bx(t)-\bx(\tit)}^{2}+\norm{\bw(t)-\bw(\tit)}^{2}}\leq C\min\{\eps,t-\tit\}.
\]
This completes the proof.
\hfill$\square$
\end{proof}

\begin{remark}\label{rem-1}
For a constant strong magnetic field, the conclusion of Lemma~\ref{lem-3}(i) and the result of Lemma~\ref{lem-4} remain valid. Since $F(\bx)=\eps E(\bx(t))$, its derivatives are determined by the corresponding derivatives of $E$. Consequently, the only modification to Lemma~\ref{lem-3} concerns the $k$-th order derivative bounds for $F$ in part (ii), which now read
\begin{equation}\label{equ-10-13-1}
\sum\limits_{\xi\in\mI_{4}^{[[k]]}}\left\Vert\frac{1}{\xi_{1}!\cdots\xi_{4}!}\frac{\partial^{k}F(\hbz)}{\partial z_{1}^{\xi_{1}}\cdots\partial z_{4}^{\xi_{4}}}
\right\Vert\leq C\eps, \qquad k\geq 1.
\end{equation}
\end{remark}

Let $\bz(t_{n})-\bZ_{n}$ denote the error between the exact solution $\bz(t_{n})$ of \eqref{equ-5-15-1} and the approximation $\bZ_{n}$ from \eqref{LLES-1}.
In what follows, the convergence result of the exponential integrators is established.

\begin{theorem}\label{2d-conv}
Suppose that the electric field $E(\bx)$ and magnetic field $b(\bx)$ satisfy Assumption (A) and that the initial data of the CPD system \eqref{2d} are uniformly bounded. Let the numerical solution $\bx_{n}, \bv_{n}$ are obtained from \eqref{scheme-1}–\eqref{scheme-3}, then under the maximal ordering scaling strong magnetic field, for any $0<h<h_{0}$, we have 
\begin{equation}\label{equ-4-27-1}
\norm{\bx_{n}-\bx(t_{n})}+\eps\norm{\bv_{n}-\bv(t_n)}\leq C\min\{\eps^{k+1}h,\eps h^{k+1}\},  \ k\geq 2,
\end{equation} 
where the constants $C, h_{0}$ are independent of $\eps, h$.
\end{theorem}

\begin{proof}
\textbf{Local error.} 
We first bound the local error using the truncated system \eqref{trunc}. Applying the variation-of-constants formula to \eqref{equ-5-15-1} and \eqref{trunc} yields
\begin{equation}\label{equ-4-30-5}
\begin{aligned}
&\bz^{[k,\bz(t_{n})]}(t)\\
=&\fe^{\frac{1}{\eps}A_{1}^{[k]}(\bz(t_{n}))(t-t_{n})}e_{1}^{D^{[k]}}+\int_{t_{n}}^{t}\fe^{\frac{1}{\eps}A_{1}^{[k]}(\bz(t_{n}))(t-\theta)}A_{0}^{[k]}(\bz(t_{n}))\bz^{[k,\bz(t_{n})]}(\theta)\dd\theta \\
&+\int_{t_{n}}^{t}\fe^{\frac{1}{\eps}A_{1}^{[k]}(\bz(t_{n}))(t-\theta)}\bR^{[k]}(\bz^{[k,\bz(t_{n})]}(\theta))\dd\theta,
\end{aligned}
\end{equation}
and 
\begin{equation}\label{equ-4-30-6}
\begin{aligned}
&\tilde{\bz}^{[k,\bz(t_{n})]}(t)\\
=&\fe^{\frac{1}{\eps}A_{1}^{[k]}(\bz(t_{n}))(t-t_{n})}e_{1}^{D^{[k]}}+\int_{t_{n}}^{t}\fe^{\frac{1}{\eps}A_{1}^{[k]}(\bz(t_{n}))(t-\theta)}
A_{0}^{[k]}(\bz(t_{n}))\tilde{\bz}^{[k,\bz(t_{n})]}(\theta)\dd\theta.
\end{aligned}
\end{equation}
Subtracting \eqref{equ-4-30-6} from \eqref{equ-4-30-5} gives
\begin{equation}\label{equ-4-30-7}
\begin{aligned}
\Xi_{1}^{[k]}(t):=&\bz^{[k,\bz(t_{n})]}(t)-\tilde{\bz}^{[k,\bz(t_{n})]}(t) \\
=&\int_{t_{n}}^{t}\fe^{\frac{1}{\eps}A_{1}^{[k]}(\bz(t_{n}))(t-\theta)}A_{0}^{[k]}(\bz(t_{n}))(\bz^{[k,\bz(t_{n})]}(\theta)-\tilde{\bz}^{[k,\bz(t_{n})]}(\theta))\dd\theta \\
&+\int_{t_{n}}^{t}\fe^{\frac{1}{\eps}A_{1}^{[k]}(\bz(t_{n}))(t-\theta)}\bR^{[k]}(\bz^{[k,\bz(t_{n})]}(\theta))\dd\theta.
\end{aligned}
\end{equation}
Combining \eqref{equ-21-3} with Lemma~\ref{lem-4} and the derivative estimates in \eqref{equ-5-19-1}, we deduce the following bound for $\br^{j} (j=2,3,\ldots,k+1)$
\begin{equation}\label{equ-10-11-1}
\norm{\br^{j}(\bz(\theta);\bz(t_{n}))}\leq \begin{cases}
                                       C\min\{\eps^{j},h^{j}\} , & \mbox{if } j=2, \\
                                       C\eps\min\{\eps^{j},h^{j}\}, & \mbox{if } j\geq 3.
                                      \end{cases}
\end{equation}
From the expression for $\bR^{[k]}$ given in \eqref{equ-10-8-1} and using \eqref{equ-10-11-1} and \eqref{equ-10-7-2}, we have the following estimates for the different blocks of $\bR^{[k]}$
\begin{itemize}
  \item[(1)]  for $1\leq j=\abs{\xi} \leq k-1$, $k-j+2\geq 3$, 
  \begin{equation}\label{equ-10-11-2}
  \norm{\bR^{[[j]]}(\bz^{[k,\bz(t_{n})]}(\theta))}\leq C\eps\min\{\eps^{k+1},h^{k+1}\};
  \end{equation}
  \item[(2)] for the block of highest order, i.e., $j=k$, 
  \begin{equation}\label{equ-10-11-3}
  \norm{\bR^{[[k]]}(\bz^{[k,\bz(t_{n})]}(\theta))}\leq C\min\{\eps^{k+1},h^{k+1}\}.
  \end{equation}
\end{itemize}

Applying Definition~\ref{def-2}(ii) with $m_{1}=D^{[k-1]}+1$ and $m_{2}=D^{[k]}$, and using $D^{[[k]]}=D^{[k]}-D^{[k-1]}$, we obtain the extraction matrix for the degree-$k$ block:
\begin{equation*}
P_k\coloneqq \Pi_{D^{[k-1]}+1}^{D^{[k]}}=\left(\bm{0}_{D^{[[k]]}\times D^{[k-1]}}, \bm{I}_{D^{[[k]]}}\right).
\end{equation*}
Let $Q_{k}$ denote the operator that retains only the degree-$k$ block and sets all lower-degree blocks to zero, that is
\begin{equation*}
Q_{k}\coloneqq P_{k}^{\intercal}P_{k}=\operatorname{diag} \{\bm{0}_{D^{[k-1]}\times D^{[k-1]}}, \bm{I}_{D^{[[k]]}}\}.
\end{equation*}
Consequently, 
\begin{equation*}
Q_{k}\bR^{[k]}=\begin{pmatrix}
               0 \\
               \vdots \\
               \bm{0}_{D^{[[k-1]]}\times 1} \\
               \bR^{[[k]]}
               \end{pmatrix}, \quad 
(\bm{I}_{D^{[k]}}-Q_{k})\bR^{[k]}=\begin{pmatrix}
                               \bR^{[[0]]} \\
                               \vdots \\
                               \bR^{[[k-1]]} \\
                               \bm{0}_{D^{[[k]]}\times 1}
                             \end{pmatrix}.
\end{equation*}
Using the estimate \eqref{equ-10-11-2} and \eqref{equ-10-11-3}, we derive that
\begin{equation}\label{esti-rem}
\begin{aligned}
\norm{\bR^{[k]}(\bz^{[k,\bz(t_{n})]}(\theta))}\leq & \norm{(\bm{I}_{D^{[k]}}-Q_{k})\bR^{[k]}}+\norm{Q_{k}\bR^{[k]}} \\
\leq& C(1+\eps)\min\{\eps^{k+1},h^{k+1}\}\leq C\min\{\eps^{k+1},h^{k+1}\}.
\end{aligned}
\end{equation}
Taking norms in \eqref{equ-4-30-7} and substituting \eqref{esti-rem} leads to
\begin{equation*}
\norm{\Xi_{1}^{[k]}(t_{n+1})}\leq C\int_{t_{n}}^{t_{n+1}}\norm{\bz^{[k,\bz(t_{n})]}(\theta)-\tilde{\bz}^{[k,\bz(t_{n})]}(\theta)}\dd\theta+Ch\min\{\eps^{k+1},h^{k+1}\},
\end{equation*}
and Gronwall's inequality in integral form immediately gives
\begin{equation}\label{loc-err}
\norm{\Xi_{1}^{[k]}(t_{n+1})}\leq Ch\min\{\eps^{k+1},h^{k+1}\}.
\end{equation}
Owing to the block bidiagonal structure \eqref{equ-10-10-1} of $A_{1}^{[k]}$, the matrix exponential $e^{\frac{1}{\eps}A_{1}^{[k]}(z(t_n))(t-\theta)}$ is block lower triangular, namely,
\begin{equation*}
e^{\frac{1}{\eps}A_{1}^{[k]}(z(t_n))(t-\theta)}=\setlength{\arraycolsep}{8pt}
\begin{pmatrix}
1 &  &  & \\
\ast  & e^{\frac{1}{\eps}A_{1}(z(t_n))(t-\theta)} &  & \\
 \vdots &  \vdots & \ddots &  \\
\ast & \ast &  \cdots  & e^{\frac{1}{\eps}A_{1}^{[[k]]}(z(t_n))(t-\theta)}
\end{pmatrix}.                                             
\end{equation*}
Therefore, 
\begin{equation*}
e^{\frac{1}{\eps}A_{1}^{[k]}(z(t_n))(t-\theta)}Q_{k}\bR^{[k]}=\begin{pmatrix}
                                                                0 \\
                                                                \bm{0}_{4\times 1} \\
                                                                \vdots  \\
                                                                \bm{0}_{D^{[[k-1]]}\times 1} \\
                                                                e^{\frac{1}{\eps}A_{1}^{[[k]]}(z(t_n))(t-\theta)}\bR^{[[k]]}
                                                              \end{pmatrix},
\end{equation*}
which implies that
\begin{equation}\label{equ-10-12-2}
\Pi_{2}^{5}e^{\frac{1}{\eps}A_{1}^{[k]}(z(t_n))(t-\theta)}Q_{k}\bR^{[k]}=\bm{0}_{4\times 1}, \quad k\geq 2.
\end{equation}
Applying $\Pi_{2}^{5}$ to both sides of \eqref{equ-4-30-7} and invoking \eqref{loc-err} and \eqref{equ-10-12-2}, we obtain
\begin{equation}\label{equ-10-12-3}
\begin{aligned}
\norm{\Pi_{2}^{5}\Xi_{1}^{[k]}(t_{n+1})}\leq& C\int_{t_{n}}^{t_{n+1}}\norm{\bz^{[k,\bz(t_{n})]}(\theta)-\tilde{\bz}^{[k,\bz(t_{n})]}(\theta)}\dd\theta \\
&+\int_{t_{n}}^{t_{n+1}}\norm{(\bm{I}_{D^{[k]}}-Q_{k})\bR^{[k]}(\bz^{[k,\bz(t_{n})]}(\theta))}\dd\theta \\
\leq& C(h^2+\eps h)\min\{\eps^{k+1},h^{k+1}\} \\
\leq& C\min\{\eps^{k+1}h^{2},\eps h^{k+2}\}.
\end{aligned}
\end{equation}

\textbf{Global error.}
The global error analysis proceeds as follows. Employing the nonsingular transformation from Lemma~\ref{lem-1}, we introduce  $\bZ^{[k,\bm{0}]}=(S^{[k]}(\bZ_{n}))^{-1}\bZ^{[k,\bZ_{n}]}$ and $\tilde{\bz}^{[k,\bm{0}]}=(S^{[k]}(\bz(t_{n})))^{-1}\tilde{\bz}^{[k,\bz(t_{n})]}$. Inserting these into \eqref{trunc} and \eqref{LLES-1} leads to
\begin{equation*}
\begin{aligned}
&\frac{d\bZ^{[k,\bm{0}]}}{dt}=\frac{1}{\eps}A_{1}^{[k]}(\bm{0})\bZ^{[k,\bm{0}]}+\tilde{A}_{0}^{[k]}(\bZ_{n})\bZ^{[k,\bm{0}]}, \quad \bZ^{[k,\bm{0}]}(t_{n})=(S^{[k]}(\bZ_{n}))^{-1}e_{1}^{D^{[k]}}, \\
&\frac{d\tilde{\bz}^{[k,\bm{0}]}}{dt}=\frac{1}{\eps}A_{1}^{[k]}(\bm{0})\tilde{\bz}^{[k,\bm{0}]}+\tilde{A}_{0}^{[k]}(\bz(t_{n}))\tilde{\bz}^{[k,\bm{0}]}, \quad \tilde{\bz}^{[k,\bm{0}]}(t_{n})=(S^{[k]}(\bz(t_{n})))^{-1}e_{1}^{D^{[k]}},
\end{aligned}
\end{equation*}
where $\tilde{A}_{0}^{[k]}(\hbz)=(S^{[k]}(\hbz))^{-1}A_{0}^{[k]}(\hbz)S^{[k]}(\hbz)$. Applying the variation-of-constants formula to above and subtracting gives
\begin{equation}\label{equ-5-7-4}
\tilde{\bz}^{[k,\bm{0}]}(t)-\bZ^{[k,\bm{0}]}(t)=\fe^{\frac{1}{\eps}A_{1}^{[k]}(\bm{0})(t-t_{n})}((S^{[k]}(\bz(t_{n})))^{-1}-(S^{[k]}(\bZ_{n}))^{-1})e_{1}^{D^{[k]}}+\Xi^{[k]}_{2}(t),
\end{equation}
with $\Xi^{[k]}_{2}(t)$ defined by
\begin{equation}\label{equ-5-7-5}
\begin{aligned}
\Xi^{[k]}_{2}(t)=&\int_{t_{n}}^{t}\fe^{\frac{1}{\eps}A_{1}^{[k]}(\bm{0})(t-\theta)}(\tilde{A}_{0}^{[k]}(\bz(t_{n}))-\tilde{A}_{0}^{[k]}(\bZ_{n}))\tilde{\bz}^{[k,\bm{0}]}(\theta)\dd\theta \\
&+\int_{t_{n}}^{t}\fe^{\frac{1}{\eps}A_{1}^{[k]}(\bm{0})(t-\theta)}\tilde{A}_{0}^{[k]}(\bZ_{n})(\tilde{\bz}^{[k,\bm{0}]}(\theta)-\bZ^{[k,\bm{0}]}(\theta))\dd\theta.
\end{aligned}
\end{equation}
We now establish the Lipschitz continuity of $\tilde{A}^{[k]}_{0}(\cdot)$. Lemma~\ref{lem-1} directly implies that both $S^{[k]}(\cdot)$ and $(S^{[k]}(\cdot))^{-1}$ are Lipschitz continuous. By \eqref{equ-10-6-2}, the entries of $A_{0}^{[k]}(\cdot)$ are expressed in terms of derivatives of $E(\bx)$ and $b(\bx)$. Hence, the Lipschitz continuity of these derivatives implies that of $A_{0}^{[k]}(\cdot)$.

Combining \eqref{equ-5-7-4} and \eqref{equ-5-7-5} yields 
\begin{align*}
\norm{\tilde{\bz}^{[k,\bm{0}]}(t)-\bZ^{[k,\bm{0}]}(t)}\leq C\norm{\bz(t_{n})-\bZ_{n}}+C\int_{t_{n}}^{t}\norm{\tilde{\bz}^{[k,\bm{0}]}(\theta)-\bZ^{[k,\bm{0}]}(\theta)}\dd\theta.
\end{align*}
Gronwall's inequality then gives
\begin{equation*}
\norm{\tilde{\bz}^{[k,\bm{0}]}(t)-\bZ^{[k,\bm{0}]}(t)}\leq C\fe^{C(t-t_{n})}\norm{\bz(t_{n})-\bZ_{n}}, \quad t_{n}\leq t\leq t_{n+1}.
\end{equation*}
Inserting this into the second term of \eqref{equ-5-7-5} produces 
\begin{equation}\label{equ-5-7-7}
\norm{\Xi^{[k]}_{2}(t_{n+1})}\leq Ch\norm{\bz(t_{n})-\bZ_{n}}.
\end{equation}
By Definition~\ref{def-2}(ii), when $\hbz=\bm{0}$, we have $\tilde{\bz}(t_{n+1})=\Pi_{2}^{5}\tilde{\bz}^{[k,\bm{0}]}(t_{n+1})$ and $\bZ_{n+1}=\Pi_{2}^{5}\bZ^{[k,\bm{0}]}(t_{n+1})$. 
The block lower-triangular structure of $(S^{[k]}(\hbz))^{-1}$ in \eqref{equ-10-10-2}, together with $\bar{S}_{0,0}(\hbz)=0$, $\bar{S}_{1,1}(\hbz)=I_{4}$ and  $\bar{S}_{1,0}(\hbz)=\hbz$ from \eqref{equ-10-6-3}, gives 
\[
\left((S^{[k]}(\bz(t_{n})))^{-1}-(S^{[k]}(\bZ_{n}))^{-1}\right)e_{1}^{D^{[k]}}=\left(0,\bz(t_{n})-\bZ_{n},\ast,\cdots,\ast \right)^{\intercal}.
\] 
Consequently, applying the projection to \eqref{equ-5-7-4} and invoking \eqref{equ-22-2} yields
\begin{equation*}
\tilde{\bz}(t_{n+1})-\bZ_{n+1}=\fe^{\frac{1}{\eps}A_{1}h}(\bz(t_{n})-\bZ_{n})+\Pi_{2}^{5}\Xi^{[k]}_{2}(t_{n+1}).
\end{equation*}
Projecting \eqref{equ-4-30-7} analogously with $\Pi_{2}^{5}$ and adding to the above gives 
\begin{equation*}
\bz(t_{n+1})-\bZ_{n+1}=\fe^{\frac{1}{\eps}A_{1}h}(\bz(t_{n})-\bZ_{n})+\Pi_{2}^{5}(\Xi^{[k]}_{1}(t_{n+1})+\Xi^{[k]}_{2}(t_{n+1})).
\end{equation*}
Solving the recursion with $\bz(0)=\bZ_{0}$ leads to 
\begin{equation}\label{equ-5-7-8}
\bz(t_{n})-\bZ_{n}=\sum\limits_{j=1}^{n}\fe^{\frac{n-j}{\eps}A_{1}h}\Pi_{2}^{5}(\Xi^{[k]}_{1}(t_{j})+\Xi^{[k]}_{2}(t_{j})).
\end{equation}
Combining \eqref{equ-10-12-3}, \eqref{equ-5-7-7}, and Lemma~\ref{lem-2} produces
\begin{equation*}
\norm{\bz(t_{n})-\bZ_{n}}\leq \sum\limits_{j=1}^{n}(Ch\norm{\bz(t_{j-1})-\bZ_{j-1}}+C\min\{\eps^{k+1}h^{2},\eps h^{k+2}\}).
\end{equation*}
Applying the discrete Gronwall's lemma then gives 
\[
\norm{\bz(t_{n})-\bZ_{n}}\leq C\min\{\eps^{k+1}h,\eps h^{k+1}\}, 
\]
from which it follows that
\begin{equation*}
\norm{\bx_{n}-\bx(t_{n})}+\eps\norm{\bv_{n}-\bv(t_{n})}\leq C\min\{\eps^{k+1}h,\eps h^{k+1}\}.
\end{equation*}
This completes the proof.  
\hfill$\square$
\end{proof}

\begin{remark}\label{rem-2}
The relative sizes of the step size $h$ and $\eps$ give rise to two distinct scaling regimes for the error bound in Theorem~\ref{2d-conv}. Throughout the derivation, the two terms in the minimum share a common constant $C$, whose value may change from one estimate to the next. Accordingly, this bound can be restated as
\begin{equation}\label{equ-10-12-4}
\norm{\bx_{n}-\bx(t_{n})}+\eps\norm{\bv_{n}-\bv(t_{n})}\leq \begin{cases}
                                                             C\eps h^{k+1}, & h\leq \eps, \\
                                                             C\eps^{k+1}h, & h>\eps.
                                                            \end{cases}
\end{equation}
For either fixed $h$ or fixed $\eps$, the theoretical error of the local linear extension exponential integrator therefore exhibits a piecewise behavior with two distinct scaling regimes, with the crossover occurring at $h=\eps$, where the two bounds coincide.
\end{remark} 

For maximal ordering scaling magnetic field satisfying certain conditions, we have the following improved convergence result.

\begin{theorem}\label{theo-1}
Under the assumptions of Theorem \ref{2d-conv}, suppose in addition that the maximal ordering scaling strong magnetic field $b(\eps\bx)$ satisfies
\begin{equation}\label{equ-10-18-2}
\nabla_{\alpha}b(\alpha)|_{\alpha=\bm{0}}=\bm{0}.
\end{equation}
Let $(\bx_n,\bv_n)$ denote the numerical solution generated by \eqref{scheme-1}--\eqref{scheme-3}. Then, for any $0<h<h_0$, we have
\begin{equation}\label{equ-10-18-3}
\norm{\bx_{n}-\bx(t_{n})}+\eps\norm{\bv_{n}-\bv(t_n)}
\leq C\min\{\eps^{k+2},\eps h^{k+1}\},
\qquad k\geq 1,
\end{equation}
where the constants $C$ and $h_0$ are independent of $\eps$ and $h$.
\end{theorem}

\begin{proof}
The proof is analogous to that of Theorem \ref{2d-conv}. We therefore present only the differences. To sharpen the local error bound, we invoke Lemma~\ref{lem-4} and \eqref{equ-10-18-1} to bound the $j$-th order Taylor remainder as $\norm{\br^{j}(\bz(\theta);\bz(t_{n}))}\leq C\eps\min\{\eps^{j}, h^{j}\}$, which further implies $\norm{\bR^{[k]}(\bz^{[k,\bz(t_{n})]}(\theta))}\leq C\eps\min\{\eps^{k+1},h^{k+1}\}$. Consequently, the truncation contribution to \eqref{equ-4-30-7} is bounded by
\[
\left\Vert\int_{t_{n}}^{t}\fe^{\frac{1}{\eps}A_{1}^{[k]}(\bz(t_{n}))(t-\theta)}\bR^{[k]}(\bz^{[k,\bz(t_{n})]}(\theta))\dd\theta\right\Vert\leq C\eps h\min\{\eps^{k+1},h^{k+1}\},
\]
This refines the local error estimate to
\[
\norm{\Xi_{1}^{[k]}(t_{n+1})}\leq C\eps h\min\{\eps^{k+1},h^{k+1}\}.
\]
The global error analysis proceeds analogously, as \eqref{equ-5-7-7} and the error identity \eqref{equ-5-7-8} remain valid. By substituting the refined local error estimate into \eqref{equ-5-7-8} and applying Gronwall's inequality, we obtain
\begin{equation*}
\norm{\bx(t_{n})-\bx_{n}}+\eps\norm{\bv(t_{n})-\bv_{n}}\leq C\min\{\eps^{k+2}, \eps h^{k+1}\}.
\end{equation*}
This yields the desired result.
\hfill$\square$
\end{proof}

We next establish improved convergence results for the proposed local linear extension exponential integrator under a constant strong magnetic field, which differ only marginally from those under the maximal ordering scaling.

\begin{theorem}\label{2d-conv-cons}
Consider the CPD system \eqref{2d} under a constant strong magnetic field $b(\bx)=C$ with uniformly bounded initial data. Assume that the electric field $E(\bx)$ has Lipschitz-continuous derivatives up to order $k$. Then, for any $0<h<h_{0}$ and $k\geq 1$, the numerical solutions $\bx_{n}$ and $\bv_{n}$ generated by
\eqref{scheme-1}--\eqref{scheme-3} satisfy
\begin{equation}\label{equ-4-27-1}
\norm{\bx_{n}-\bx(t_{n})}+\eps\norm{\bv_{n}-\bv(t_n)}\leq C\min\{\eps^{k+2}, \eps h^{k+1}\},  
\end{equation}
where the constants $C, h_{0}$ are independent of $\eps, h$.
\end{theorem} 

\begin{proof}
By Remark~\ref{rem-1}, specifically \eqref{equ-10-13-1}, the Taylor remainder of order $j$ satisfies $\norm{\br^{j}(\bz(\theta);\bz(t_{n}))}\leq C\eps\min\{\eps^{j}, h^{j}\}$. The remainder of the proof follows directly from Theorem~\ref{theo-1}.
\hfill$\square$
\end{proof}

For both a maximal ordering scaling strong magnetic field satisfying \eqref{equ-10-18-2} and a constant strong magnetic field, the error estimates in \eqref{equ-10-18-3} and \eqref{equ-4-27-1} exhibit two distinct scaling regimes, namely
\begin{equation}\label{equ-10-13-2}
\norm{\bx_{n}-\bx(t_{n})}+\eps\norm{\bv_{n}-\bv(t_{n})}\leq \begin{cases}
                                                             C\eps h^{k+1}, & h\leq \eps, \\
                                                             C\eps^{k+2}, & h>\eps.
                                                            \end{cases}
\end{equation}
Consequently, the theoretical error bound exhibits a piecewise dependence on the relative magnitudes of $h$ and $\eps$.

\begin{remark}\label{rem-4}
For a general strong magnetic field, the estimate in part~(i) of Lemma~\ref{lem-3} and the conclusion of Lemma~\ref{lem-4} remain valid.
However, the nonlinear term now takes the form $F(\bz)=\dfrac{b(\bx)-b(\bx_{0})}{\eps}J\bw+\eps E(\bx)$. Thus, in Case~2 of Lemma~\ref{lem-3}(ii), we have $\abs{q}=1$, $\abs{p}=k-1$ and
\[
\partial_{\bx}^{p}\partial_{w_{2}}F_{1}=-\partial_{\bx}^{p}\partial_{w_{1}}F_{2}=\begin{cases}
                                                                                 \dfrac{b(\bx)-b(\bx_{0})}{\eps}, &  k=1, \\ 
                                                                                 \dfrac{\partial_{\bx}^{p}b(\bx)}{\eps}, & k\geq 2.
                                                                                 \end{cases}
\]
In contrast to the maximal-ordering scaling, no compensating factor $\eps^{k-1}$ is generated for $k\geq2$ to offset the $\dfrac{1}{\eps}$ amplitude. Consequently, the derivative estimates in part~(ii) of Lemma~\ref{lem-3} must be replaced by 
\[
 \sum\limits_{\xi\in\mI_{4}^{[[k]]}}\left\Vert \frac{1}{\xi_{1}!\cdots\xi_{4}!}\frac{\partial^{\xi}F(\hbz)}{\partial\bz^{\xi}}\right\Vert\leq \begin{cases} C, & k=1,\\ \dfrac{C}{\eps}, & k\geq 2. \end{cases}
\]
Accordingly, the $j$-th order Taylor remainder satisfies $\|\br^{j}(\bz(\theta);\bz(t_{n}))\|\leq \dfrac{C}{\eps}\min\{\eps^{j},h^{j}\}$ for $j=2,\ldots,k+1$, which yields the local error estimate $\|\Xi_{1}^{[k]}(t_{n+1})\|\leq \dfrac{Ch}{\eps}\min\{\eps^{k+1},h^{k+1}\}$. Therefore, the resulting global error bound is no longer uniform with
respect to $\eps$, and the methods lose their uniform accuracy for general strong magnetic fields.
\end{remark}

\section{Numerical results}\label{sec-4}
In this section, numerical experiments are conducted to validate the proposed method for the CPD \eqref{2d} under three strong magnetic field settings: one constant magnetic field and two maximal-ordering scaling fields. Hereafter, the $k+1$-th exponential integrator based on the local linear expansion is denoted by EI$(k+1)$, which employs the matrices $A_{1}^{[k]}$ and $A_{0}^{[k]}$ with $k\geq 1$. The reference solution is computed using the classical fourth-order Runge-Kutta method (RK4) with a very fine step size $h=4\times 10^{-7}$. To test the convergence of the proposed exponential integrator, we evaluate the relative errors in position and velocity separately, i.e., $err_{\bx}=\frac{\norm{\bx^{n}-\bx(t_{n})}_{\infty}}{\norm{\bx(t_{n})}_{\infty}}$ and $err_{\bv}=\frac{\norm{\bv^{n}-\bv(t_{n})}_{\infty}}{\norm{\bv(t_{n})}_{\infty}}$. The relative error in energy is defined as $err_{H}=\frac{\abs{H(t)-H(0)}}{H(0)}$.

\textbf{Example 1. (Constant)}
We first consider the constant strong magnetic field $b(\bx)=1$, with the electric field defined by
\[
E(\bx)=\begin{pmatrix}
\frac{1}{2}\cos(x_{1}/2)\sin(x_{2}) \\ \sin(x_{1}/2)\cos(x_{2})
\end{pmatrix},
\]
which is derived from the potential $U(\bx)=-\sin(x_{1}/2)\sin(x_{2})$. The initial position and velocity are prescribed as $\bx_{0}=(0.8,0.9)^{\intercal}$ and $\bv_{0}=(0.5,0.6)^{\intercal}$, with the final time set to $T=3$. We first examine the dependence of the errors on $h$ for various values of $\eps$. Figures~\ref{fig-1-1} displays the convergence behavior of the EI$(k+1)$ scheme for $k=1, 2$, and $3$. For relatively large $\eps$, the method achieves the expected convergence rate $\mathcal{O}(h^{k+1})$. At intermediate values of $\eps$, the error curves develop a pronounced piecewise behaviour. As $\eps$ is reduced further, the numerical error remains bounded across the entire range of tested step sizes, in accordance with the uniform accuracy established in Theorem~\ref{2d-conv-cons}. Moreover, the transition point shifts toward smaller values of $h$ as $\eps$ decreases, consistent with the prediction in \eqref{equ-10-13-2} that the transition occurs when $h=\mathcal{O}(\eps)$.
 
Figure \ref{fig-1-2} next examines the dependence of the errors on $\eps$ for several fixed step sizes $h$, so as to verify the uniformity of the method. For relatively large $h$, the position error of the EI$(k+1)$ method decays at the rate $\mathcal{O}(\eps^{k+2})$, whereas the velocity error scales as  $\mathcal{O}(\eps^{k+1})$. For small $h$, the position error converges at $\mathcal{O}(\eps)$, while the velocity error remains uniformly bounded. For intermediate $h$, the convergence is piecewise, again in agreement with \eqref{equ-10-13-2}. These results demonstrate the improved uniform accuracy of the local linear extension exponential integrator. Remarkably, even for $h=\mathcal{O}(1)$, far beyond the cyclotron period, the error decreases as $\mathcal{O}(\eps^{k+2})$ independently of $h$, thereby enabling highly accurate simulations with substantially enlarged time steps. Finally, the energy error obtained with EI4 is shown in Figure \ref{fig-1-3}, demonstrating excellent long-term near-conservation over the time interval $[0,1000]$.

\begin{figure}[htbp]
    \centering
    \includegraphics[height=3.8cm,width=5.8cm]{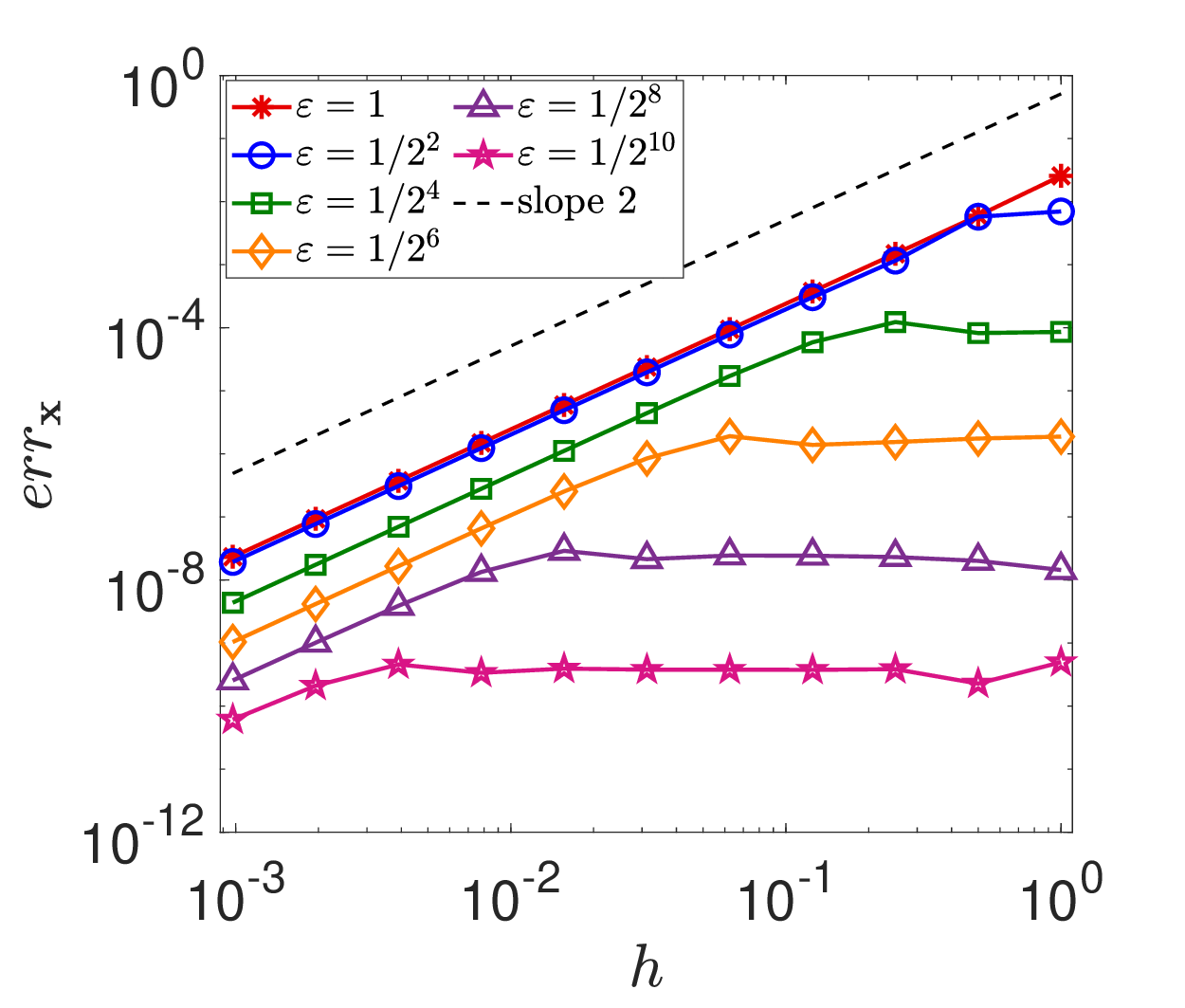}
    \hfill
    \includegraphics[height=3.8cm,width=5.8cm]{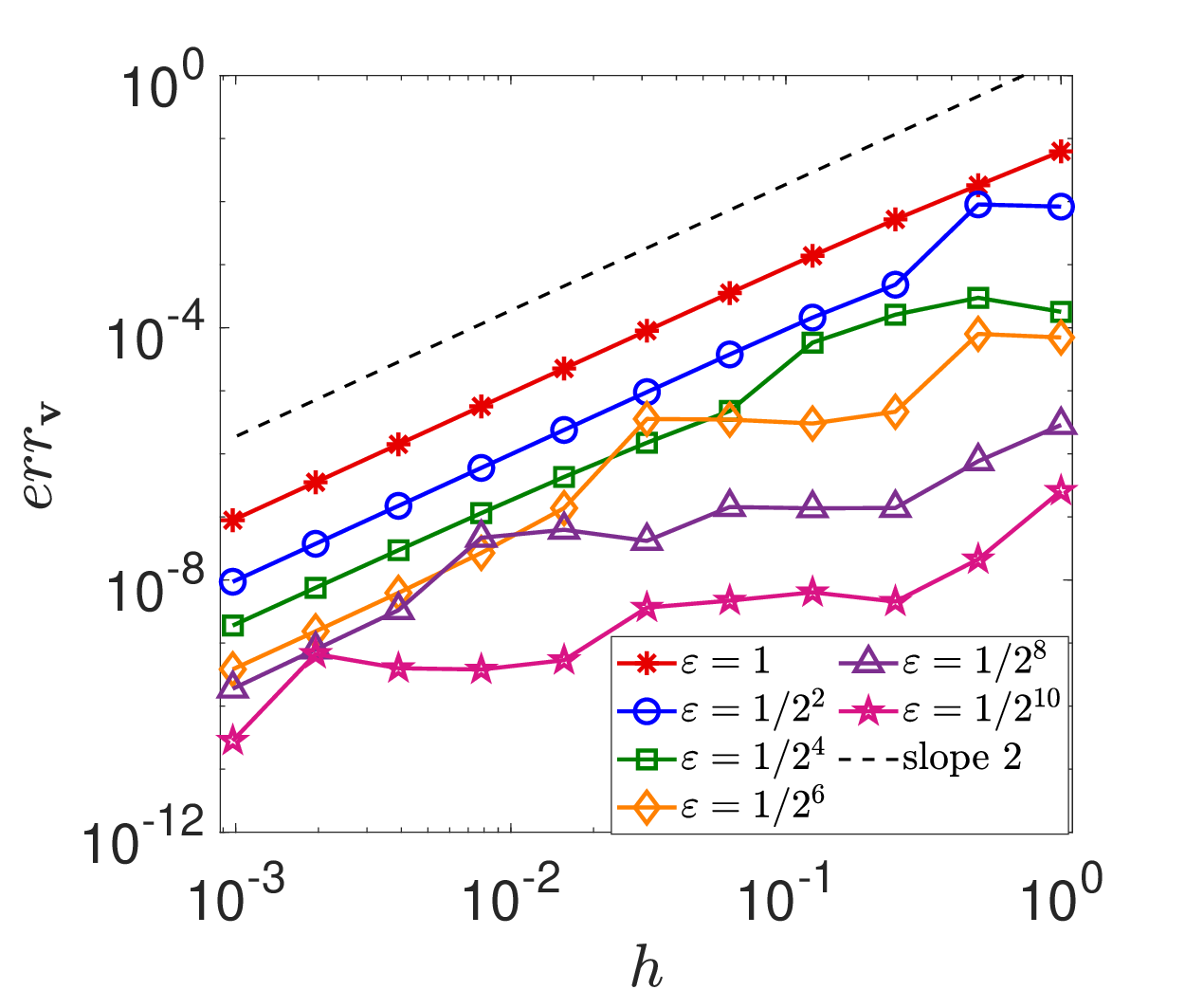}
    \par
    \includegraphics[height=3.8cm,width=5.8cm]{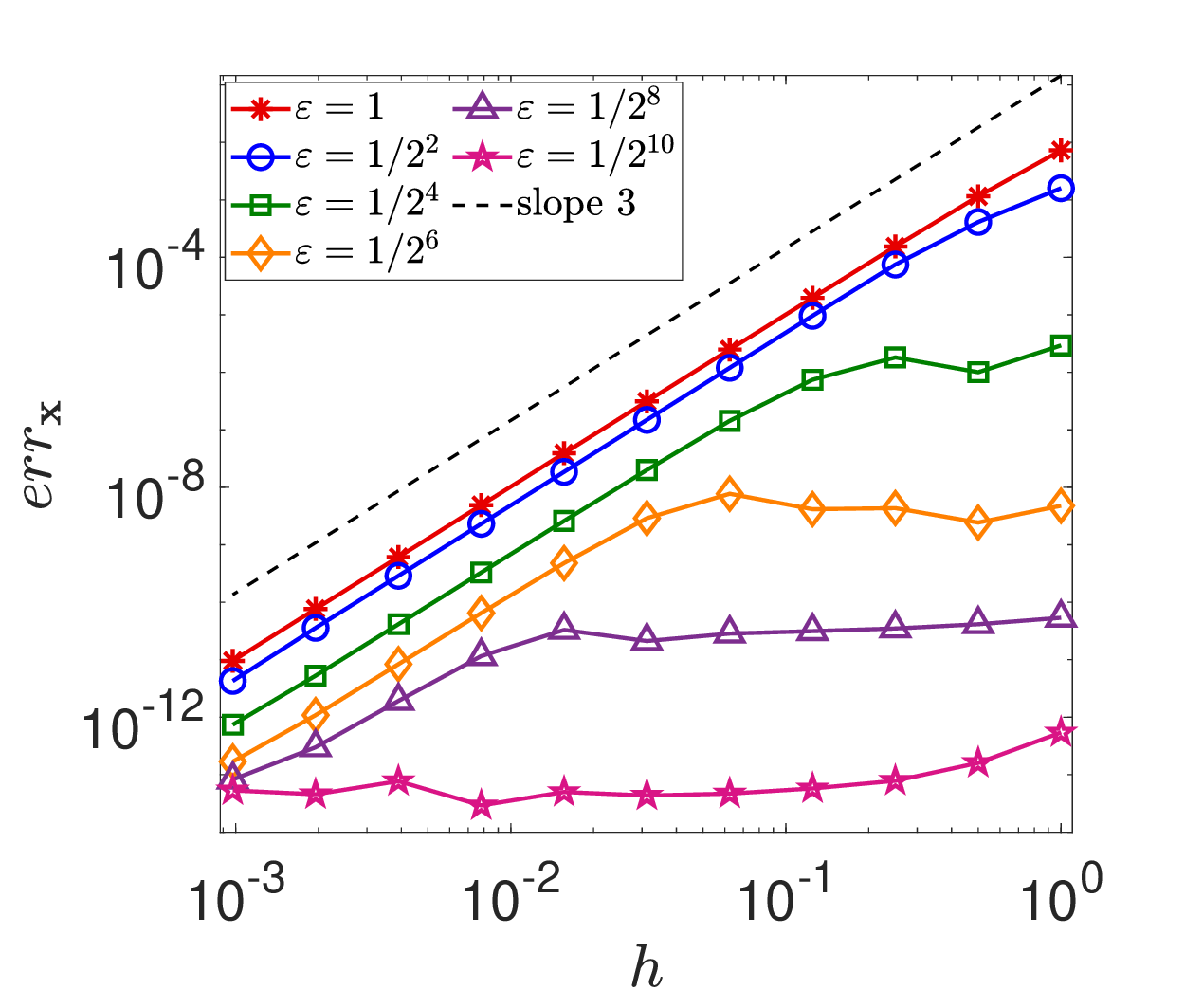}
    \hfill
    \includegraphics[height=3.8cm,width=5.8cm]{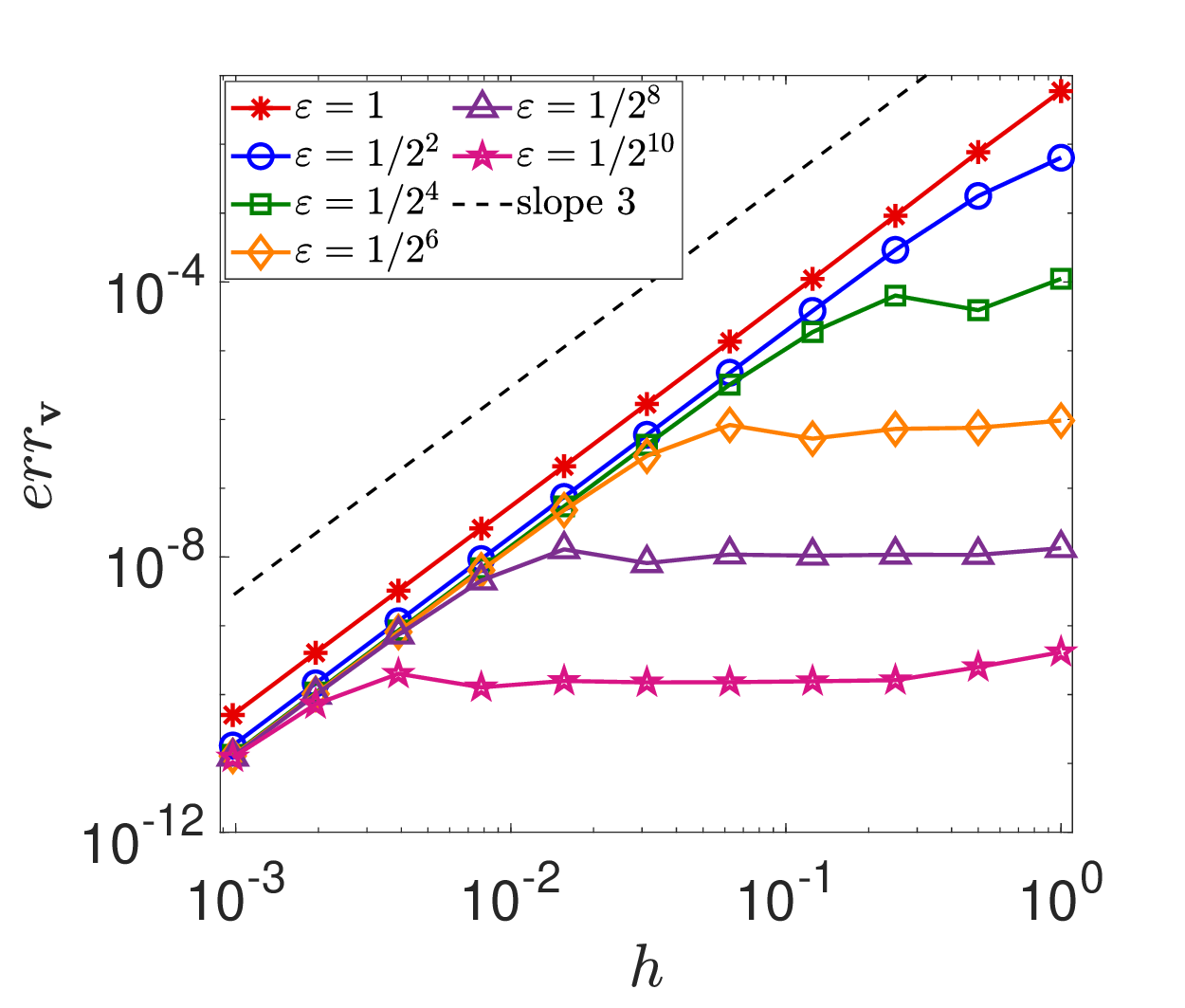}
    \par
    \includegraphics[height=3.8cm,width=5.8cm]{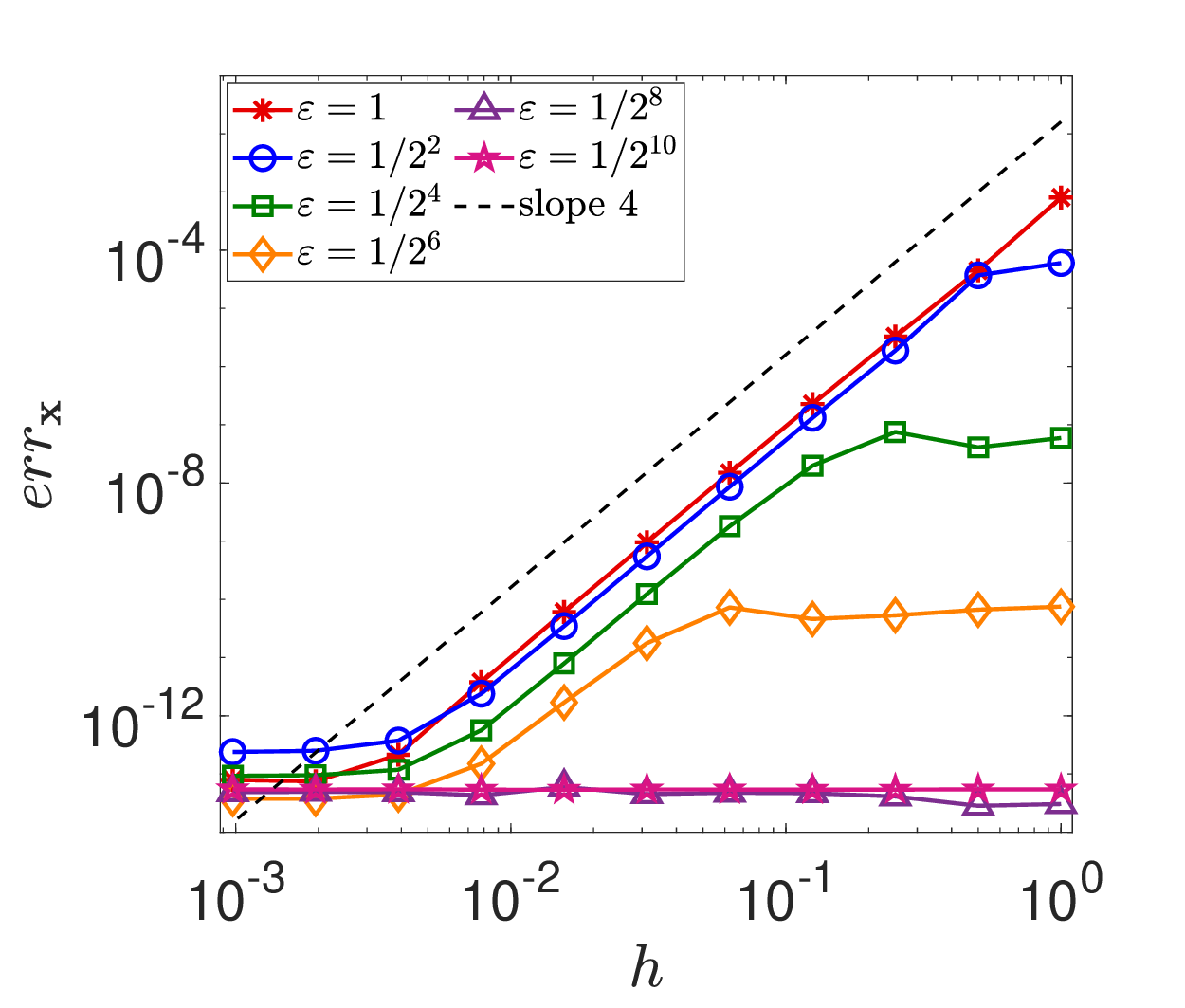}
    \hfill
    \includegraphics[height=3.8cm,width=5.8cm]{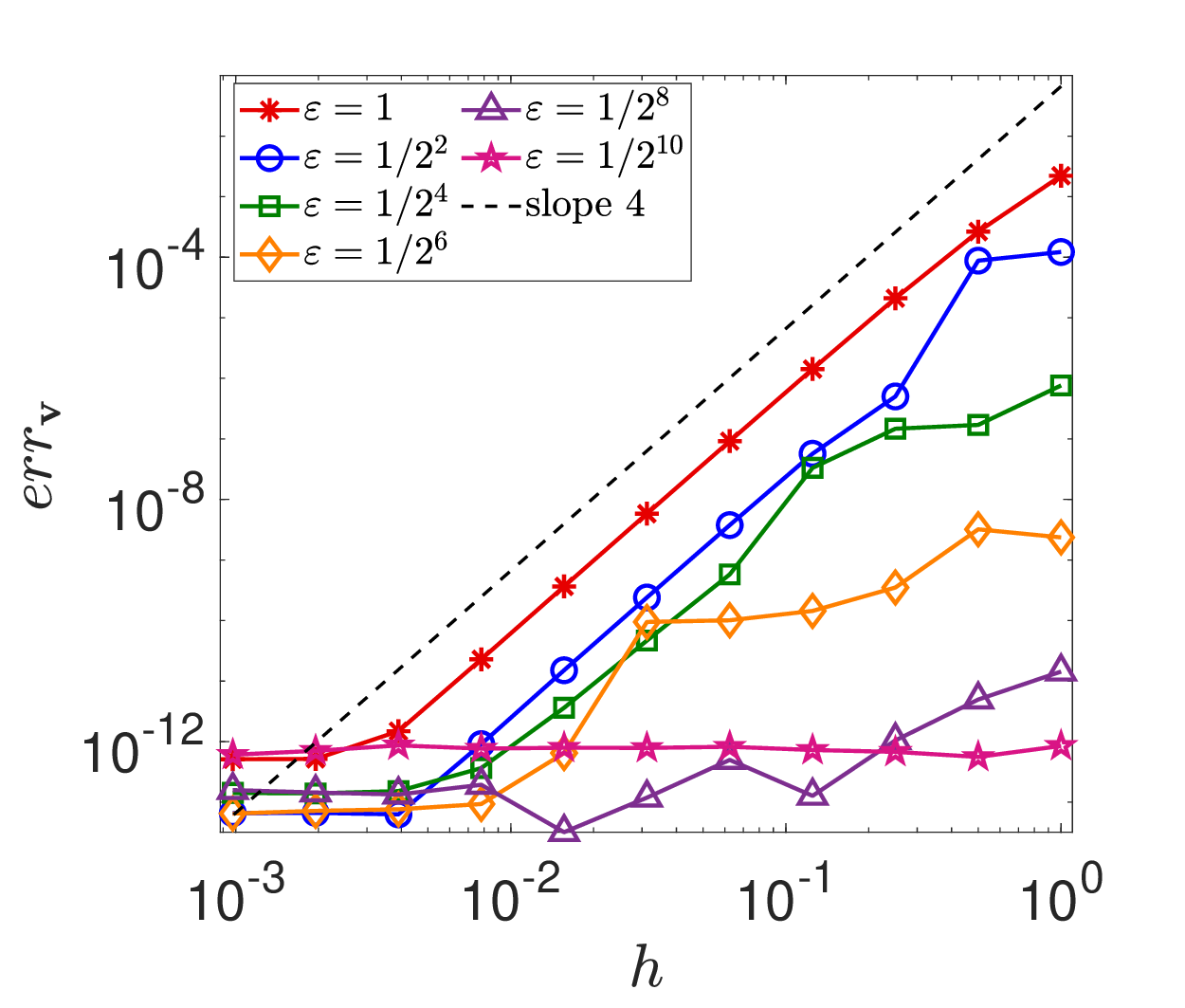}
    \caption{Example 1. Temporal errors $err_{\bx}$ and $err_{\bv}$ versus $h$ for the CPD with $k=1,2,3$ (from top to bottom).}
    \label{fig-1-1}
\end{figure}

\begin{figure}[htbp]
    \centering
    \includegraphics[height=3.8cm,width=5.8cm]{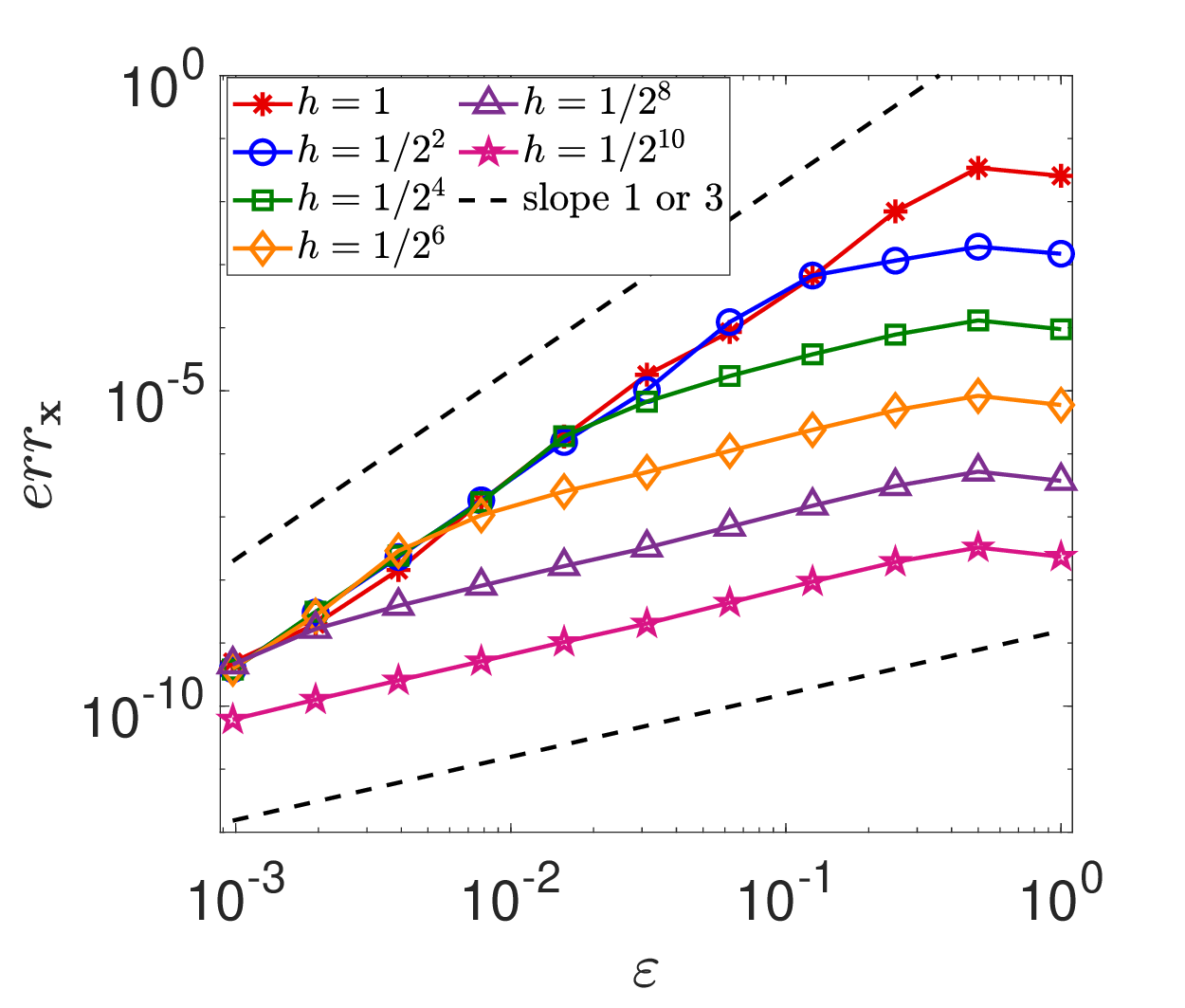}
    \hfill
    \includegraphics[height=3.8cm,width=5.8cm]{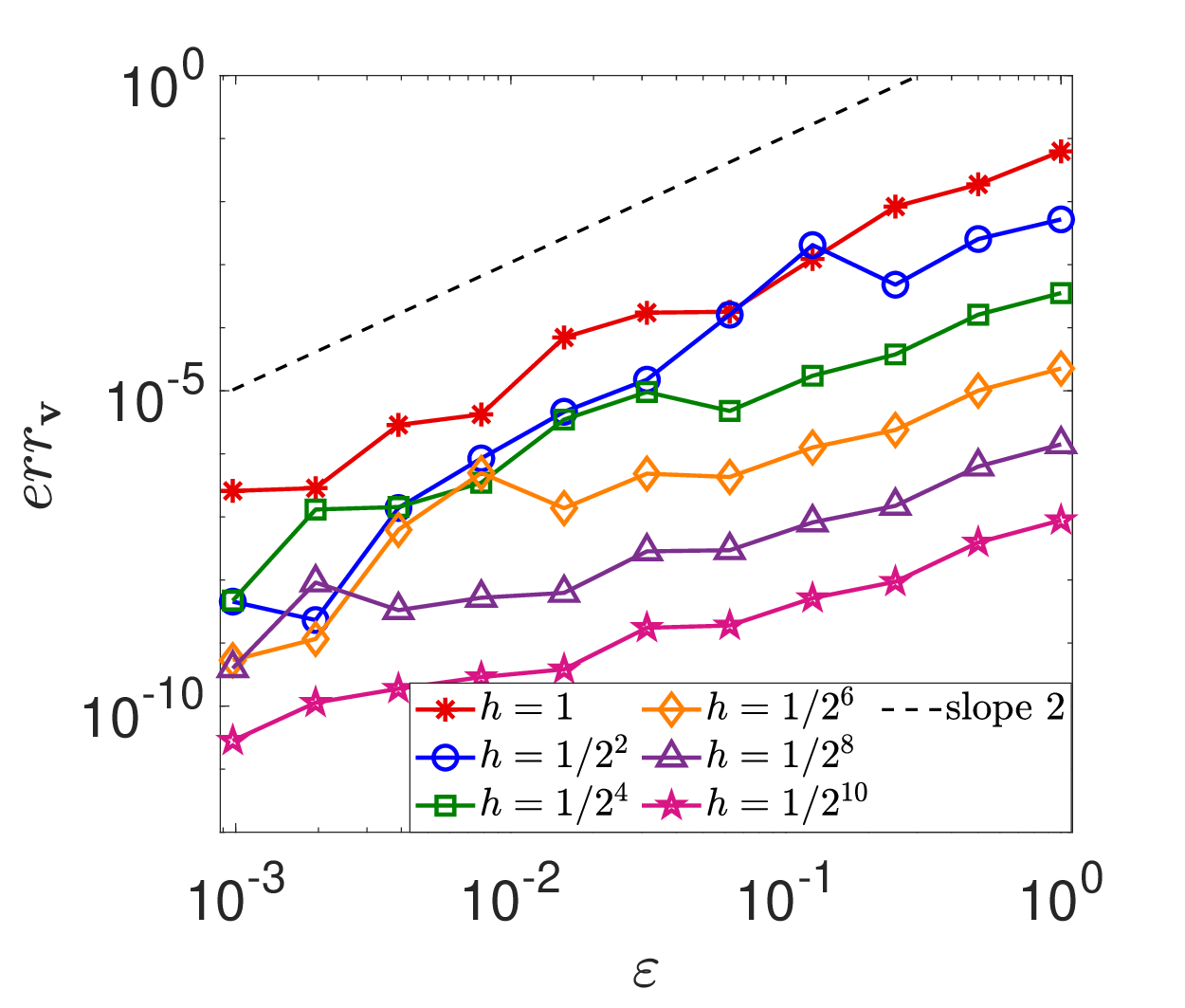}
    \par
    \includegraphics[height=3.8cm,width=5.8cm]{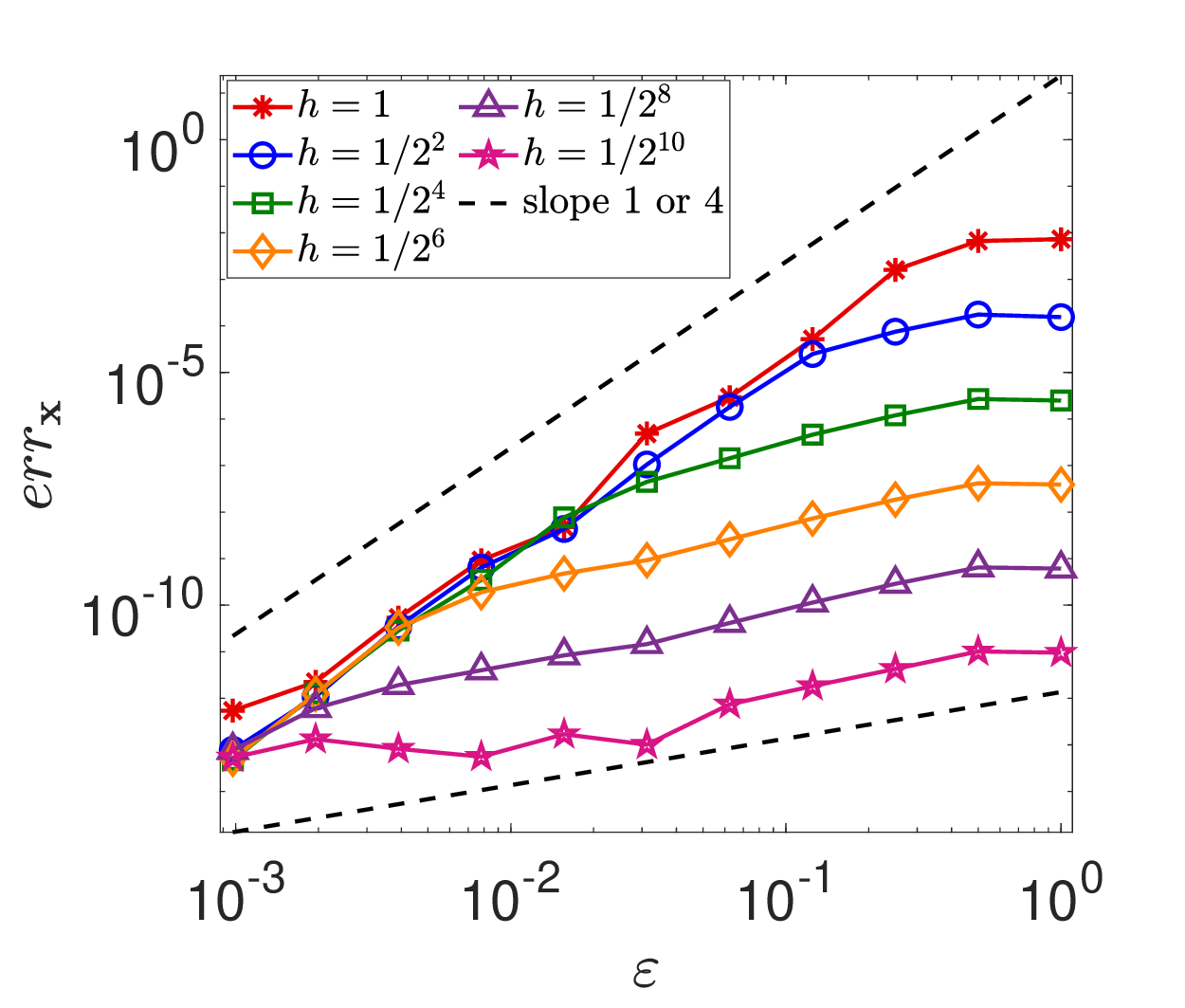}
    \hfill
    \includegraphics[height=3.8cm,width=5.8cm]{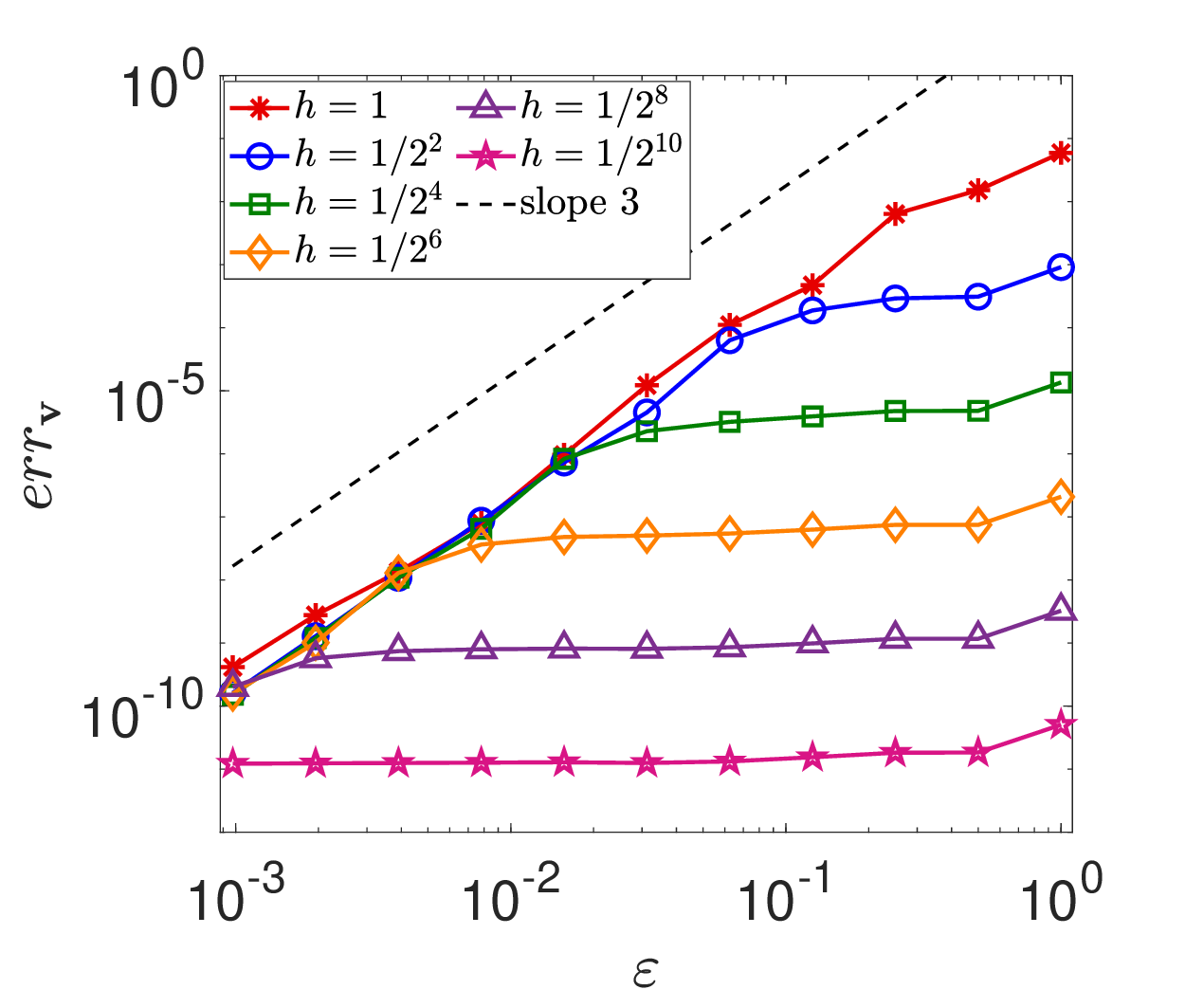}
    \par
    \includegraphics[height=3.8cm,width=5.8cm]{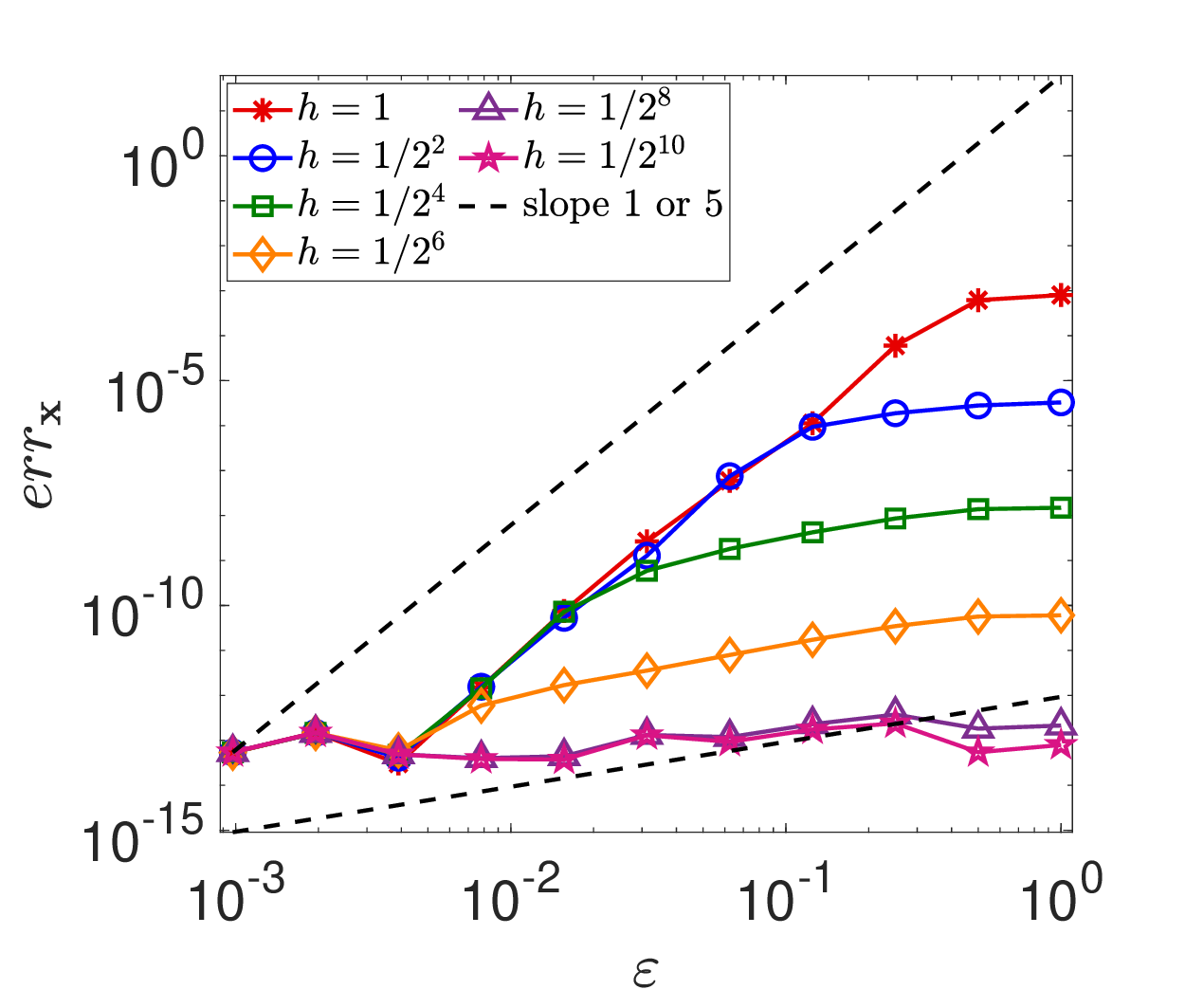}
    \hfill
    \includegraphics[height=3.8cm,width=5.8cm]{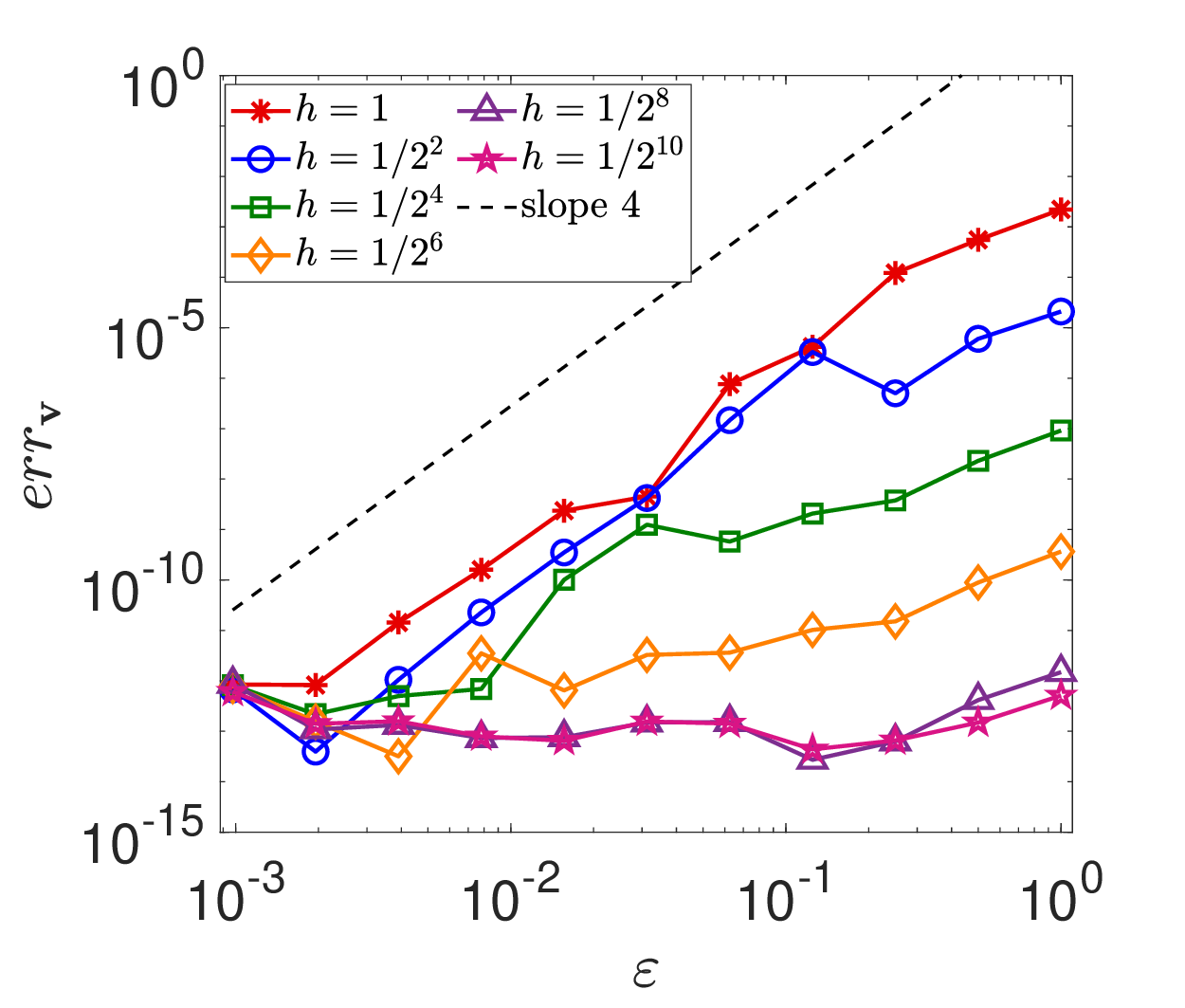}
    \caption{Example 1. Temporal errors $err_{\bx}$ and $err_{\bv}$ versus $\eps$ for the CPD with $k=1,2,3$ (from top to bottom).}
    \label{fig-1-2}
\end{figure}

\begin{figure}[htbp]
    \centering
    \includegraphics[height=3.8cm,width=5.8cm]{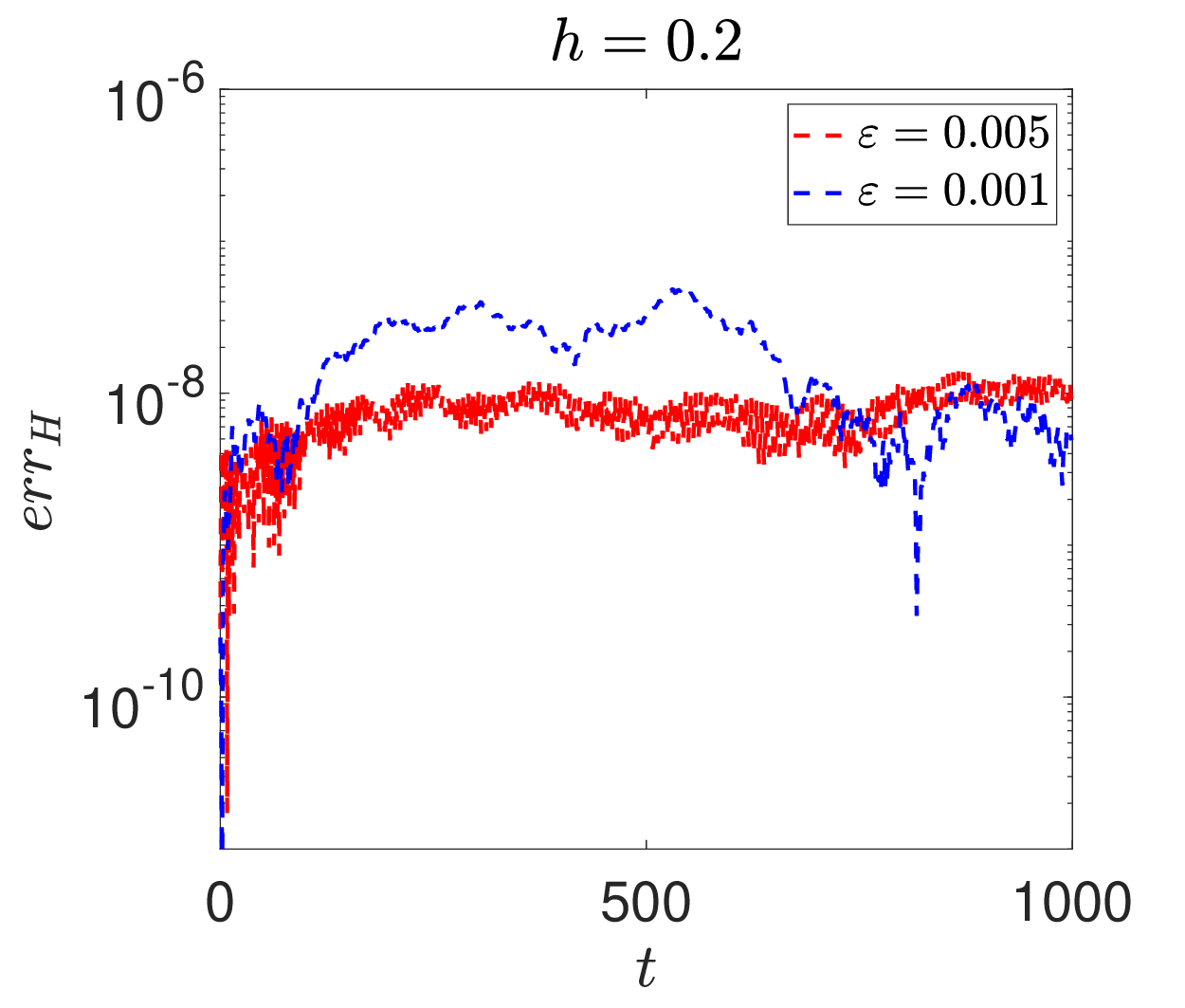}
    \hfill
    \includegraphics[height=3.8cm,width=5.8cm]{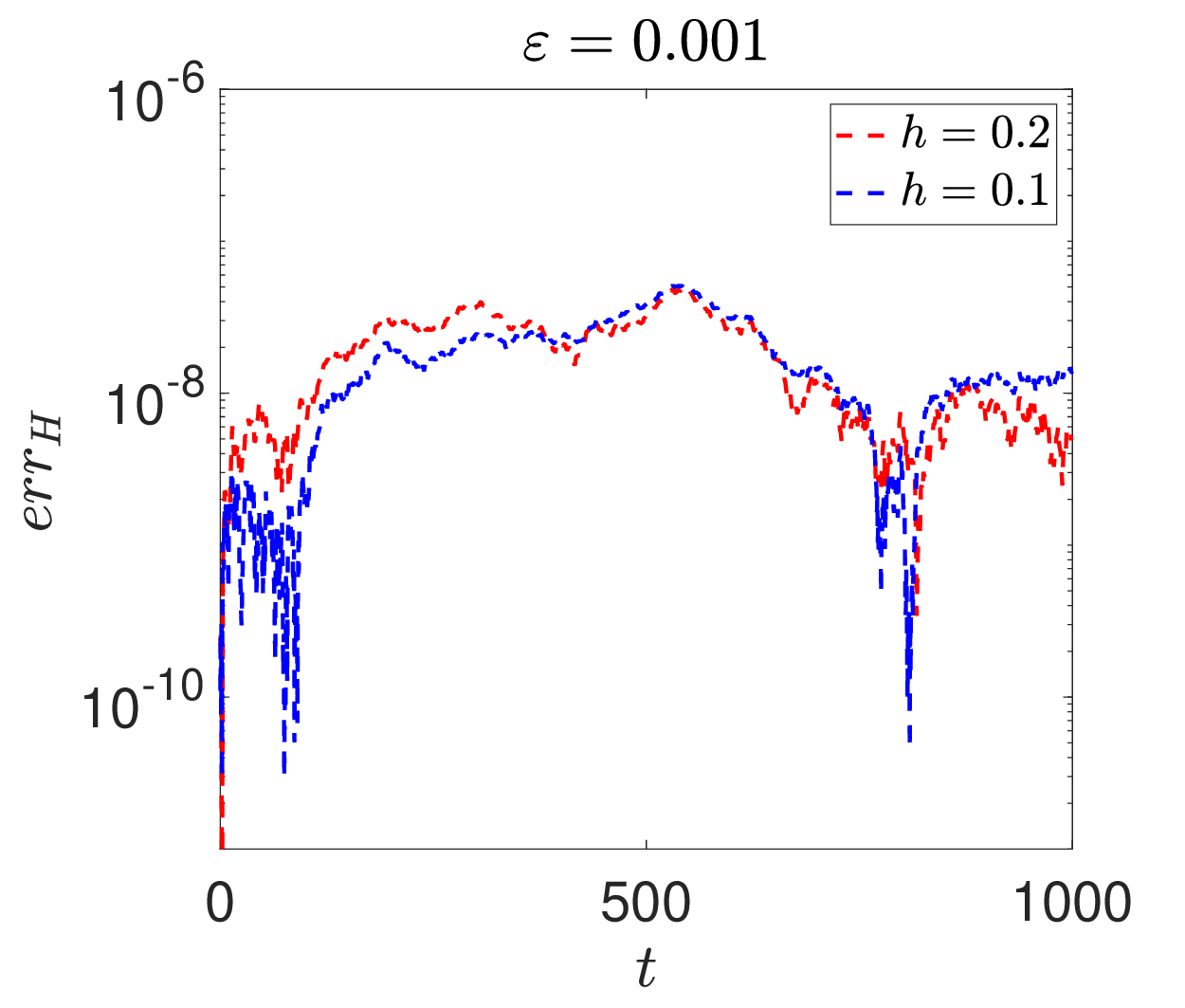}
   \caption{Example 1. Energy errors of EI4 for the CPD with various values of $\eps$ and $h$.}
   \label{fig-1-3}
\end{figure}

\vspace{0.1cm}
\textbf{Example 2. (Maximal ordering scaling)}\label{exam-2}
To investigate the maximal-ordering scaling regime, we set the magnetic and electric fields as
\[
b(\eps\bx)=1+\frac{1}{2}\sin(\eps x_{1})\cos(\eps x_{2}),
\quad
E(\bx)=\begin{pmatrix}
         \frac{x_{1}}{(1+x_{1}^2+x_{2}^2)^{\frac{3}{2}}} \\ \frac{x_{2}}{(1+x_{1}^2+x_{2}^2)^{\frac{3}{2}}}
       \end{pmatrix}
\] 
with $U(\bx)=1/\sqrt{1+x_{1}^{2}+x_{2}^{2}}$. Initial conditions are given by $\bx_{0}=(\frac{1}{3},\frac{1}{4})^{\intercal}$ and $\bv_{0}=(\frac{2}{5},\frac{2}{3})^{\intercal}$, with the final time fixed at $T=3$. In this setting, the magnetic field gradient at the origin is $\nabla_{\alpha}b(\alpha)|_{\alpha=\bm{0}}=(\frac{1}{2},0)\neq \bm{0}$.

We first verify the convergence order of the proposed EI$(k+1)$ schemes for $k\geq 2$ under a maximal ordering strong magnetic field. Figure \ref{fig-2-1} illustrates the dependence of the EI3 and EI4 errors on $h$ for various $\eps$. It is observed that when $\eps$ is relatively large compared to $h$, both schemes achieve an $\mathcal{O}(h^{k+1})$ convergence order in both position and velocity. In the intermediate regime, a piecewise behavior emerges; as $\eps$ becomes small, both schemes converge at the order $\mathcal{O}(h)$. Figure \ref{fig-2-2} then shifts the focus to the $\eps$-dependence of the errors across different values of $h$. The results indicate that for relatively large $h$, both schemes produce position and velocity errors of $\mathcal{O}(\eps^{k+1})$ and $\mathcal{O}(\eps^{k})$, respectively, for $k=2,3$. As $h$ decreases, the errors transition near the point $h=\mathcal{O}(\eps)$; for sufficiently small $h$, both schemes attain an $\mathcal{O}(\eps)$ position error while maintaining uniform accuracy in velocity. These numerical results are consistent with the theoretical predictions of Theorem \ref{2d-conv}.

To highlight the advantages of the proposed EI$(k+1)$, we next compare them with several existing numerical schemes, including traditional exponential integrators, existing local linear extension exponential integrators, and a two-scale method with uniform accuracy. The numerical methods considered are detailed below:
\begin{itemize}
  \item EXPRB$k$: the exponential Rosenbrock method of order $k$ from \cite{HOS}.
  \item TSFEI$k$: the explicit uniformly accurate two-scale exponential integrator of order $k$ from \cite{WJ}. 
  \item ERK$k$: the explicit exponential Runge--Kutta method of order $k$ from \cite{HO1}.
  \item RK$k$: the classical Runge--Kutta method of order $k$.
\end{itemize} 

Figure~\ref{fig-2-3} compares the temporal errors of $\bv$ for five methods at third and fourth order. In the weakly oscillatory regime with $\varepsilon=1/2^3$, the schemes reach their expected convergence orders. The exponential integrators yield much smaller errors than the classical RK methods, with TSFEI being the most accurate. In the highly oscillatory regime with $\varepsilon=1/2^8$, the EI methods outperform EXPRB, ERK, and classical RK throughout the tested step-size range. Large errors are incurred by the classical RK methods, whereas EXPRB and ERK exhibit a pronounced levelling-off of accuracy at larger step sizes. EI3 displays piecewise convergence consistent with the theoretical error estimates. TSFEI3 generally achieves smaller errors than EI3 and follows the third-order reference slope more closely, whereas EI4 provides accuracy comparable to TSFEI4. Increasing the EI order from three to four reduces errors by several orders of magnitude, which further widens the accuracy gap over EXPRB, whose error reduction is negligible at larger step sizes.

Figure~\ref{fig-2-4} further examines the dependence of these errors on $\eps$ at fixed step sizes. For the larger step size $h=1/2^3$, EI and TSFEI generally yield smaller errors than EXPRB, ERK, and RK in the small-$\eps$ regime, with RK being particularly sensitive to the increasing oscillation frequency. For the smaller step size $h=1/2^8$, the advantage of EI and TSFEI becomes more pronounced as $\eps$ decreases. The errors of RK, ERK, and EXPRB grow significantly, whereas EI and TSFEI maintain much lower error levels. Notably, EI4 achieves the smallest errors over a broad range of $\eps$, further demonstrating the benefit of increasing the EI order in the highly oscillatory regime.

Figure~\ref{fig-2-5} compares the total error $err_{\bx}+err_{\bv}$ versus CPU time for the third-order schemes with $\eps=1/2^8$ and the fourth-order schemes with $\eps=1/2^{10}$. 
\begin{table}[htbp]
\centering
\caption{Initial step sizes $h_0$ for Figure~\ref{fig-2-5}.}
\label{tab-1}
\renewcommand{\arraystretch}{1.2}
\setlength{\tabcolsep}{10pt} 
\begin{tabular}{cccccc}
\toprule 
$\varepsilon$ & EI & EXPRB & TSFEI & ERK & RK \\
\midrule
$1/2^8$ & $1/2$ & $1/50$ & $1/2$ & $1/200$ & $1/20000$ \\
$1/2^{10}$ & $1/2$ & $1/50$ & $1/2$ & $1/500$ & $1/40000$ \\
\bottomrule
\end{tabular} 
\end{table}
In these comparisons, the step sizes for the RK methods are restricted to $h< C\eps$, whereas the EI$(k+1)$ and TSFEI methods use substantially larger step sizes. Despite these different step-size ranges, EI3 and EI4 yield smaller errors than EXPRB, ERK, and RK at comparable CPU times. In particular, TSFEI remains highly efficient, while EI4 attains the smallest errors among all the tested methods. These results confirm the efficiency of EI$(k+1)$ with large step sizes, an advantage that becomes more pronounced as the order increases.
Finally, Figure~\ref{fig-2-6} shows that EI4 maintains small energy errors over the time interval $[0,1000]$ for all tested values of $\eps$ and $h$, indicating good long-term near-conservation of energy.

\begin{figure}[htbp]
    \centering
    \includegraphics[height=3.8cm,width=5.8cm]{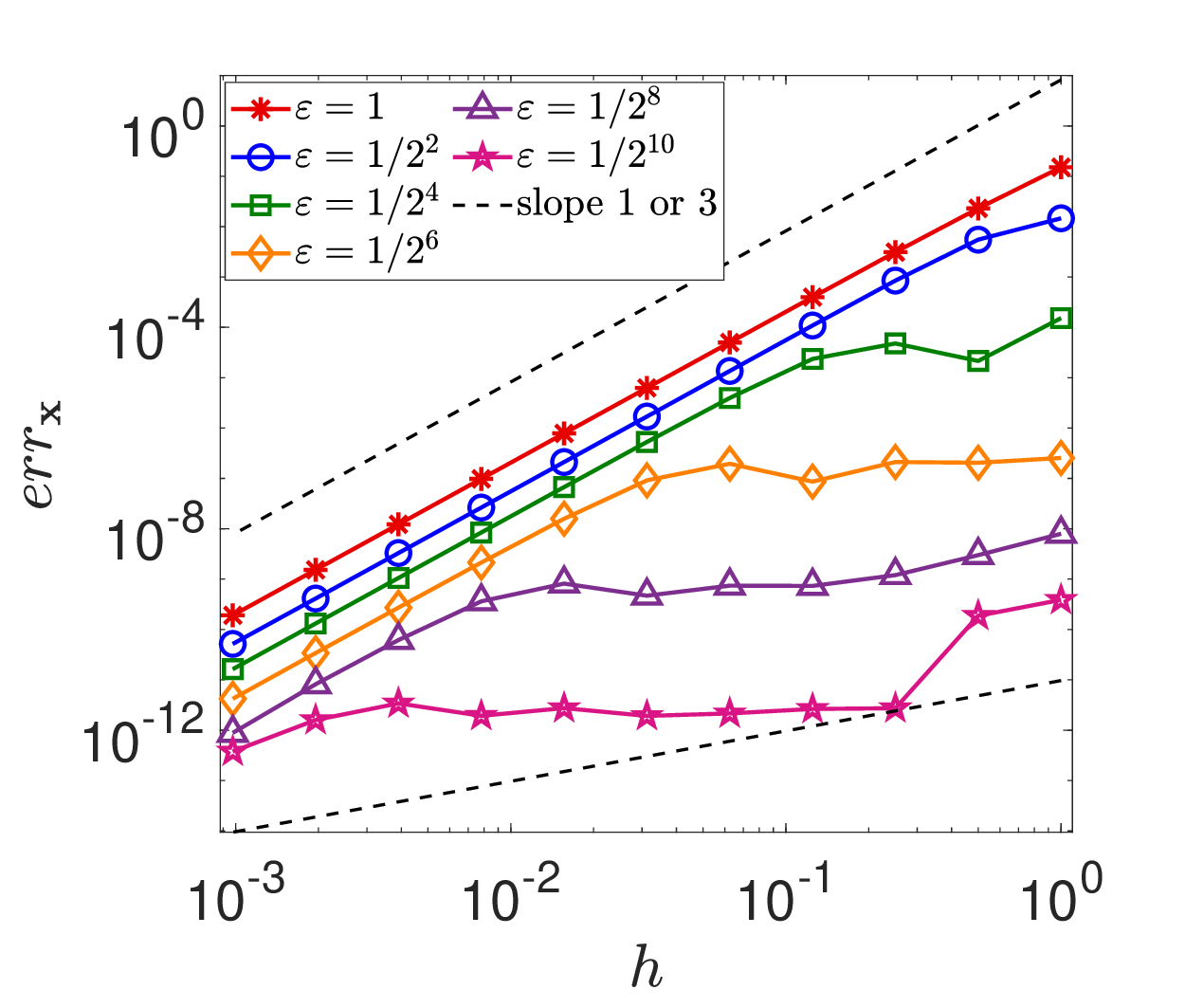}
    \hfill
    \includegraphics[height=3.8cm,width=5.8cm]{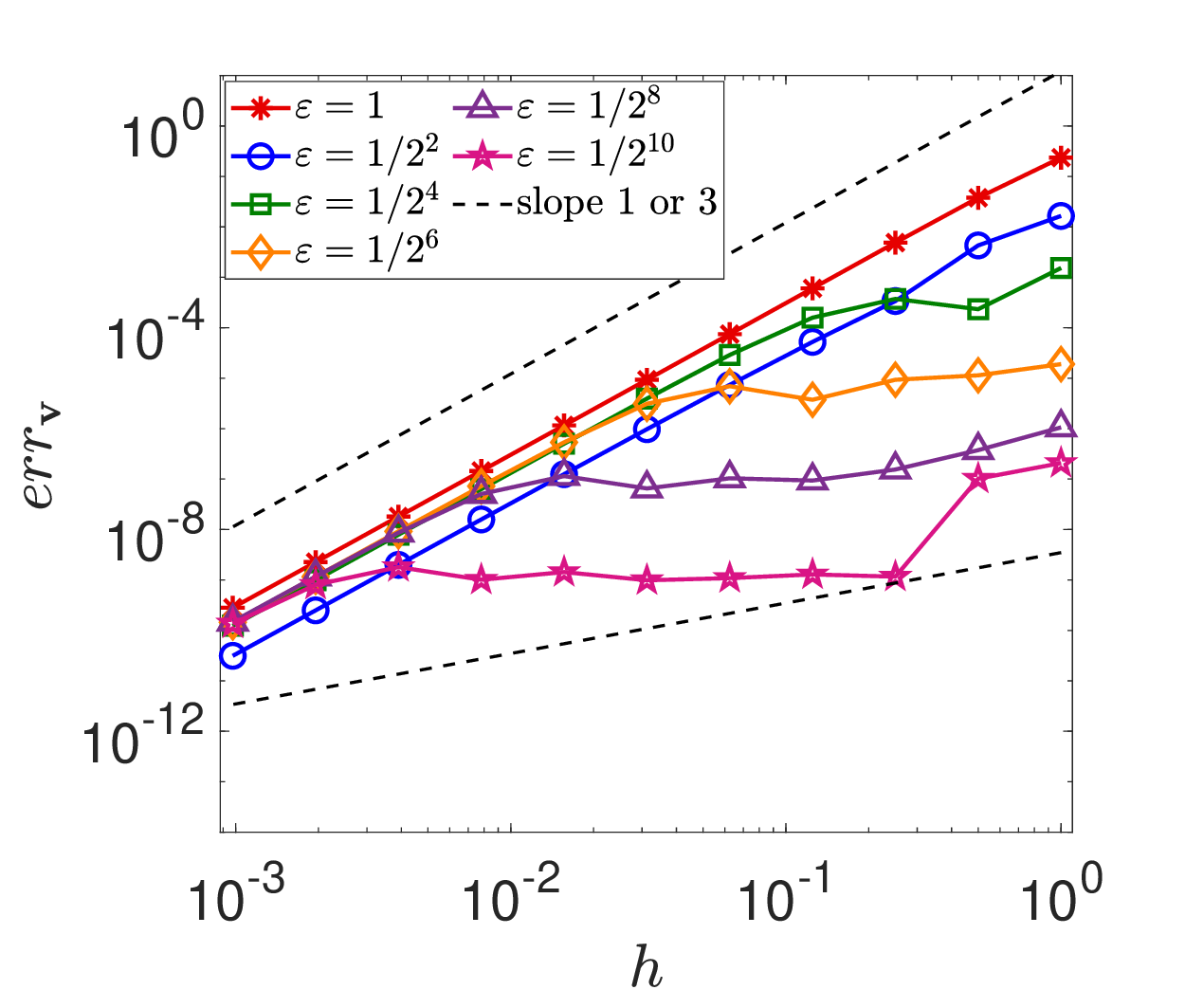}
    \par
    \includegraphics[height=3.8cm,width=5.8cm]{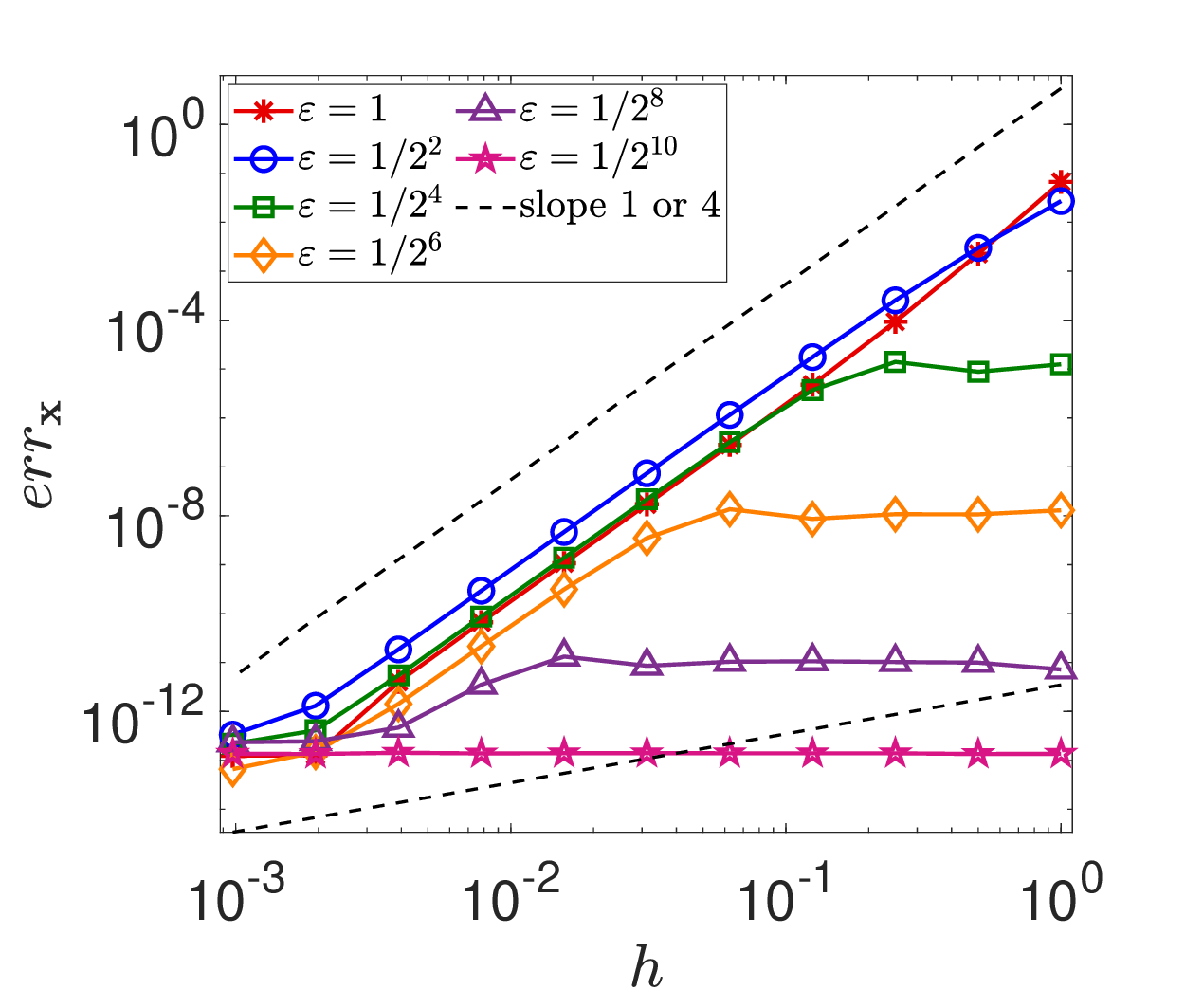}
    \hfill
    \includegraphics[height=3.8cm,width=5.8cm]{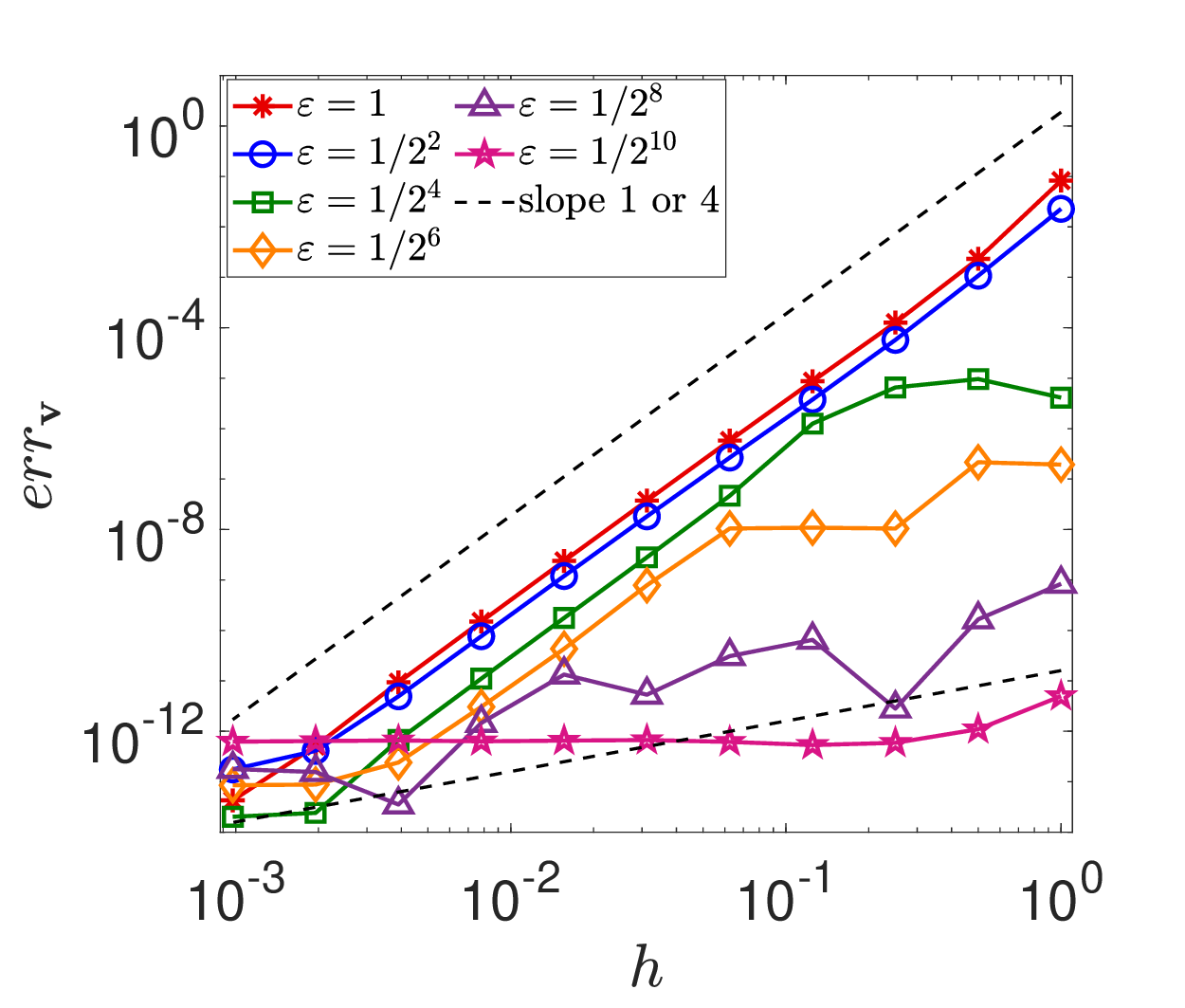}
    \caption{Example 2. Temporal errors $err_{\bx}$ and $err_{\bv}$ versus $h$ for the CPD with $k=2$ (first row) and $k=3$ (second row).}
    \label{fig-2-1}
\end{figure}

\begin{figure}[htbp]
    \centering
    \includegraphics[height=3.8cm,width=5.8cm]{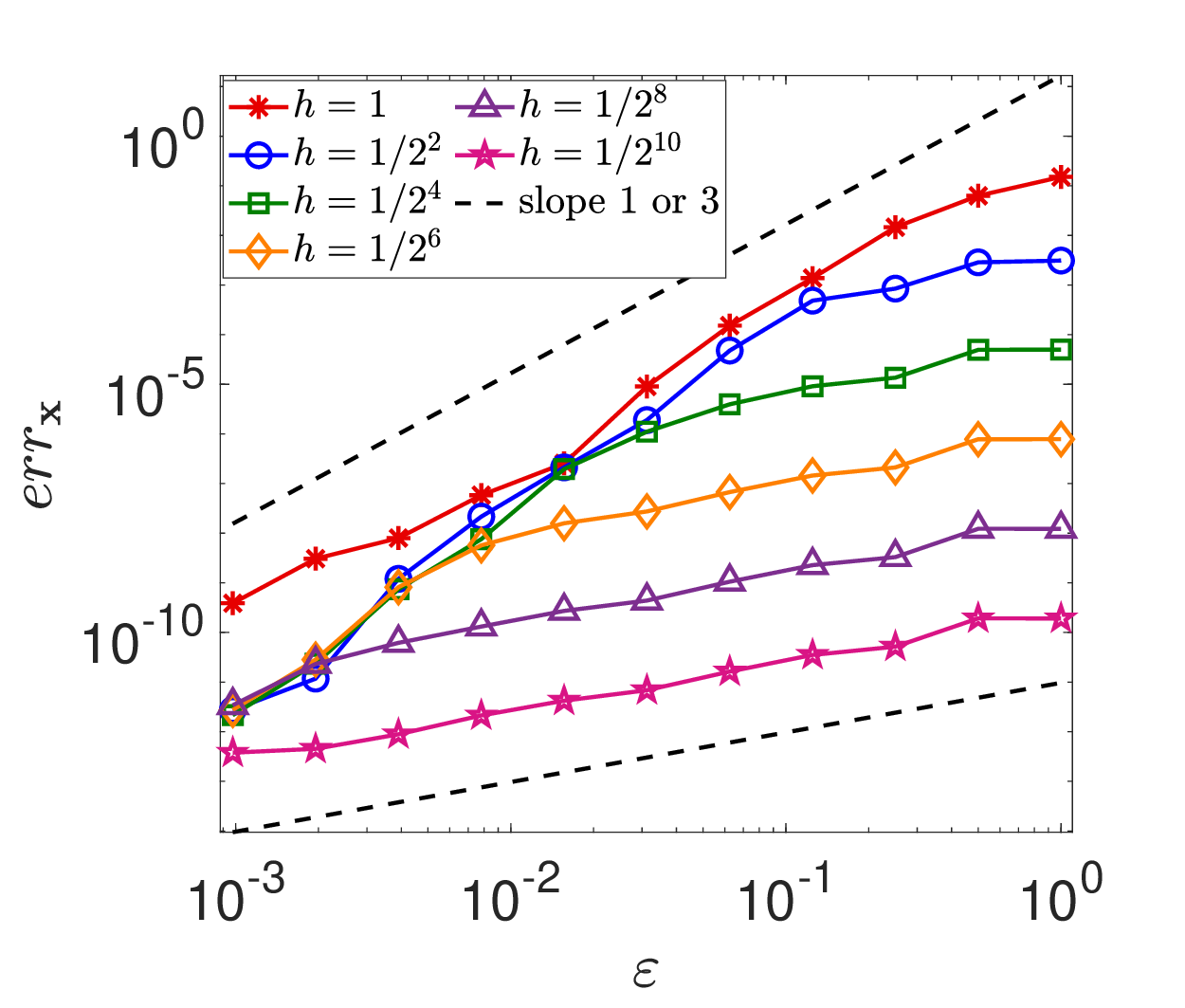}
    \hfill
    \includegraphics[height=3.8cm,width=5.8cm]{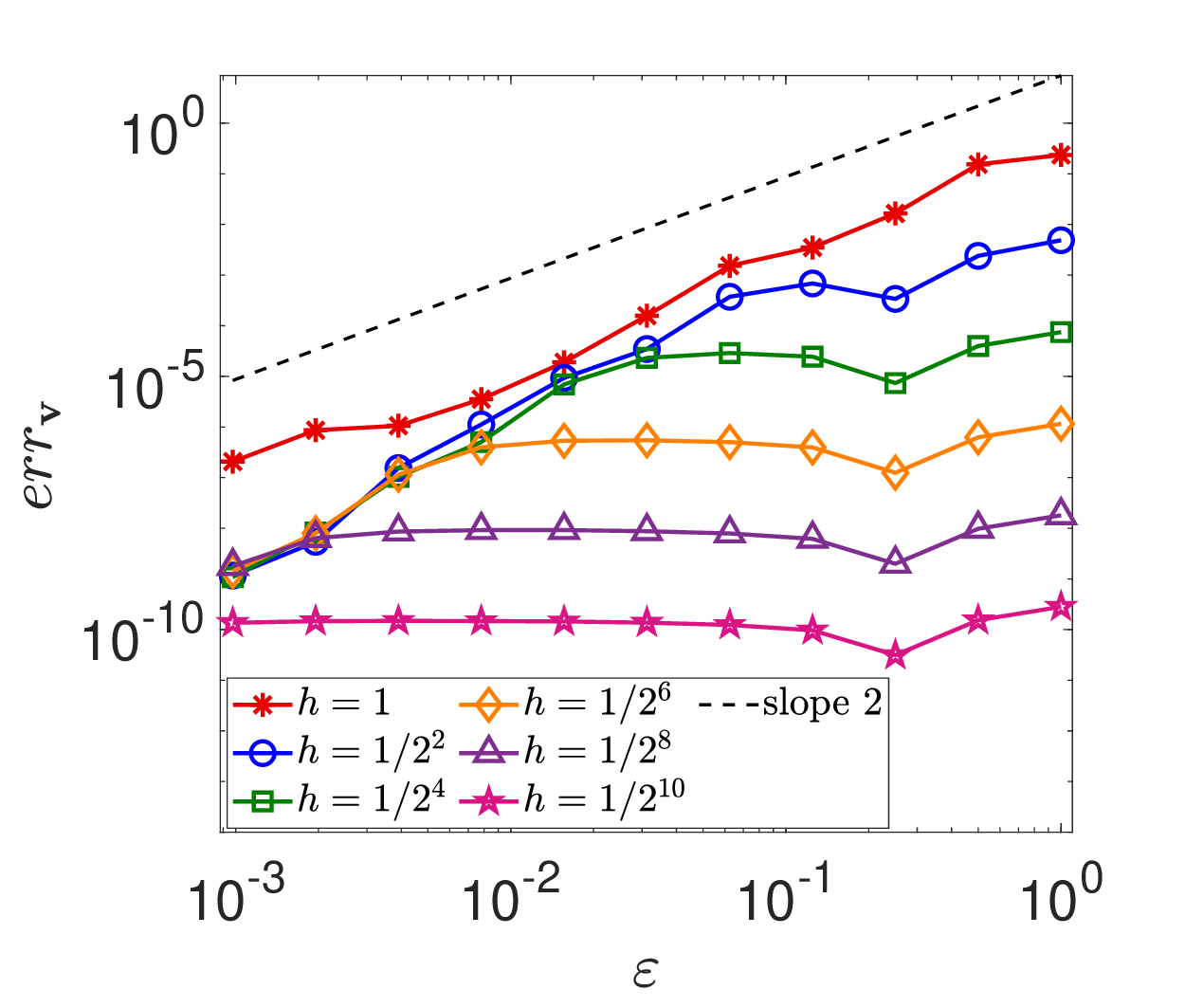}
    \par
    \includegraphics[height=3.8cm,width=5.8cm]{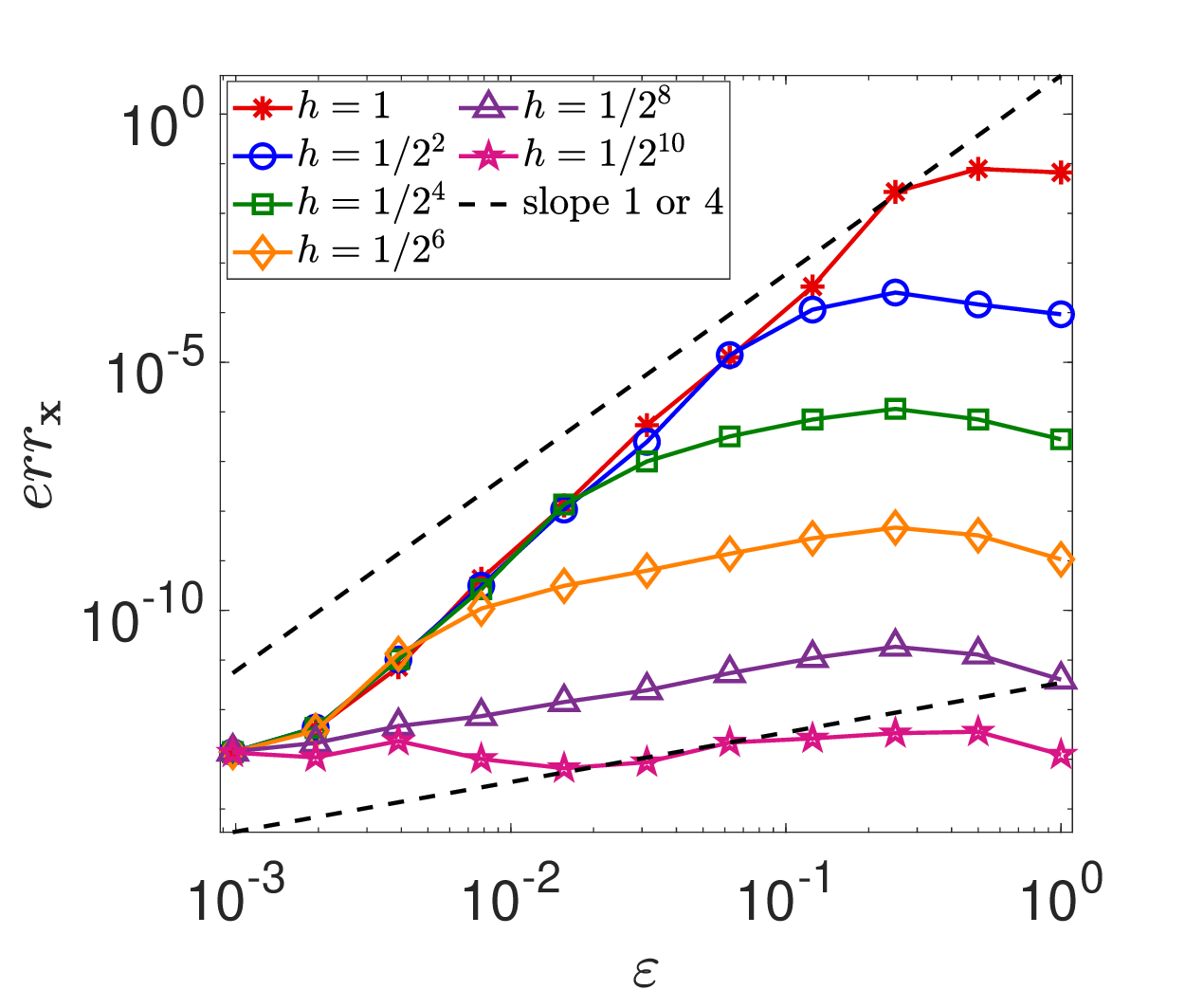}
    \hfill
    \includegraphics[height=3.8cm,width=5.8cm]{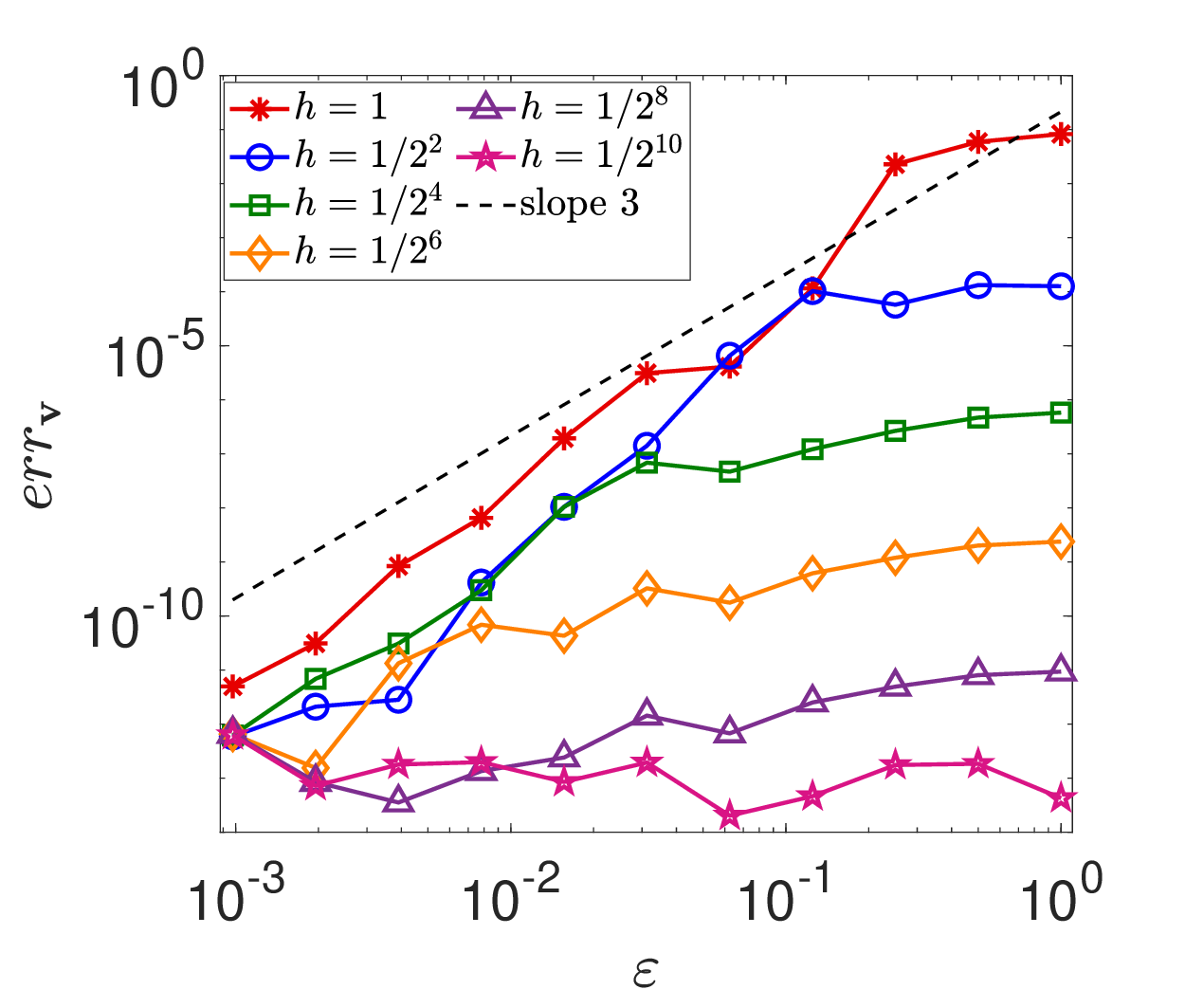}
    \caption{Example 2. Temporal errors $err_{\bx}$ and $err_{\bv}$ versus $\eps$ for the CPD with $k=2$ (first row) and $k=3$ (second row).}
    \label{fig-2-2}
\end{figure}

\begin{figure}[htbp]
    \centering
    \includegraphics[height=3.8cm,width=5.8cm]{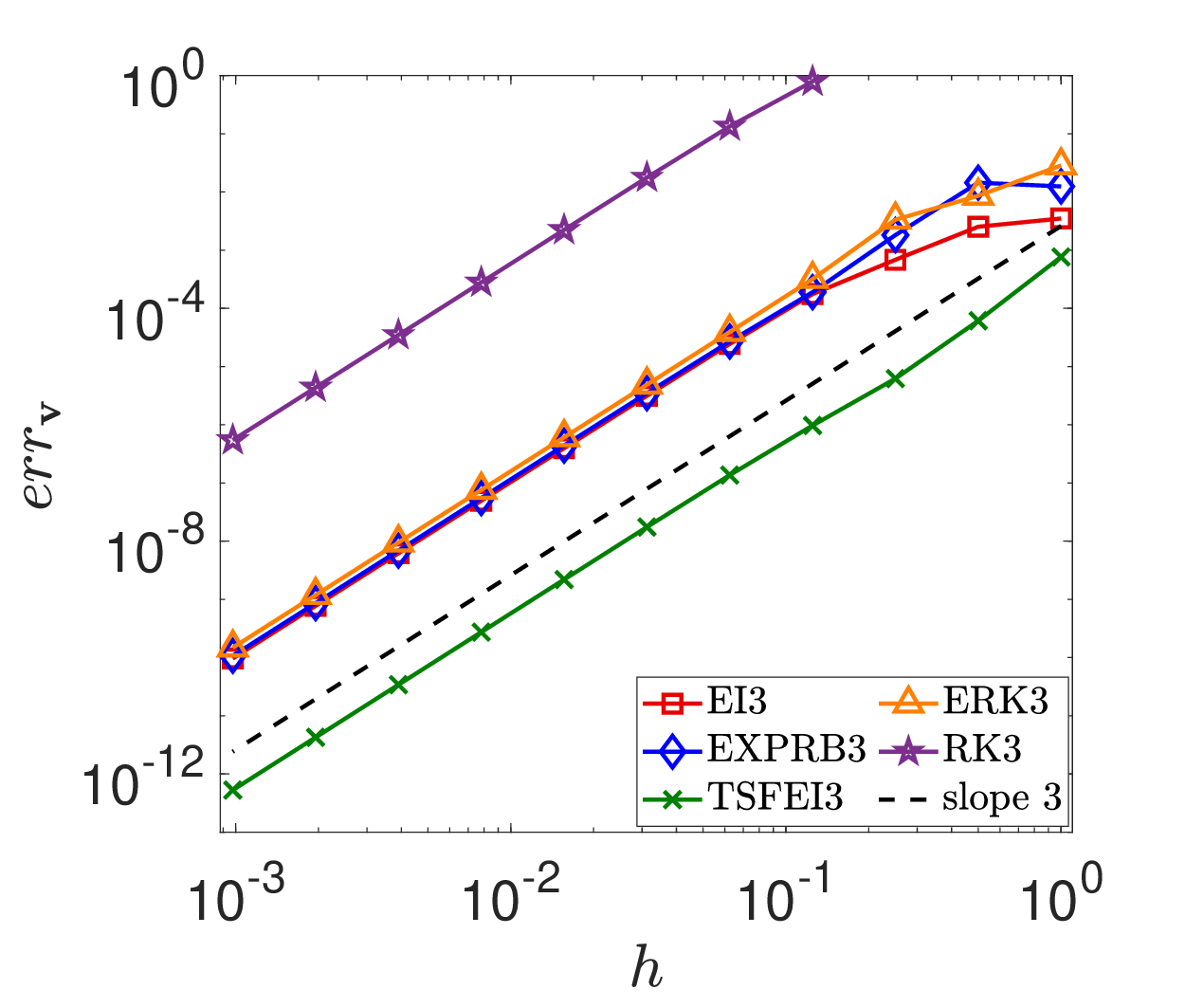}
    \hfill
    \includegraphics[height=3.8cm,width=5.8cm]{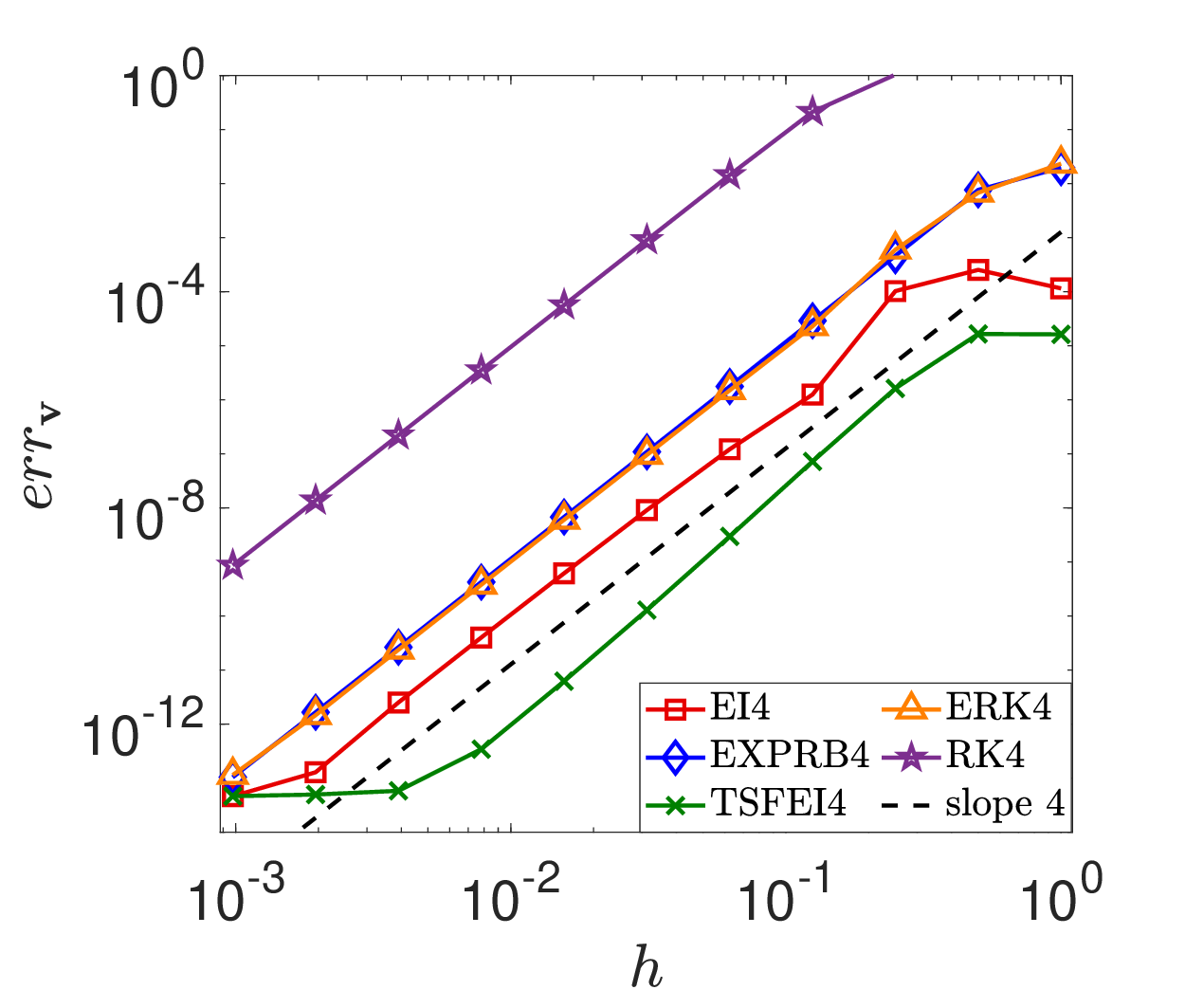}
    \par
    \includegraphics[height=3.8cm,width=5.8cm]{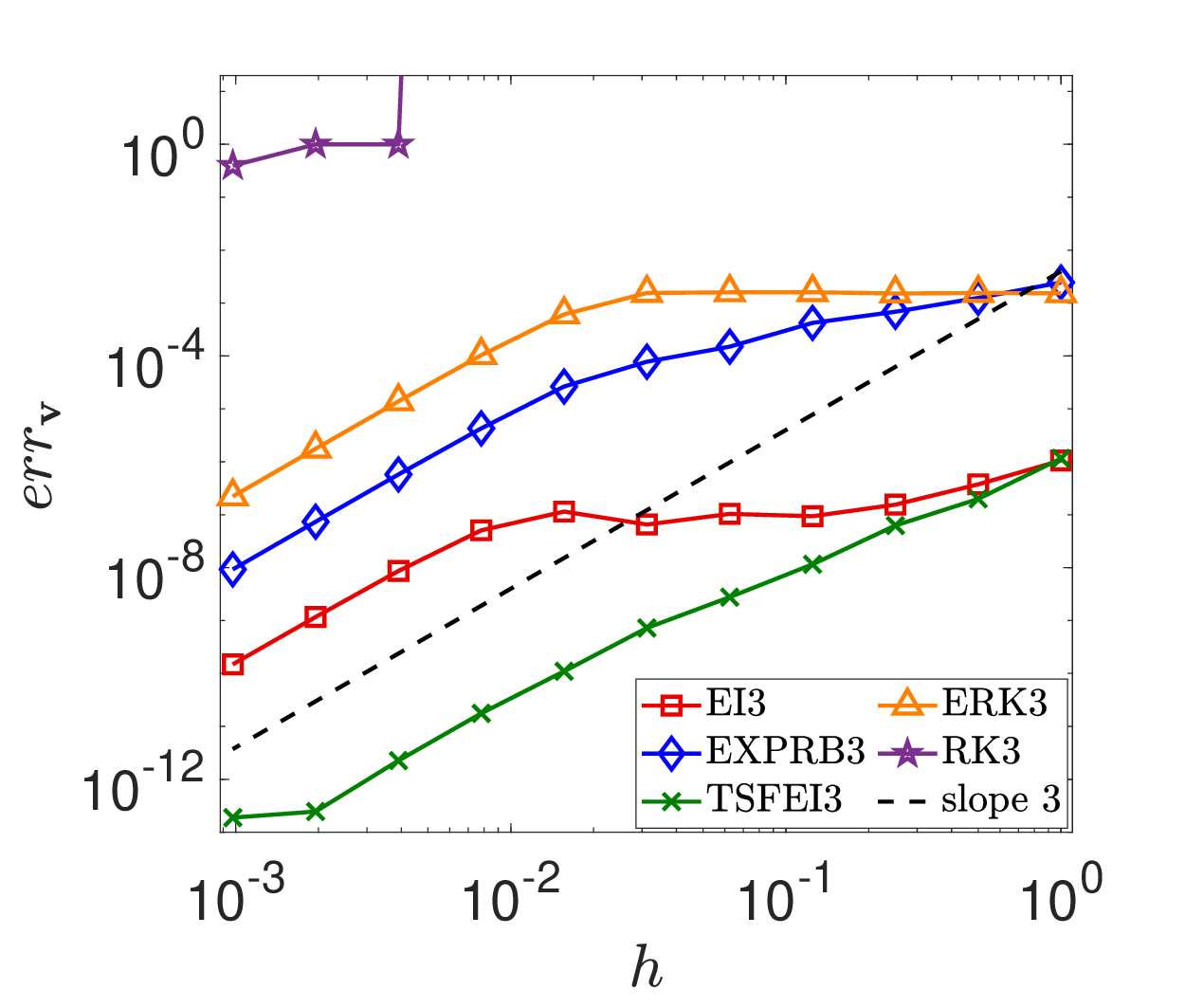}
    \hfill
    \includegraphics[height=3.8cm,width=5.8cm]{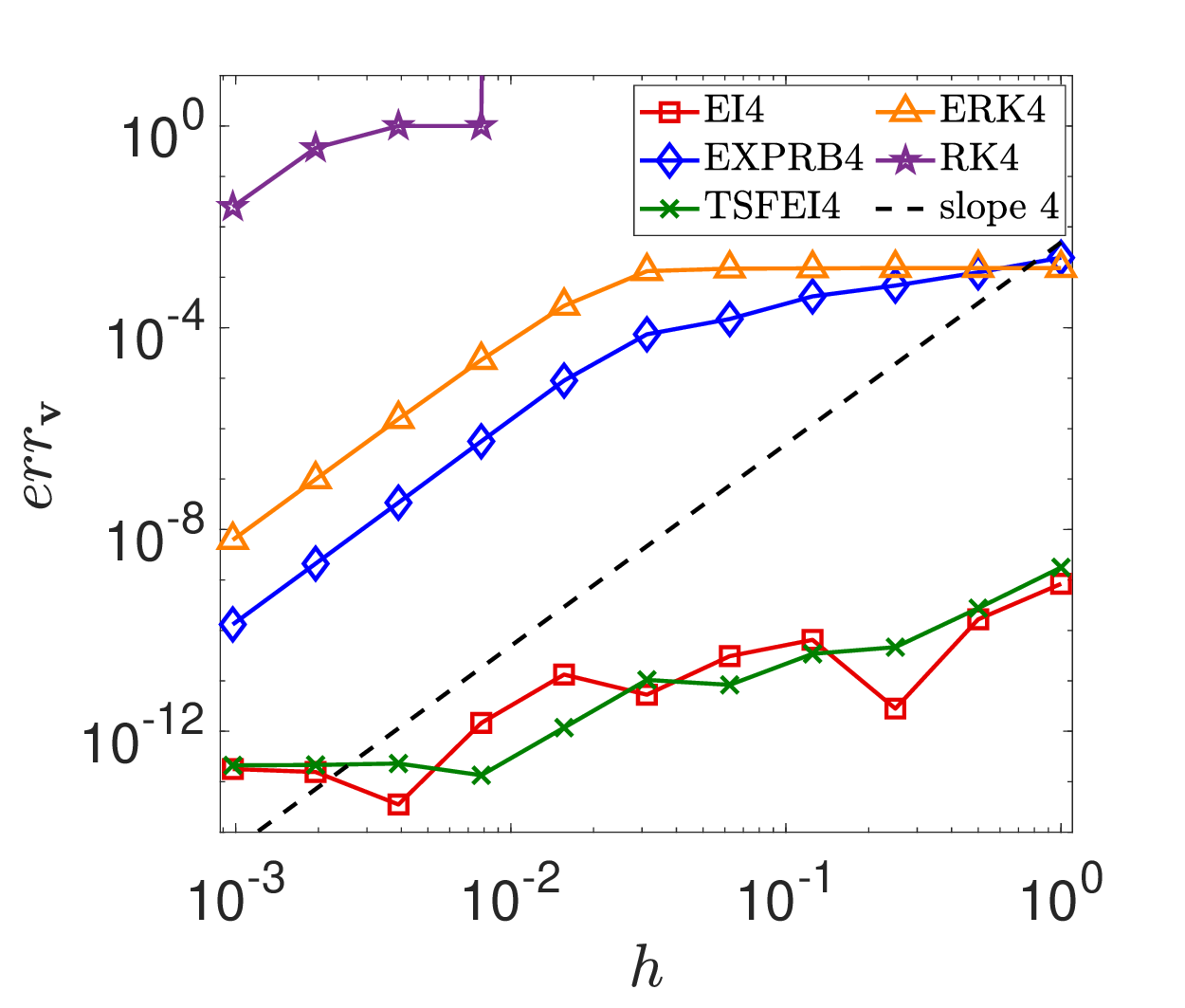}
    \caption{Example 2. Temporal errors $err_{\bv}$ of the five methods versus $h$ with $\eps=1/2^3$ (first row) and $\eps=1/2^8$ (second row).}
    \label{fig-2-3}
\end{figure}

\begin{figure}[htbp]
    \centering
    \includegraphics[height=3.8cm,width=5.8cm]{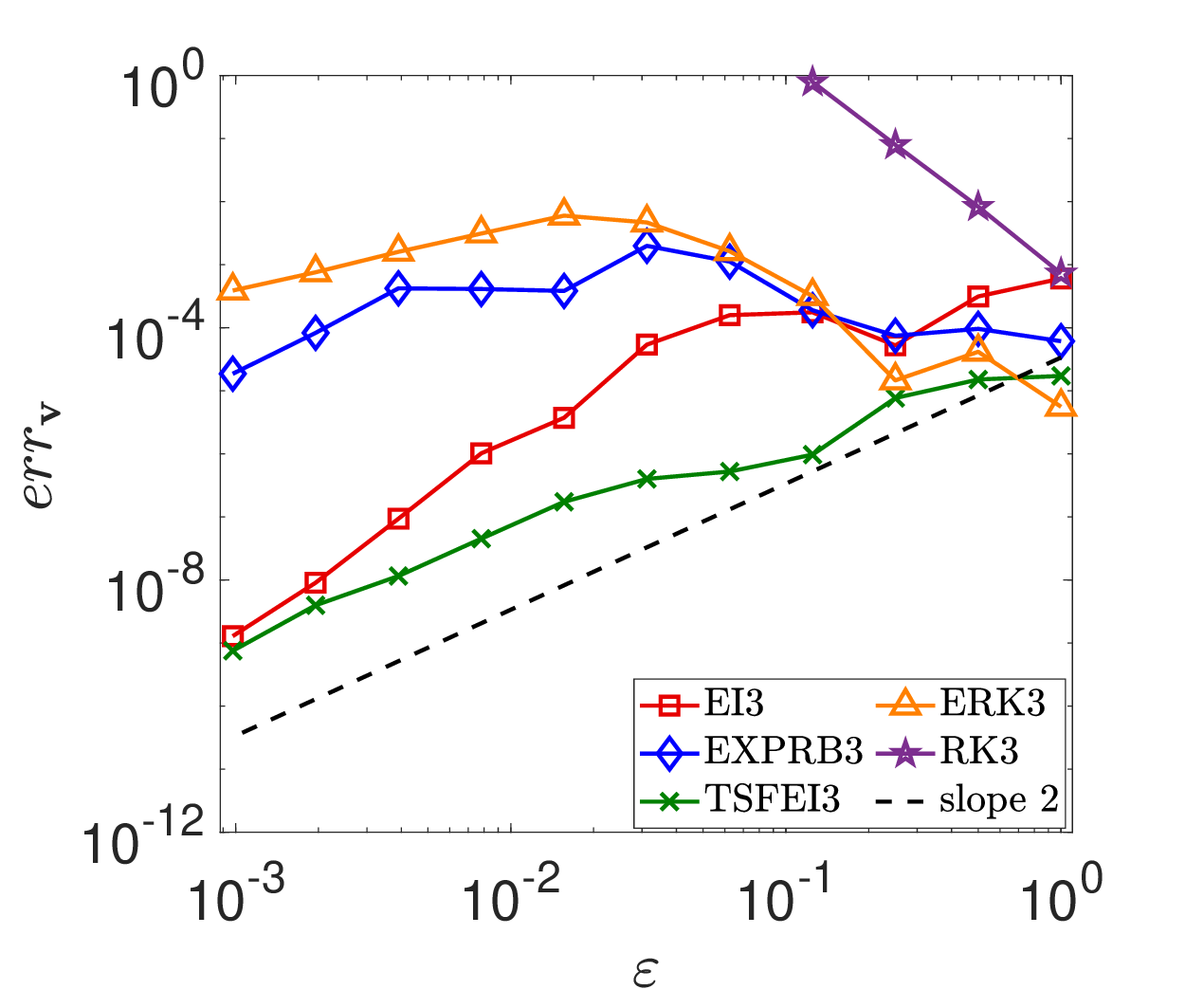}
    \hfill
    \includegraphics[height=3.8cm,width=5.8cm]{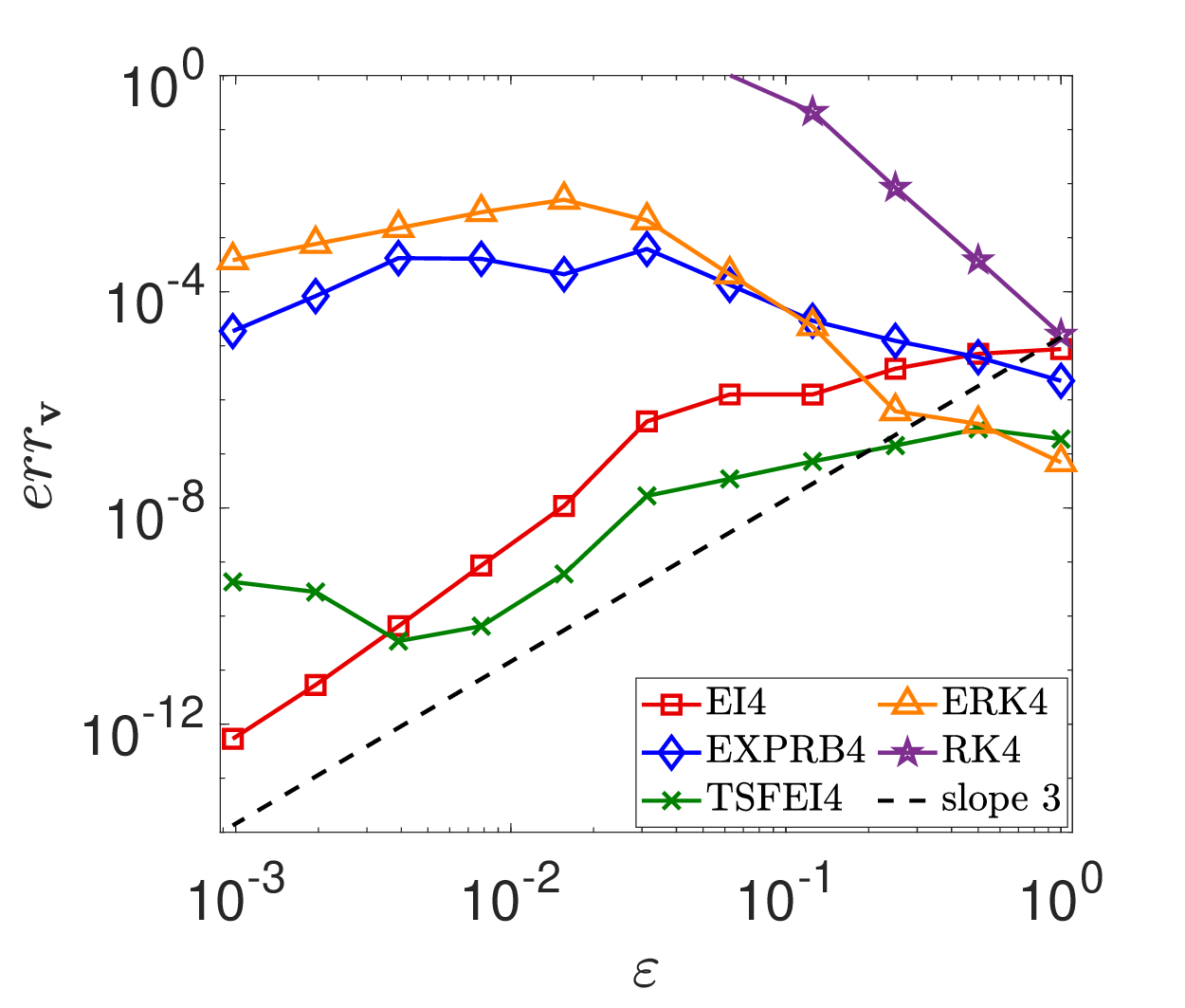}
    \par
    \includegraphics[height=3.8cm,width=5.8cm]{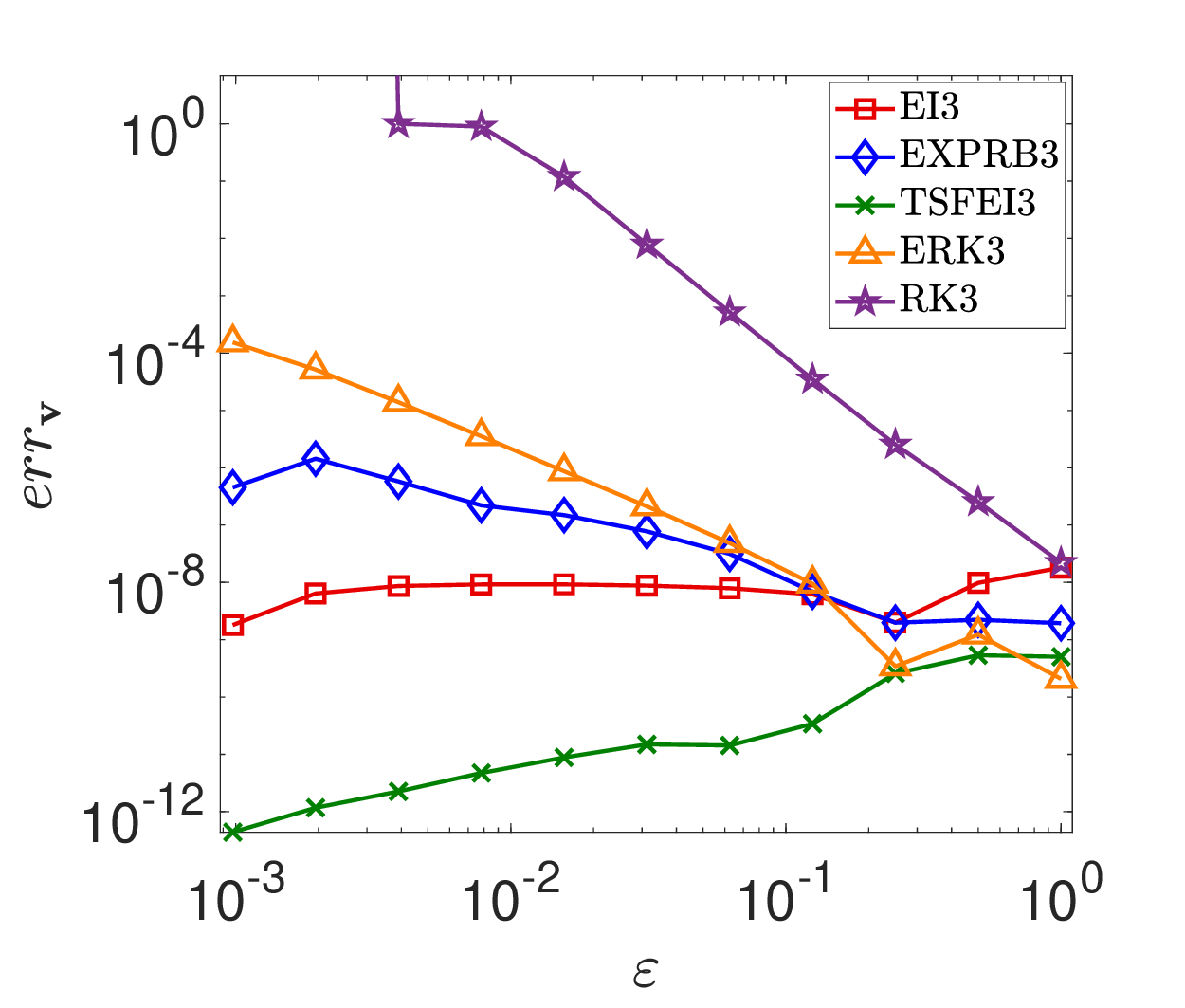}
    \hfill
    \includegraphics[height=3.8cm,width=5.8cm]{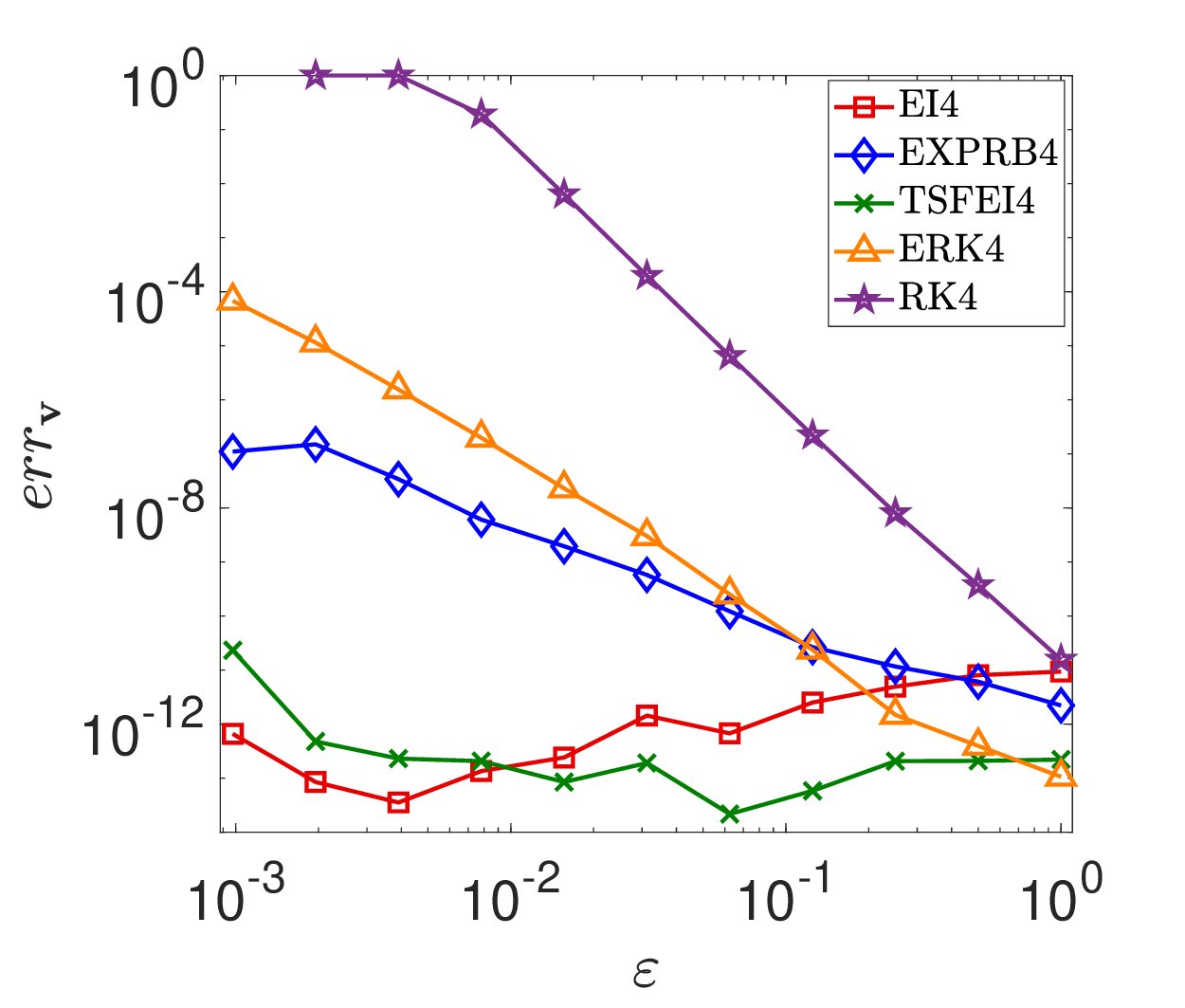}
    \caption{Example 2. Temporal errors $err_{\bv}$ of the five methods versus $\eps$ with $h=1/2^3$ (first row) and $h=1/2^8$ (second row).}
    \label{fig-2-4}
\end{figure}

\begin{figure}[htbp]
    \centering
    \includegraphics[height=3.8cm,width=5.8cm]{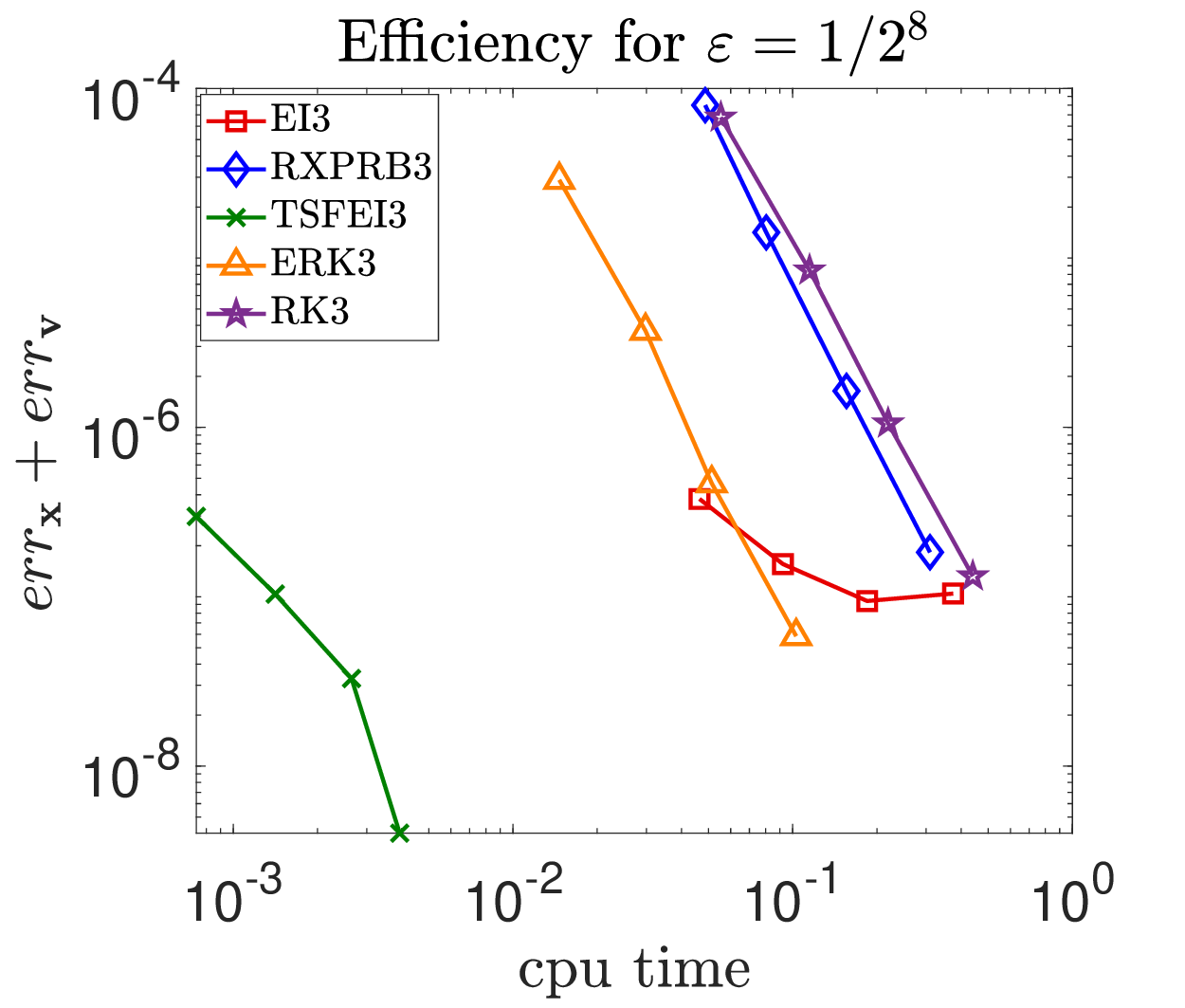}
    \hfill
    \includegraphics[height=3.8cm,width=5.8cm]{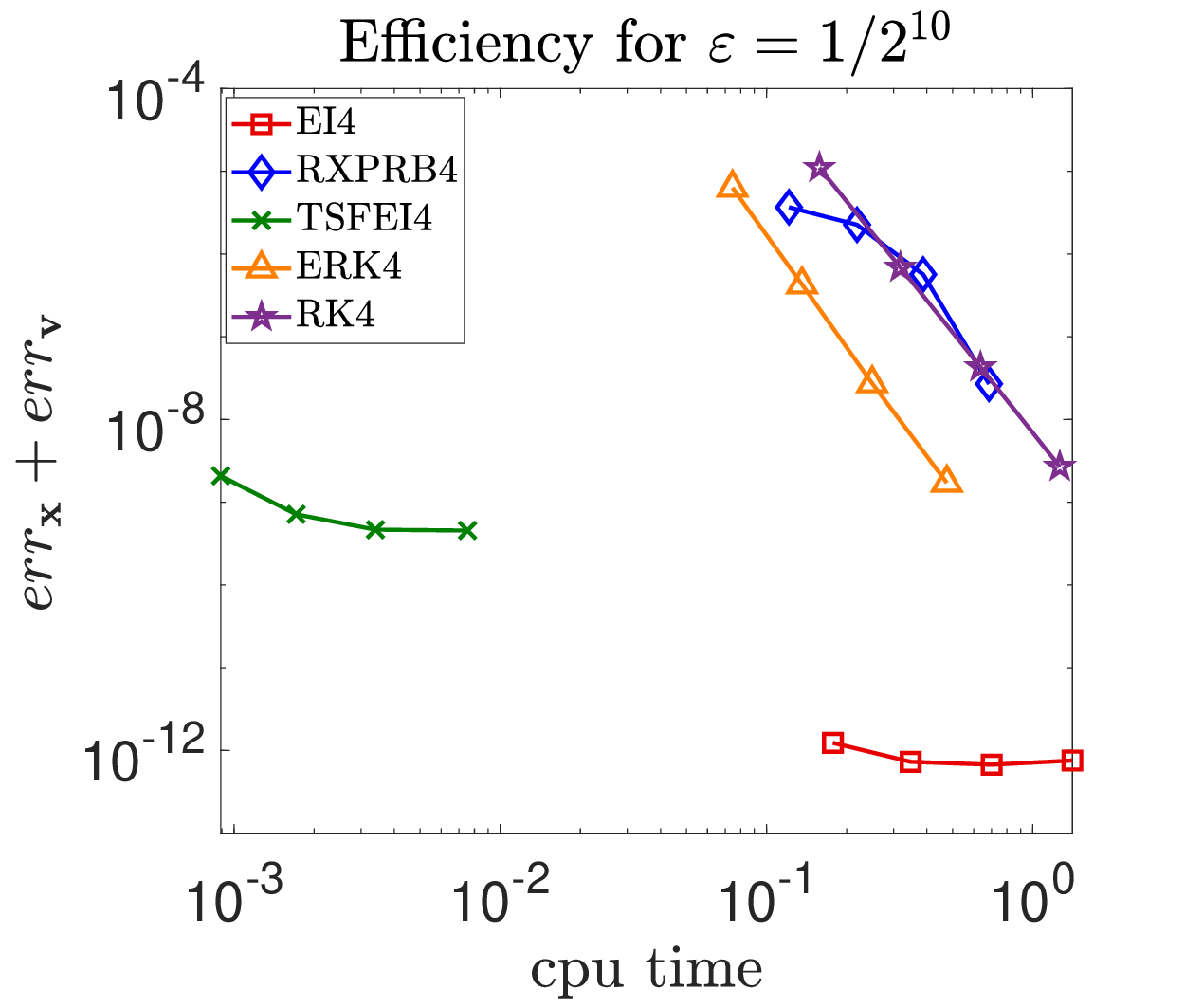}
    \caption{Example 2. Errors of the five methods versus CPU time for $\eps=1/2^8$ and $\eps=1/2^{10}$. Each method is tested using four step sizes $h_j=h_0/2^j$, $j=0,1,2,3$, with the corresponding initial step sizes $h_0$ listed in Table~\ref{tab-1}.}
    \label{fig-2-5}
\end{figure}

\begin{figure}[htbp]
    \centering
    \includegraphics[height=3.8cm,width=5.8cm]{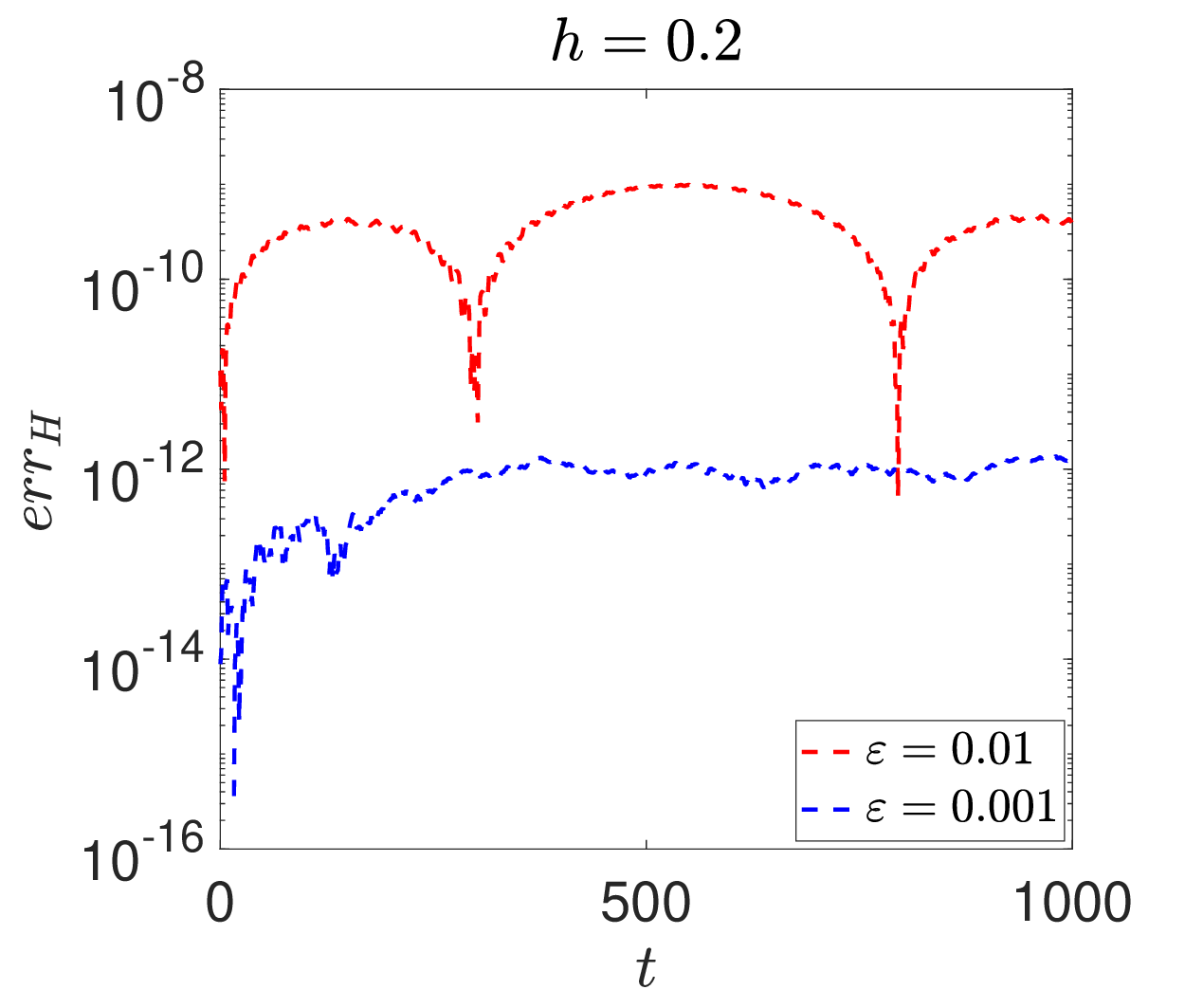}
    \hfill
    \includegraphics[height=3.8cm,width=5.8cm]{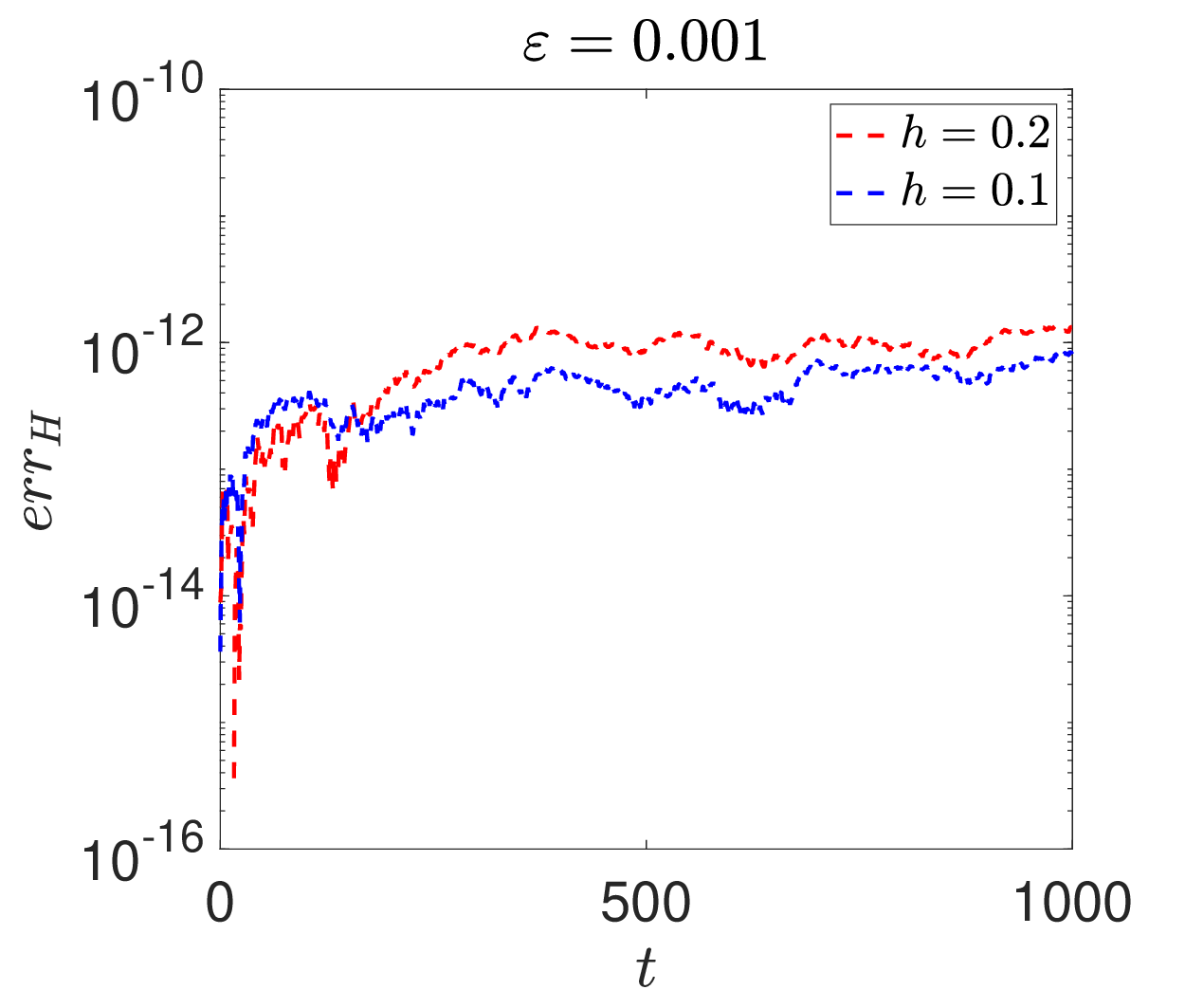}
   \caption{Example 2. Energy errors of EI4 for the CPD with various values of $\eps$ and $h$.}
   \label{fig-2-6}
\end{figure}

\vspace{0.1cm}
\textbf{Example 3. (Maximal ordering scaling with \eqref{equ-10-18-2})}
In the final example, we consider a maximal ordering scaling strong magnetic field satisfying $\nabla_{\alpha}b(\alpha)|_{\alpha=\bm{0}}=\bm{0}$. The magnetic and electric fields are prescribed as 
\[
b(\eps\bx)=1+\sin(\eps x_{1})\sin(\eps x_{2})/2, \quad \mbox{and} \
E(\bx)=\begin{pmatrix}
         2 x_{1}\fe^{-x_{1}^2}\cos(x_{2}) \\ \fe^{-x_{1}^2}\sin(x_{2})
       \end{pmatrix}
\]
with the corresponding potential $U(\bx)=\fe^{-x_{1}^2}\cos(x_{2})$. The initial data coincide with those of Example 2, and the final time is 
$T=3$. Figure~\ref{fig-3-1} shows the errors of the EI($k+1$) schemes with $k=1,2,3$ versus the step size $h$ for several values of $\eps$. For relatively large $\eps$, the methods attain the expected convergence order $\mathcal{O}(h^{k+1})$. At intermediate $\eps$, the errors exhibit a piecewise dependence on $h$, while for smaller $\eps$, uniform accuracy is observed. The transition between these regimes shifts toward finer step sizes as $\eps$ decreases, consistent with the scale $h=\mathcal{O}(\eps)$ predicted by \eqref{equ-10-13-2}.

Figure~\ref{fig-3-2} further illustrates the dependence of the errors on $\eps$ at fixed step sizes. For relatively large $h$, the position and velocity errors decay as $\mathcal{O}(\eps^{k+2})$ and $\mathcal{O}(\eps^{k+1})$, respectively. Intermediate step sizes yield piecewise scaling with $\eps$. For sufficiently small $h$, the position error behaves as $\mathcal{O}(\eps)$, whereas the velocity error remains uniformly bounded in $\eps$. These numerical results confirm the convergence orders established in Theorem~\ref{theo-1}. Figure~\ref{fig-3-3} reveals the long-term behavior of the energy error for EI4 over $[0,1000]$, indicating excellent near-conservation.

\begin{figure}[htbp]
    \centering
    \includegraphics[height=3.8cm,width=5.8cm]{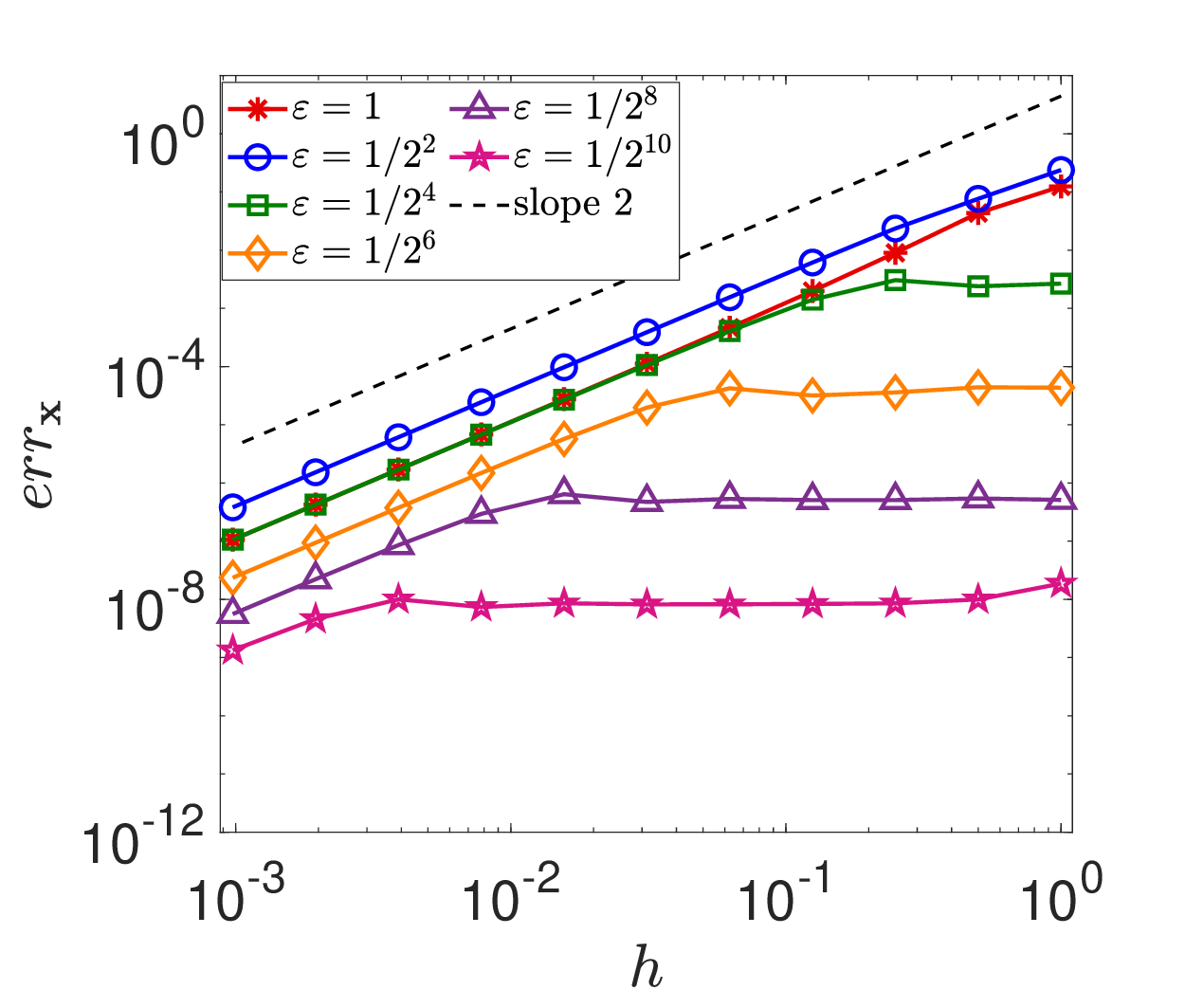}
    \hfill
    \includegraphics[height=3.8cm,width=5.8cm]{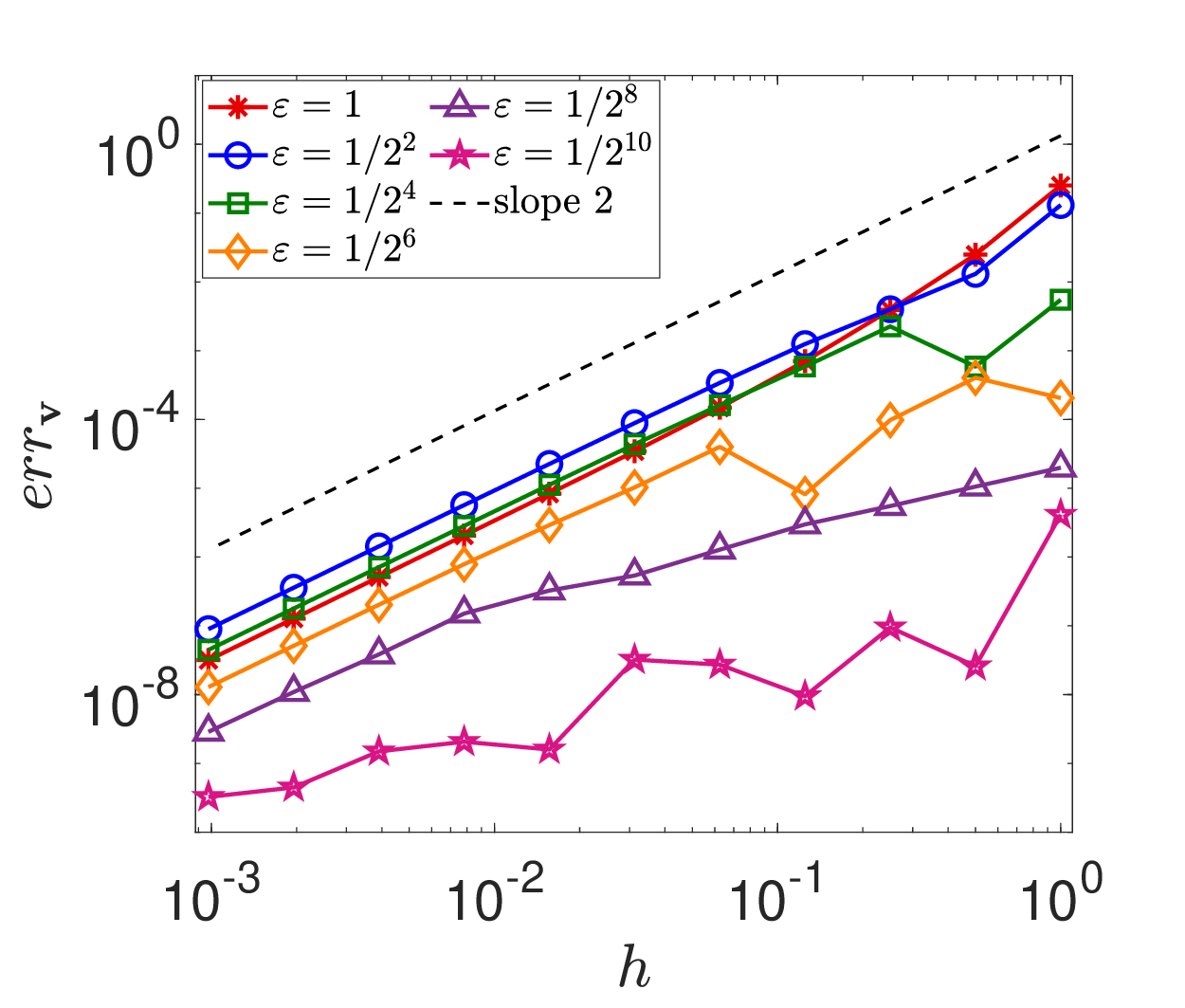}
    \par
    \includegraphics[height=3.8cm,width=5.8cm]{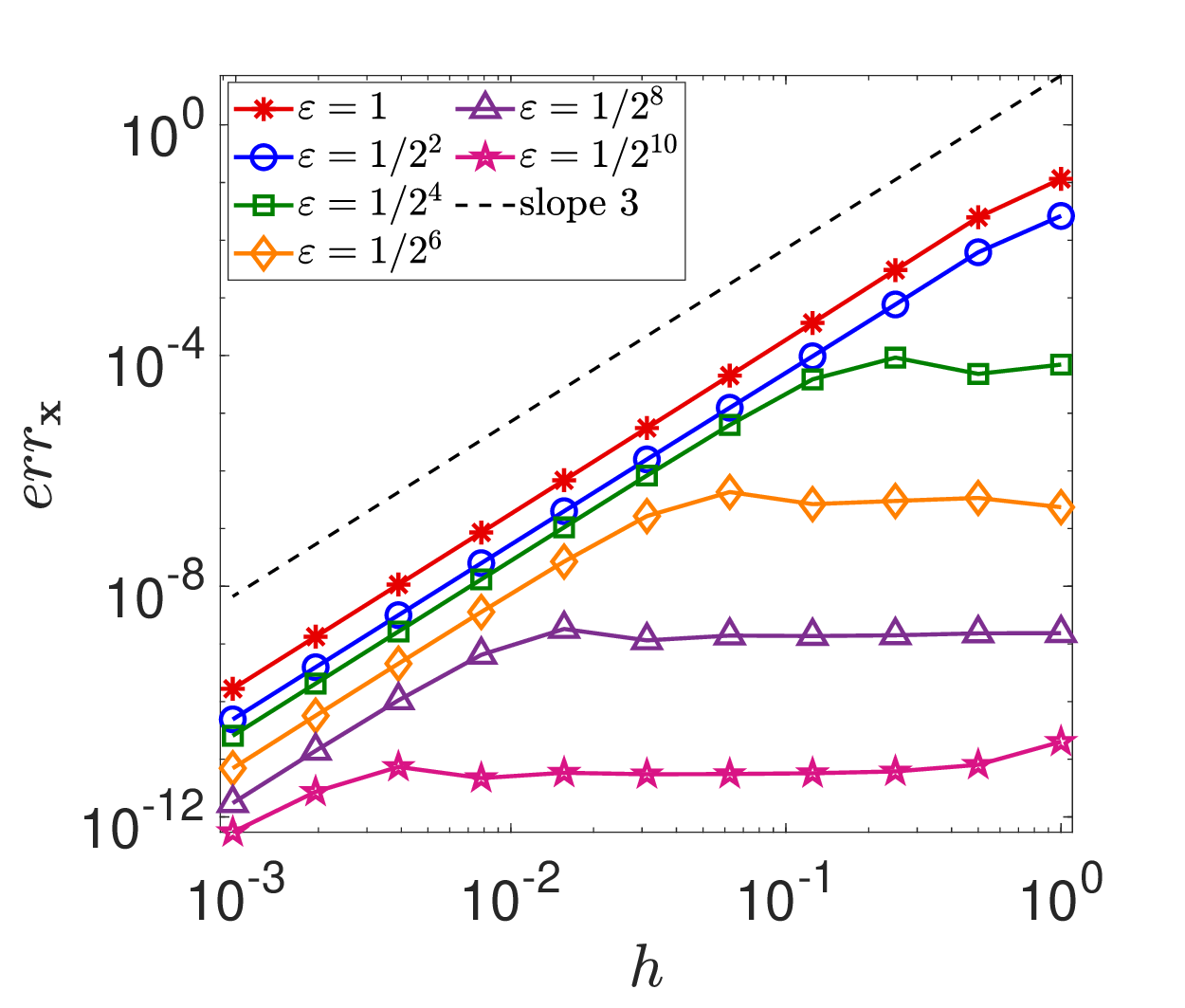}
    \hfill
    \includegraphics[height=3.8cm,width=5.8cm]{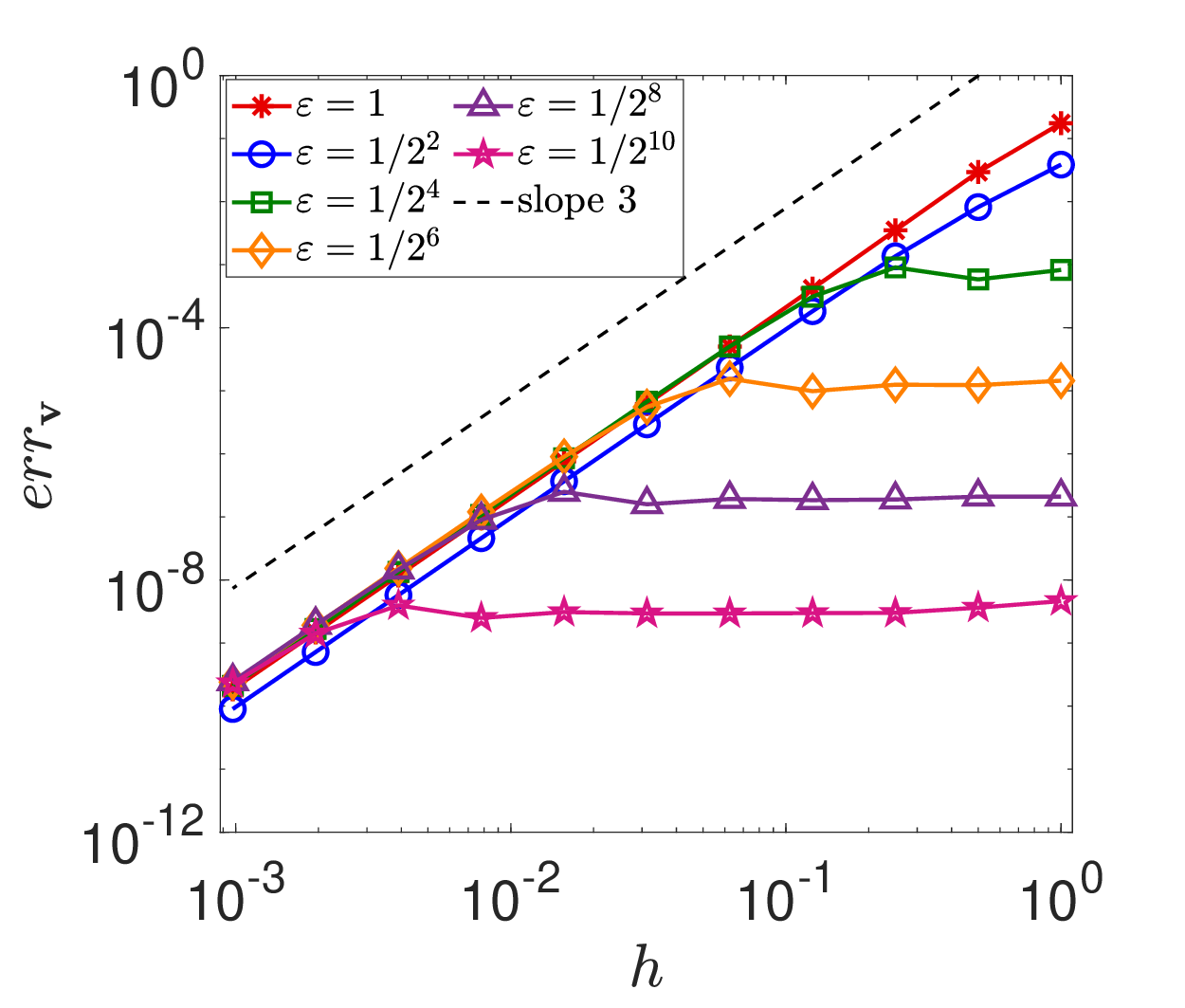}
    \par
    \includegraphics[height=3.8cm,width=5.8cm]{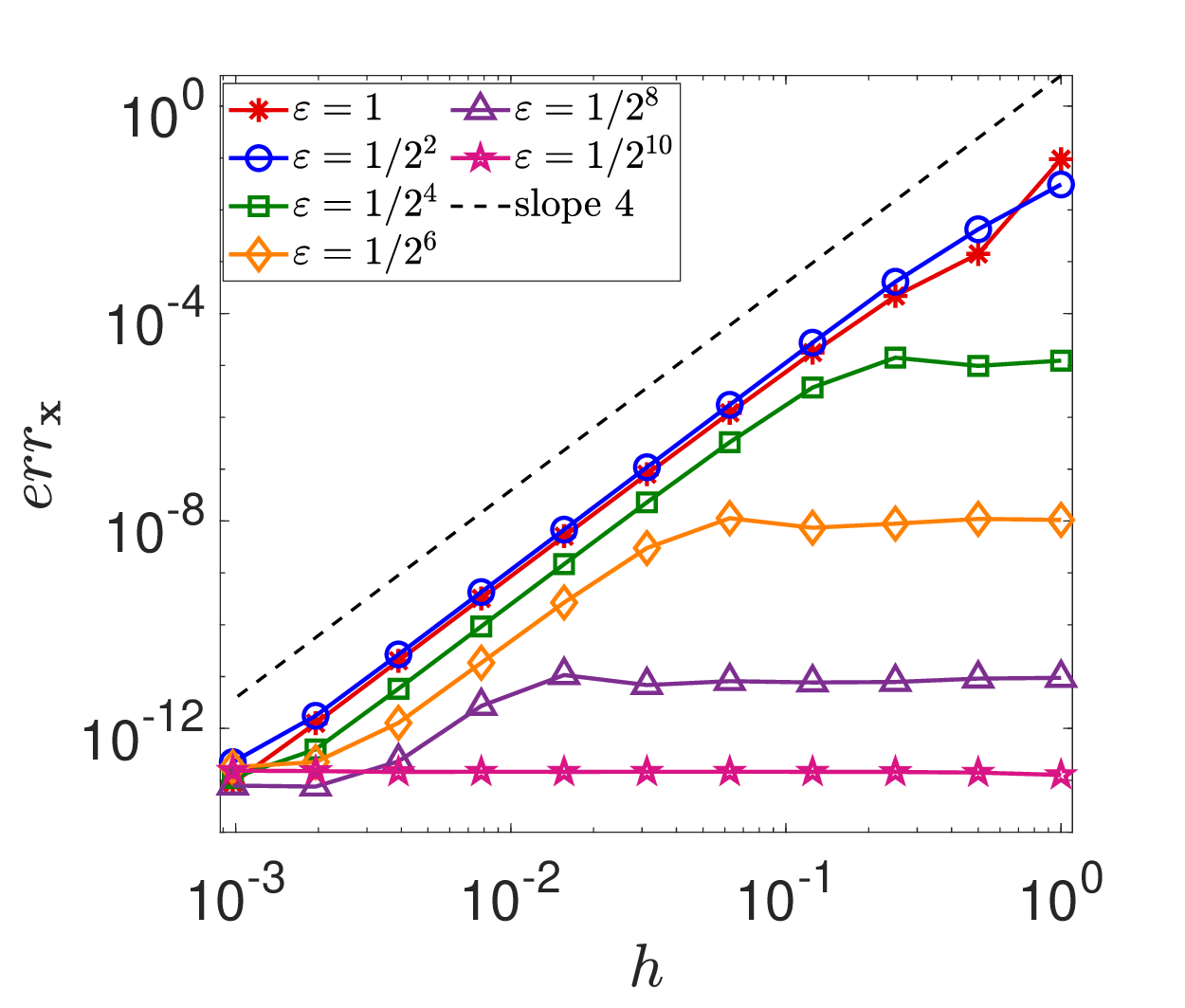}
    \hfill
    \includegraphics[height=3.8cm,width=5.8cm]{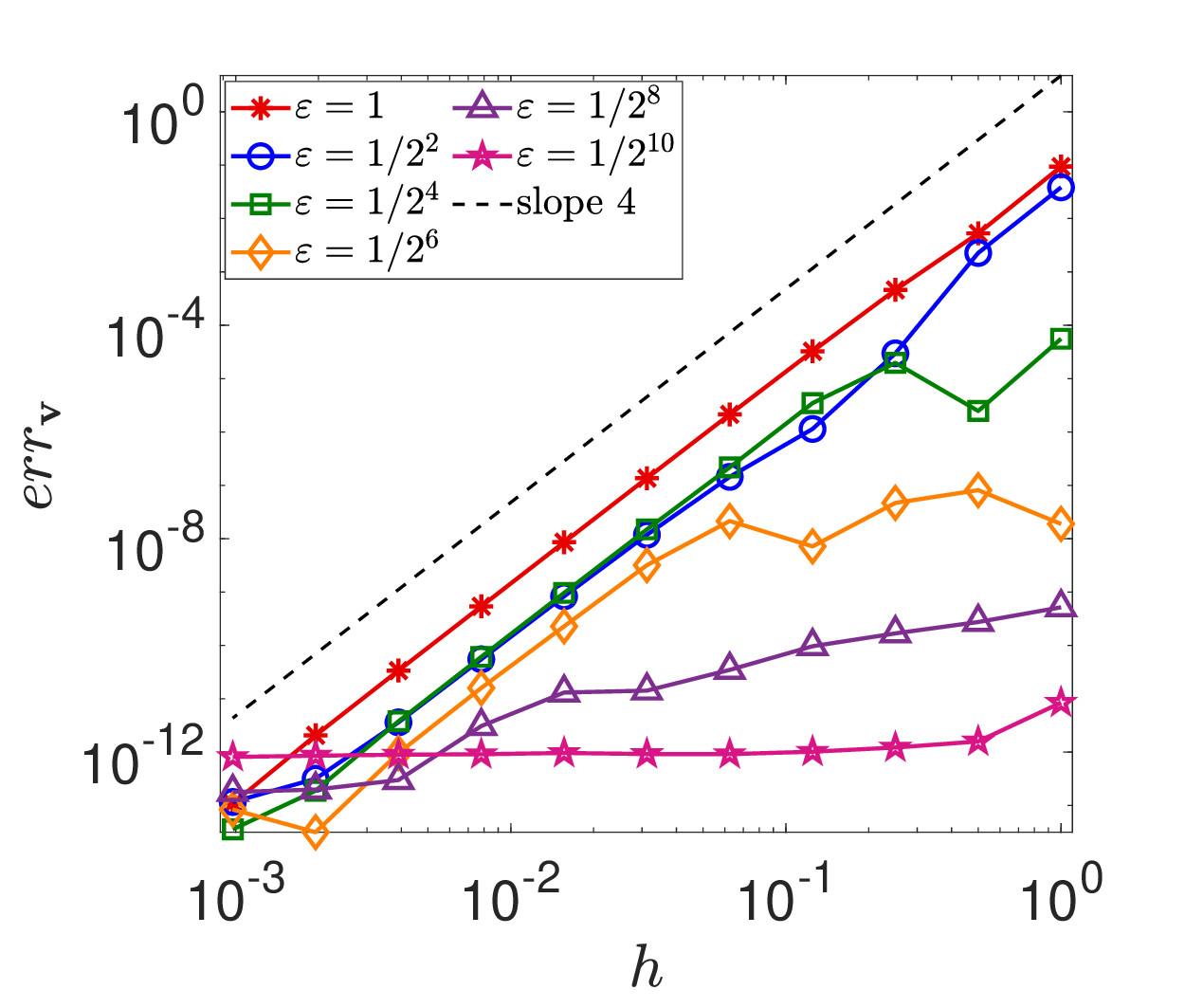}
    \caption{Example 3. Temporal errors $err_{\bx}$ and $err_{\bv}$ versus $h$ for the CPD with $k=1,2,3$ (from top to bottom).}
    \label{fig-3-1}
\end{figure}

\begin{figure}[htbp]
    \centering
    \includegraphics[height=3.8cm,width=5.8cm]{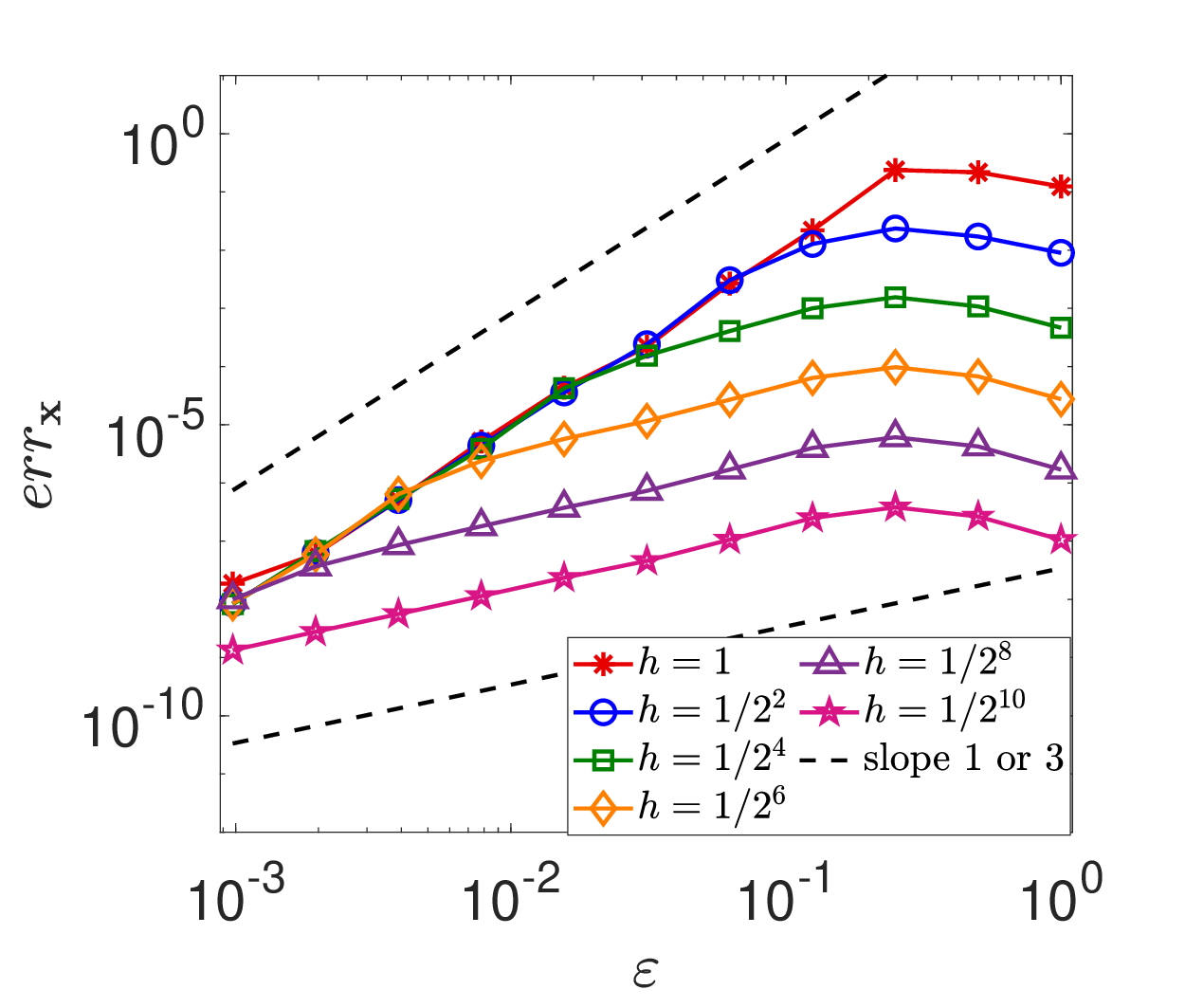}
    \hfill
    \includegraphics[height=3.8cm,width=5.8cm]{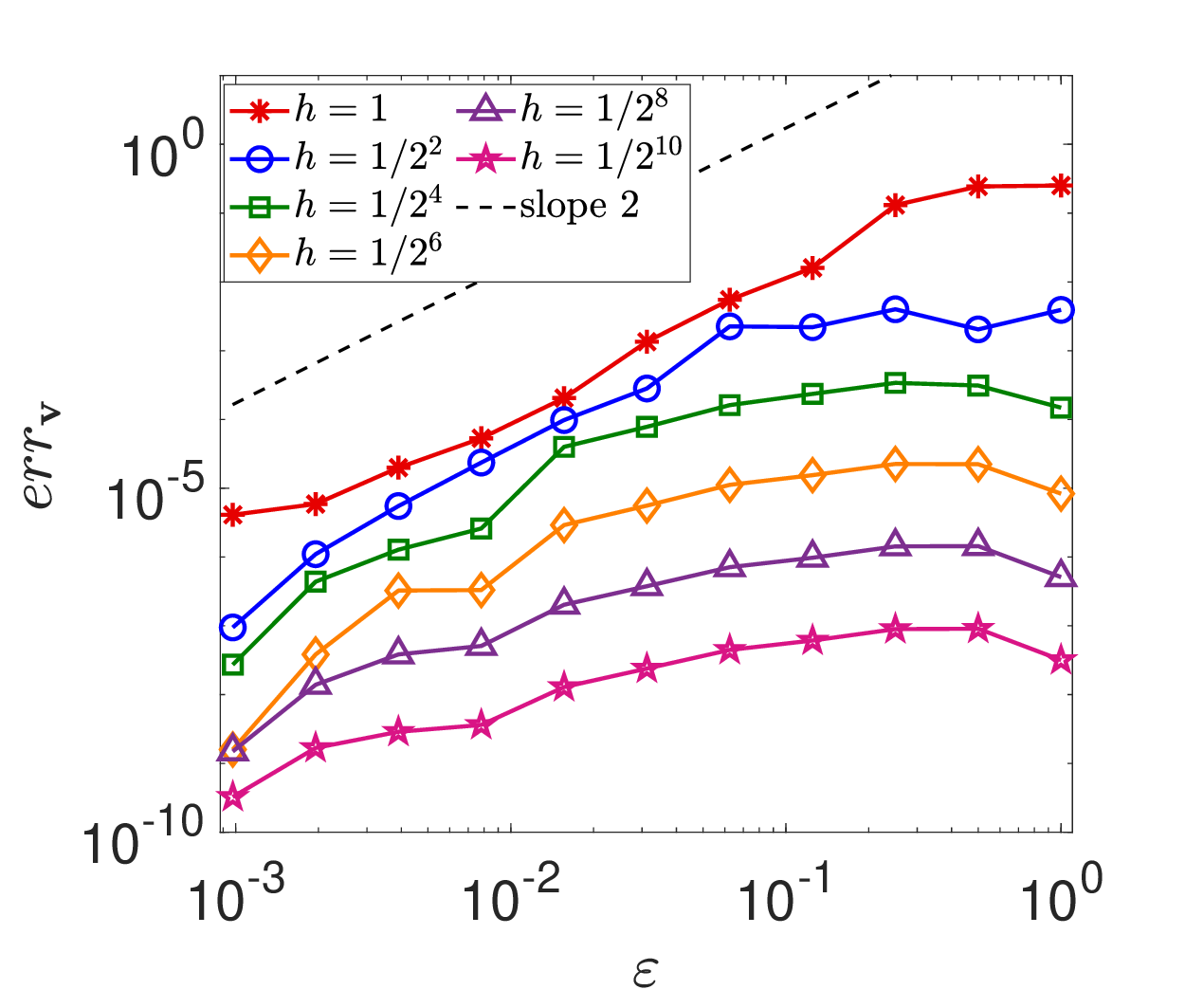}
    \par
    \includegraphics[height=3.8cm,width=5.8cm]{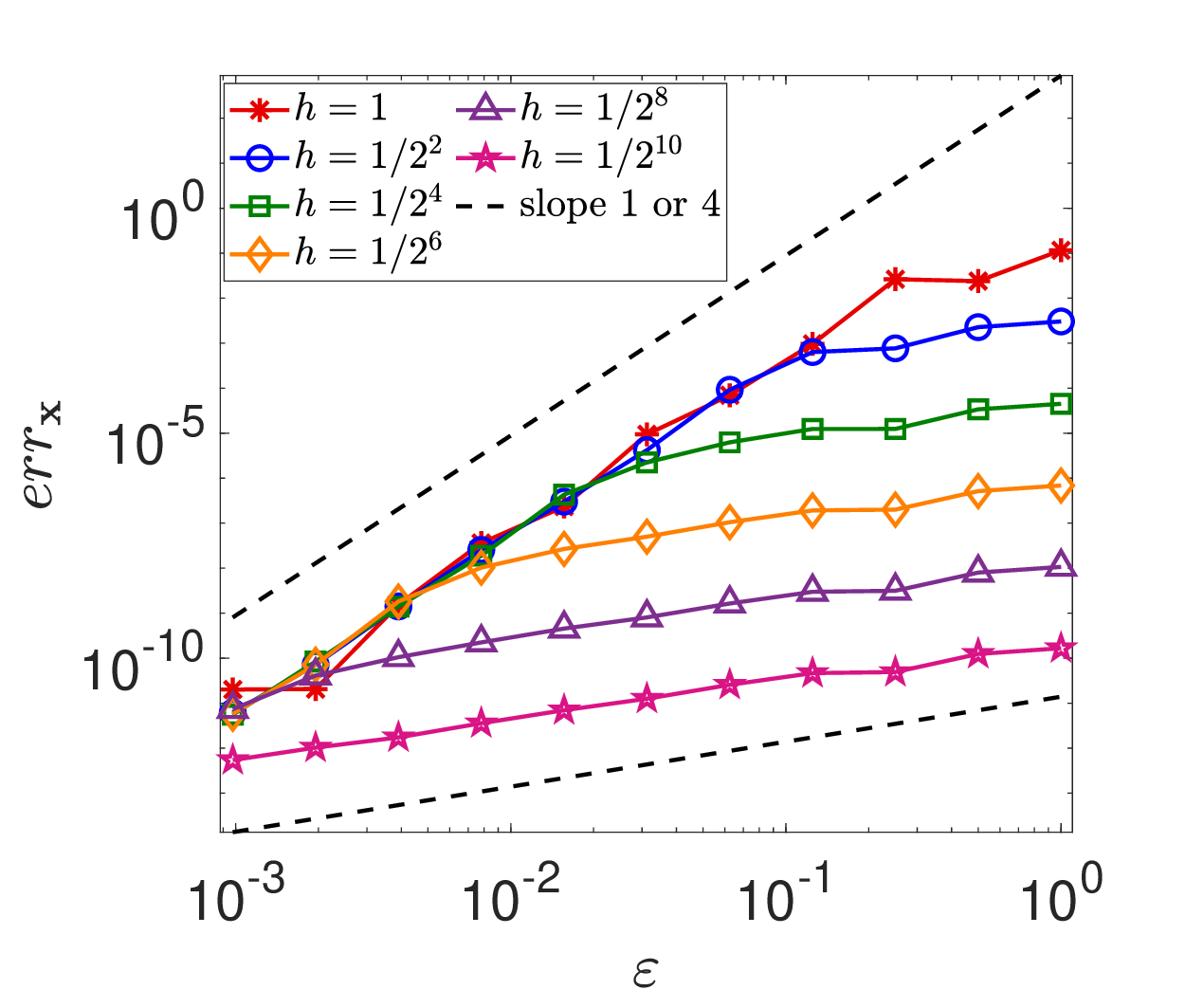}
    \hfill
    \includegraphics[height=3.8cm,width=5.8cm]{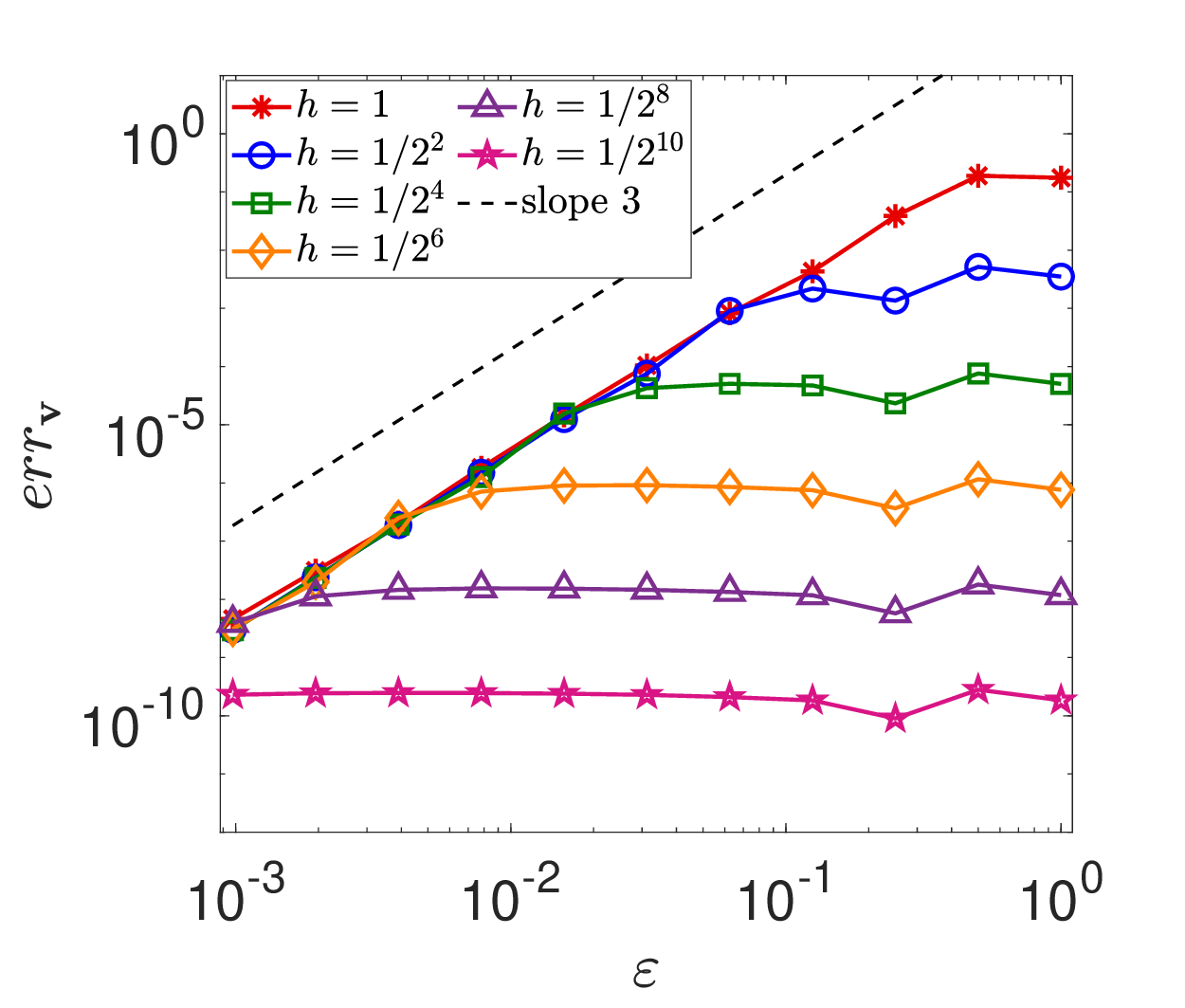}
    \par
    \includegraphics[height=3.8cm,width=5.8cm]{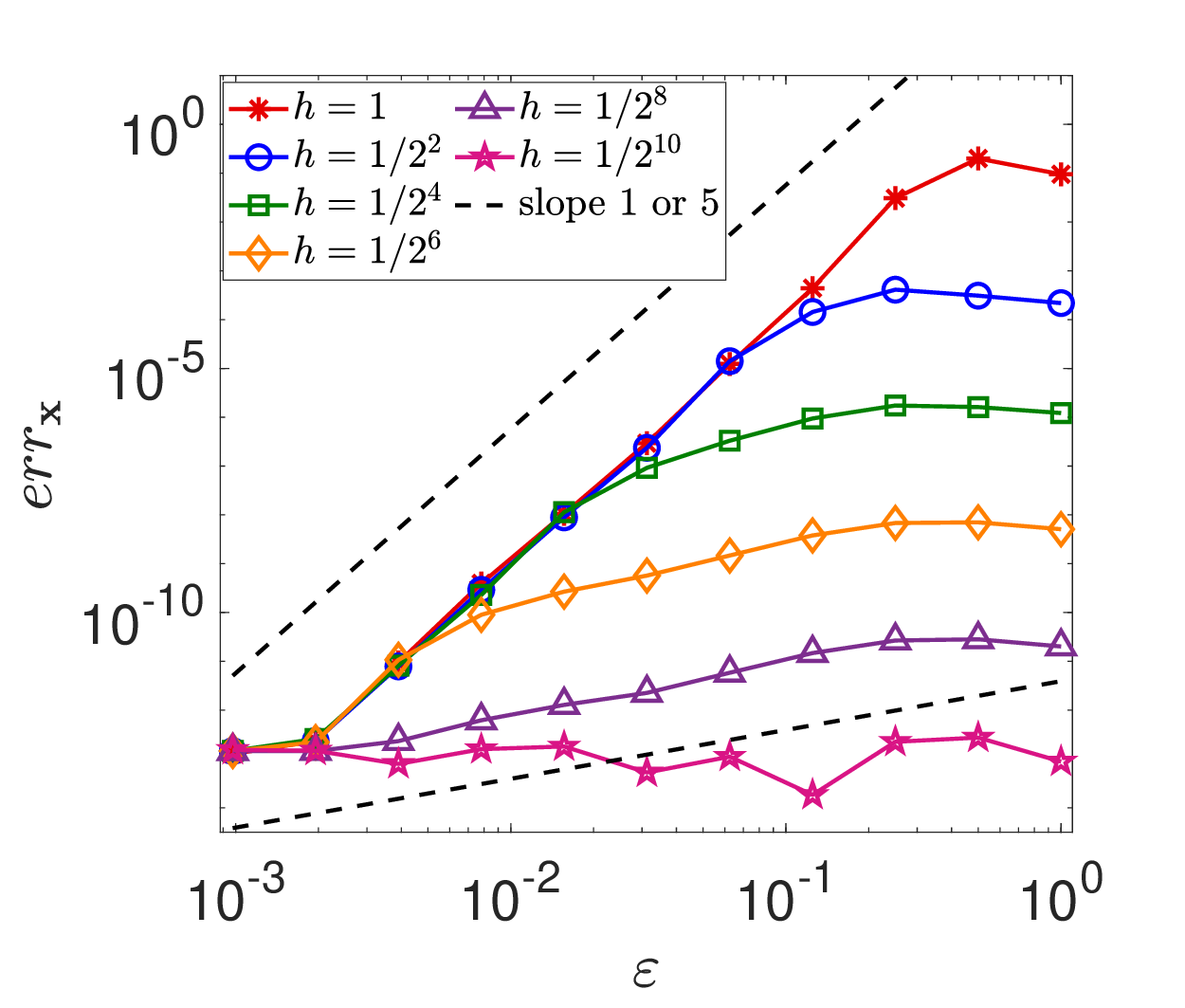}
    \hfill
    \includegraphics[height=3.8cm,width=5.8cm]{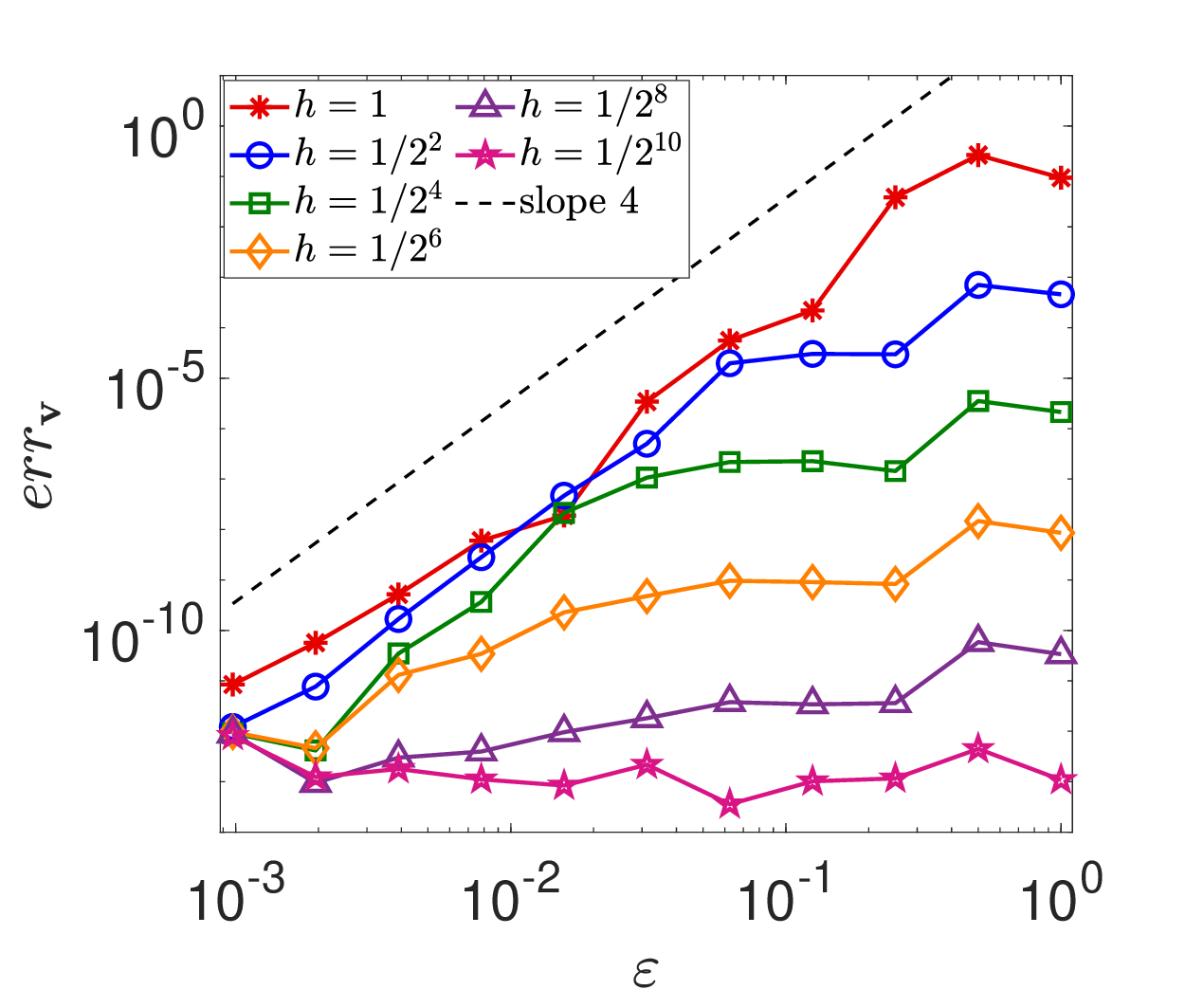}
    \caption{Example 3. Temporal errors $err_{\bx}$ and $err_{\bv}$ versus $\eps$ for the CPD with $k=1,2,3$ (from top to bottom).}
    \label{fig-3-2}
\end{figure}

\begin{figure}[htbp]
    \centering
    \includegraphics[height=3.8cm,width=5.8cm]{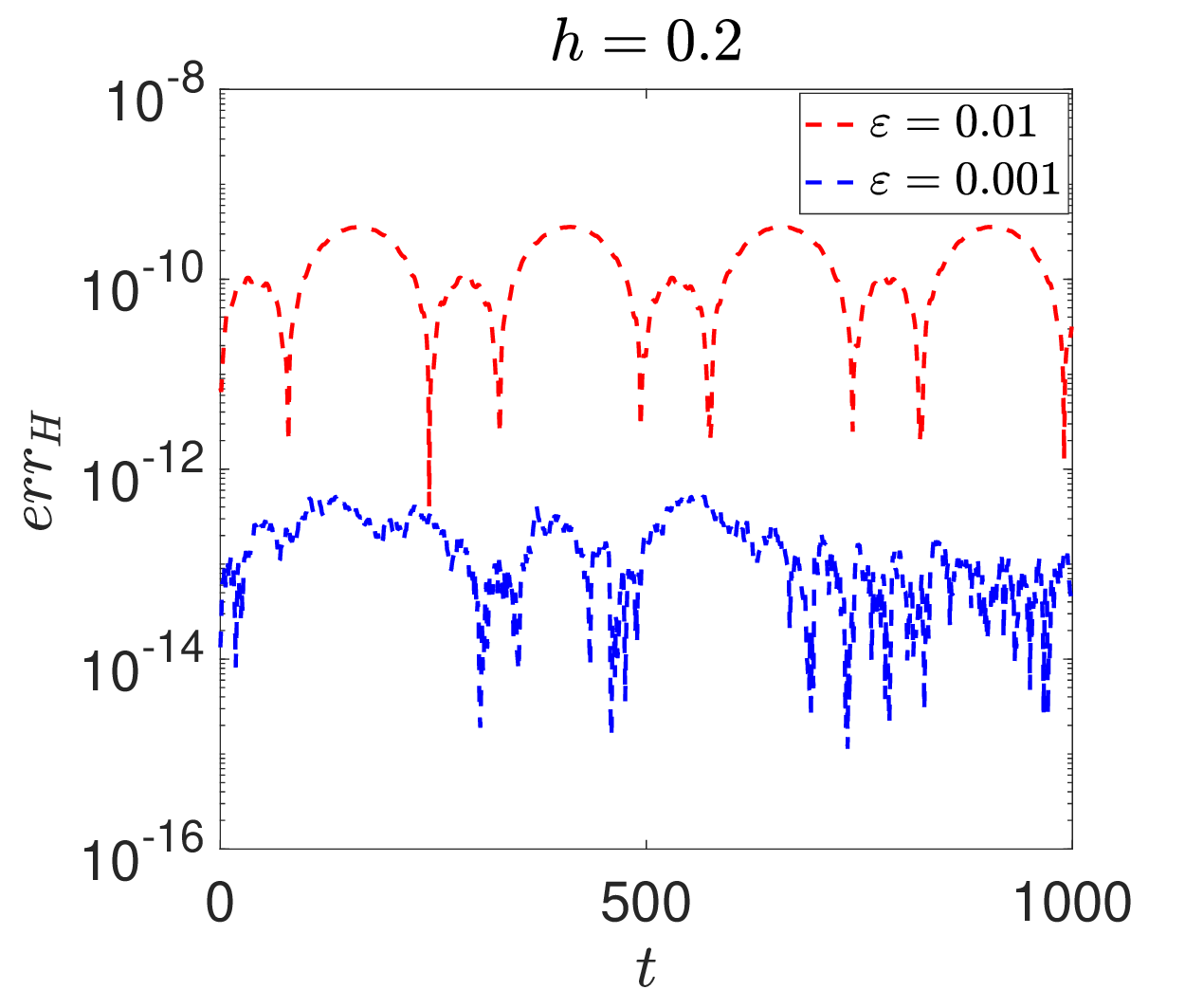}
    \hfill
    \includegraphics[height=3.8cm,width=5.8cm]{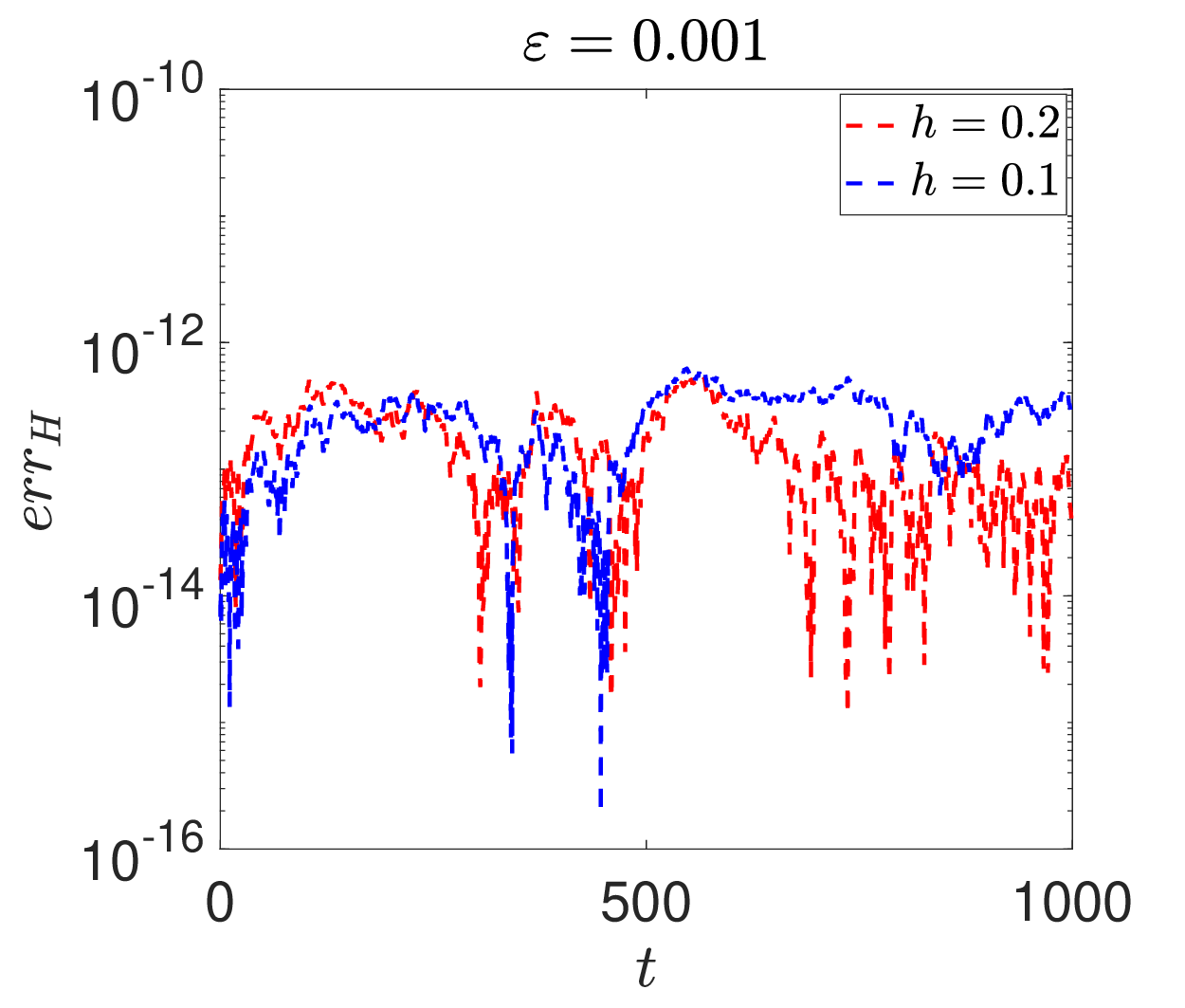}
   \caption{Example 3. Energy errors of EI4 for the CPD with various values of $\eps$ and $h$.}
   \label{fig-3-3}
\end{figure}

\section{Conclusion}
In this work, we extend the local linear extension exponential integrator framework of \cite{QDF} to the two-dimensional CPD system under strong magnetic fields. Our main
contribution is to adapt this framework to CPD and relax the assumptions required in \cite{QDF}, rather than to modify the algorithmic core. In particular, Lemma~\ref{lem-3} establishes the CPD-specific uniform bounds required for the convergence analysis under uniformly bounded initial data, thereby eliminating the additional small-initial-data condition imposed in \cite{QDF}. We further analyze constant strong magnetic field and the maximal ordering scaling regime, and identify a condition leading to sharper error estimates in the latter case. The resulting schemes retain the arbitrarily high order, order-condition-free construction and full explicitness of the original method. Numerical experiments confirm the theoretical convergence results.

%
\section*{Conflict of interest}
 The authors declare that they have no conflict of interest.



\end{document}